\documentclass[a4paper,english,oneside,12pt]{amsart}
\usepackage[latin1]{inputenc}
\usepackage{amsmath}
\usepackage{amssymb}
\usepackage{amsfonts}
\usepackage{latexsym}
\usepackage{mathtools}
\usepackage{amscd}
\usepackage{xy} \xyoption{all}
\usepackage{hyperref}
\usepackage{enumerate}
\usepackage{color}

\usepackage{ulem} 
\newcommand{\red}[1]{\textcolor{red}{#1}}
\newcommand{\blue}[1]{\textcolor{blue}{#1}}

\newcommand{\gen}[1]{\left\langle #1\right\rangle}

\newcommand{\TOKEok}[1]{}

\newcommand{\ut}{\widetilde{\mathcal U}}
 
\newcommand{\cok}{\mathsf{cok}}
\newcommand{\N}{\mathbb{N}}
\newcommand{\Z}{\mathbb{Z}}
\newcommand{\R}{\mathbb{R}}

\newcommand{\moprnh}{\mathcal M^o_{\mathcal P}(\mathbf {n'},\mathcal H)}
\newcommand{\mopnh}{\mathcal M^o_{\mathcal P}(\mathbf n,\mathcal H)}
\newcommand{\mopmh}{\mathcal M^o_{\mathcal P}(\mathbf m,\mathcal H)}

\newcommand{\mpppcnh}{\mathcal M^{++}_{\mathcal P}(\mathcal C,\mathbf n,\mathcal H)}
\newcommand{\mpppcmh}{\mathcal M^{++}_{\mathcal P}(\mathcal C,\mathbf m,\mathcal H)}

\newcommand{\elzg}{\textnormal{El}(\Z G)}
 
\newcommand{\elpzg}{\textnormal{El}_{\mathcal P}(\Z G)}

\newcommand{\elpnh}{\textnormal{El}_{\mathcal P}(\mathbf n,\mathcal
  H)}
\newcommand{\elpkh}{\textnormal{El}_{\mathcal P}(\mathbf k,\mathcal
  H)}
\newcommand{\elpnprimeh}{\textnormal{El}_{\mathcal P}(\mathbf{n'},\mathcal  H)}
\newcommand{\elpmh}{\textnormal{El}_{\mathcal P}(\mathbf m,\mathcal  H)}

\newcommand{\elph}{\textnormal{El}_{\mathcal P}(\mathcal H)}
\newcommand{\elphp}{\textnormal{El}_{\mathcal P'}(\mathcal H')}
\newcommand{\elniHi}{\textnormal{El}(n_i,H_i)}

\newcommand{\mpns}{\mathcal M_{\mathcal P}(\mathbf n, S)}
\newcommand{\mps}{\mathcal M_{\mathcal P}(S)}
\newcommand{\mpnzg}{\mathcal M_{\mathcal P}(\mathbf n, \ZG)}

\newcommand{\mpnG}{\mathcal M_{\mathcal P}(\mathbf n, \Z_+G)}
\newcommand{\mpnh}{\mathcal M_{\mathcal P}(\mathbf n,\mathcal H)}
\newcommand{\mpkh}{\mathcal M_{\mathcal P}(\mathbf k,\mathcal H)}
\newcommand{\mpmh}{\mathcal M_{\mathcal P}(\mathbf m,\mathcal H)}

\newcommand{\mpppnh}{\mathcal M^{++}_{\mathcal P}(\mathbf n,\mathcal
  H)}

\newcommand{\mppph}{\mathcal M^{++}_{\mathcal P}(\mathcal H)}

\newcommand{\uph}{\mathcal U_{\mathcal P}(\mathcal H)}
\newcommand{\upnh}{\mathcal U_{\mathcal P}(\mathbf n ,\mathcal H)}

\newcommand{\mopcmh}{\mathcal M^o_{\mathcal P}(\mathcal C, \mathbf
  m,\mathcal H)}

\newcommand{\mopckh}{\mathcal M^o_{\mathcal P}(\mathcal C, \mathbf
  k,\mathcal H)}

\newcommand{\mopcnh}{\mathcal M^o_{\mathcal P}(\mathcal C, \mathbf
  n,\mathcal H)}

\newcommand{\mopcnprimehprime}{\mathcal M^o_{\mathcal P}(\mathcal C, \mathbf
  {n'},\mathcal H')}

\newcommand{\moprcnh}{\mathcal M^o_{\mathcal P}(\mathcal C, \mathbf{n'},\mathcal H)}

\newcommand{\moph}{\mathcal M^o_{\mathcal P}(\mathcal H)}

\newcommand{\mopG}{\mathcal M^o_{\mathcal P}(\mathbb Z_+ G)}
\newcommand{\mopprimeG}{\mathcal M^o_{\mathcal P'}(\mathbb Z_+ G)}
\newcommand{\mopnG}{\mathcal M^o_{\mathcal P} (\mathbf n,\mathbb Z_+  G)}

\newcommand{\moppcG}{\mathcal M^o_{\mathcal P}(\mathcal C, \mathbb Z_+ G)}

\newcommand{\mopch}{\mathcal M^o_{\mathcal P}(\mathcal C,\mathcal H)}

\newcommand{\mopcnG}{\mathcal M^o_{\mathcal P}(\mathcal C,\mathbf  n,\ZZ_+G)}

\newcommand{\rc}{\mathcal R^{\mathcal C}}

\newcommand{\benum}{\begin{enumerate} }
\newcommand{\eenum}{\end{enumerate} }
\newcommand{\bpmat}{\begin{pmatrix} }
\newcommand{\epmat}{\end{pmatrix} }

\newcommand\GL{\textnormal{GL}}

\newcommand\El{\textnormal{El}}
\newcommand\ZG{\mathbb Z G}
\newcommand\ZZ{\mathbb Z}

\newcommand\RR{\mathbb R}

\renewcommand{\phi}{\varphi}

\newcommand{\myId}{I}
\newcommand{\rep}[1]{(#1)}
\numberwithin{equation}{section}
\newtheorem{lemma}[equation]{Lemma}

\newtheorem{corollary}[equation]{Corollary}
\newtheorem{theorem}[equation]{Theorem}
\newtheorem{prop}[equation]{Proposition}
\newtheorem{proposition}[equation]{Proposition}
\newtheorem{observation}[equation]{Observation} 
\theoremstyle{definition}

\newtheorem{definition}[equation]{Definition}
\newtheorem{definitions}[equation]{Definitions}
\newtheorem{question}[equation]{Question}
\newtheorem{problem}[equation]{Problem}

\newtheorem{example}[equation]{Example}

\newtheorem{remark}[equation]{Remark}

\newtheorem{convention}[equation]{Standing Convention} 

\input xy
\xyoption{all}
\newcommand{\perm}[1]{{\mathsf{(#1)}}}
\newcommand{\perme}{\mathsf{e}}

\title{Flow equivalence of G-SFTs}
\author{Mike Boyle}
\author{Toke Meier Carlsen}
\author{S\o ren Eilers}
\date{\today}

\begin{document}

\begin{abstract}
In this paper, a $G$-shift of finite type ($G$-SFT)
is a  shift of finite type together with 
a free continuous shift-commuting action by a finite group $G$. 
We reduce  the classification of $G$-SFTs
up to equivariant flow equivalence to an
algebraic classification of a class of poset-blocked matrices
over the integral group ring of $G$. For a special
case of two irreducible components with $G=\Z_2$,
we compute explicit complete invariants.
We relate our matrix structures to the Adler-Kitchens-Marcus
group actions approach. We give examples of $G$-SFT applications,
including a new connection to involutions of cellular automata. 

 \end{abstract}

\maketitle
\setcounter{tocdepth}{1} 
\tableofcontents

\section{Introduction}

The shifts of finite type (SFTs) are the fundamental building 
blocks of symbolic dynamics.
Elaborations of these include the class of $G$-SFTs:
SFTs equipped with a continuous action 
by a group $G$ which commutes with the shift. 
Apart from a few remarks, in this paper $G$-SFT 
means $G$-SFT with $G$ finite and acting freely
($gx=x$ only when $g=e$).

We 
 will  give an algebraic classification of these 
$G$-SFTs up to equivariant flow equivalence
($G$-flow equivalence). This generalizes the 
$G$-flow equivalence classification for irreducible 
 $G$-SFTs in \cite{BSullivan} and the 
Huang flow equivalence  classification for general SFTs without group action 
\cite{mb:fesftpf,mbdh:pbeim}.

Square matrices over $\Z_+G$ (the positive cone in the integral
group ring of $G$) present $G$-SFTs. When such matrices $A$ and $B$
present nontrivial\footnote{By definition, an
    irreducible SFT is trivial iff
    it contains only one orbit. 
    The only trivial mixing SFT is a single point; so, a 
    trivial mixing  $G$-SFT has $G=\{e\}$.}
mixing $G$-SFTs, 
these $G$-SFTs are $G$-flow equivalent if and only if
there exist $n$ in $\mathbb N$ and identity matrices
  $I_j$, $I_k$ such that 
  there are matrices $U,V$ in the elementary group $\textnormal{El}(n,\Z G)$
  such that $U((I-A)\oplus I_j)V=(I-B)\oplus I_k$.
This reduces the dynamical  classification of $G$-SFTs up to
  $G$-flow equivalence to the algebraic  classification of 
  square matrices over $\ZG$ up to the stabilized elementary equivalence
  described above.
  The algebraic classification is in general highly nontrivial, but 
  far more manageable  than the dynamical problem.

For nonmixing $G$-SFTs,
this stabilized elementary $\Z G$  
equivalence no longer implies  $G$-flow equivalence.
To get an analogous result (Theorem \ref{classification2})
for general $G$-SFTs, we
consider $G$-SFTs presented by matrices in a
special block triangular form, with entries in an $ij$ block
lying in $\Z  H_{ij}$ for some union $H_{ij}$ of
double cosets 
of $G$, and their equivalence by elementary matrices
from a restricted class  
subordinate to this blocked coset structure. 
Using this, we  classify  $G$-SFTs up to $G$-flow equivalence
(Theorems \ref{classification2} and \ref{classification}).

The paper is organized as follows.

In Section \ref{motivationsec}, we discuss  uses of $G$-SFTs, and
  application of the results of this paper.
This includes a new use of the $G$-SFT structure, for
   cellular automata.
Especially,
the main theorem of
the current paper is a key result for the classification
up to
flow equivalence of a large class of irreducible sofic
shifts in \cite{bce:sofic}. 
In Section \ref{backgroundsec}, we give a
bare-bones
review of
the necessary background.
In Section \ref{sec:coset}, we introduce the notion
of coset structure,
  crucial for defining our restricted class of elementary matrices,
  and define  various classes of related matrices. 
We  prove Proposition \ref{pro:right form} which
tells us that, 
in order to classify $G$-SFTs up to
$G$-flow 
equivalence,  
it is enough to work
with square matrices over 
$\ZZ_+G$ having a special block form.

In Section \ref{sec:main theorem}, we present
our classification, with comments.
  The full classification statement is Theorem \ref{classification};
  the essence is the simpler statement of Theorem \ref{classification2}
  for the case the $G$-flow equivalence respects the ordering of
  irreducible components. 
  In each statement, for matrices in a suitable class  chosen
  to present the $G$-SFTs, a matrix condition is necessary and sufficient for
  the $G$-flow equivalence.
  In Section \ref{sec:1to2},  we
  prove a strengthened version of necessity of the matrix condition,
    Theorem \ref{necessary}. 
In Section \ref{sec:factorizationtheoremsetting},
  we present  the functorial Factorization Theorem
\ref{theoremfactor}, 
which we later use to prove sufficiency of the matrix condition 
  in Theorem \ref{classification}, and develop
    the setting for its proof. In Section
    \ref{sec:factorizationtheoremproof}, we give that proof.
  Section \ref{proof}
  contains the proof of Theorem \ref{classification}
 (a short argument appealing to Theorems \ref{necessary} and
    \ref{theoremfactor}), 
a result on range of invariants and a finiteness result. 
In Appendices \ref{cohomologyappendix}, \ref{permutation appendix},
and \ref{resolvingappendix}, we establish 
three  types of positive $\elph$-equivalences which we
use in  the paper.

In Appendix \ref{akmappendix}, we relate 
the  group actions viewpoint on $G$-SFTs 
developed in the Adler-Kitchens-Marcus paper \cite{akmgroup} 
to  our matrix-based setup.
(This is analogous to
the relation of matrices and linear transformations in elementary
linear algebra.) 
Adler, Kitchens and Marcus  were concerned only with
nonwandering SFTs. We explain
in Remark \ref{correspondenceremark}  how our coset structures
give a kind of algebraic calculus to describe the transitions
among irreducible components of a general $G$-SFT.

In Appendix \ref{specialcaseappendix}, we work out 
algebraic invariants of $G$-flow equivalence
for a special class of systems with exactly two irreducible components. 
For the subclass with  $G=\Z_2$, we give a complete algebraic classification, 
with algorithms to answer all questions  (at least,
the natural ones we thought of). This supplies a tractable collection
of examples and points to some of the issues involved in a general
algebraic classification.


\section{Applications of G-SFTs}\label{motivationsec}


$G$-SFTs  arise in several contexts, including the following.
(Below, we occasionally assume 
background reviewed in Section \ref{backgroundsec}.)

\begin{enumerate}[(i)] 
\item In \cite{akmgroup},
  Adler, Kitchens and Marcus introduced invariants of nonwandering
  $G$-SFTs (and a more general class), and used these
  in \cite{akmfactor} to   
classify factor maps between irreducible SFTs up to almost 
topological conjugacy.  Here, a construction replaces a
given factor map with a map $\phi\colon S'\to S$, equivalent up to almost conjugacy, 
which is constant $n$-to-1. The map $\phi$  gives rises to a
continuous function $\tau : S \to S_{  n}$
(with $S_{  n}$ the group of permutations on ${  n}$ symbols). 
This $\tau$ is used as a skewing function to generate a $G$-SFT (with
$G=S_{  n}$) extension $T$ of $S$, with maps $\pi : T\to S$ and
$\alpha: T\to S'$ such that $\pi = \phi \circ \alpha$. 
Group invariants of the $S_{  n}$-action on $T$  are then related to $\phi$
and used for the classification, which requires further constructions.

\item  Field, Golubitsky and Nicol
  used $G$-SFTs  in  studies of ``symmetry in chaos'' and
  related equivariance 
  \cite{Field1983, FieldGolubitsky,FieldNicol2004}.
  
\item The Liv\v{s}ic theorem, restricted to dimension zero, states 
  that two real-valued H\"{o}lder functions on an irreducible
  SFT are cohomologous if and only if on each periodic orbit
  the sum of their output values is the same. $G$-SFTs arise
  as a case of the study of 
  analogous rigidity possibilities for skew products in terms of
  group-valued skewing functions
  (see \cite{parrylivsic,schmidtlivsic,BoSc2}).

\item  
  In \cite{SullivanSmaleFlow1997, SullivanErratum1998,SullivanTwistwise1998},
  Michael
  Sullivan introduced 
    ``twistwise flow equivalence''. 
  Here  $G=\Z_2$. For an SFT basic set of the return map to
  a cross-section under a flow 
  on a 3-manifold,
  the return map to the cross-section induces
  a map on the local stable set which is orientation preserving or
  reversing. This 
   additional data is given  by a function $\overline X \to \Z_2$,
  which can be encoded as a matrix $A$ over $\ZG$. 
  Sullivan found $G$-FE invariants of $A$  
  to produce new invariants of the template 
  (branched two-manifold fitted with an expansive semiflow) for the flow. 
  The complete algebraic $G$-FE classification (for $G=\Z_2$) for the
  mixing case   gave further invariants
  (see \cite[Sec. 7]{BSullivan} and \cite{SullivanBeyond2005} for more).

\item The mapping class group of a nontrivial
  irreducible SFT (see \cite{bc:mcg})
  is the countable group
  of homeomorphisms of the mapping torus of the SFT which respect
  orientation of the suspension flow, up to isotopy. 
   This is a challenging group to understand. 
   The classification of irreducible $G$-SFTs 
   up to $G$-flow equivalence has provided at least a little
   information: for example (see   \cite[Theorem 8.6]{bc:mcg}),
   if the irreducible SFT is defined by a matrix $A$ such that
   $\det (I-A)$ is odd, then 
  the free orientation
  preserving involutions of the mapping torus are contained  in a finite set
  of conjugacy classes of this mapping class group.
  (For a  speculative application related to mapping class
  groups, see   Remark \ref{BFrepn}.)

\item The main result of the current paper
  has already been applied to the classification of sofic shifts up to 
  flow equivalence in \cite{bce:sofic}. We will give more detail on this
  below. 

\end{enumerate}

The automorphism group $\text{Aut}(S)$ of an SFT $S$
is the group of homeomorphisms
commuting with $S$. 
The $G$-SFTs  offer a different tool set  to 
the study of finite subgroups of $\text{Aut}(S)$. 
For simplicity, we consider just the case of $G=\Z_2$.
If $U$ is a  free involution in $\text{Aut}(S)$,
then  the pair $(S,U) $ presents a $G$-SFT.  
When $S$ is $\sigma_n$, 
the full shift on $n$ symbols, this can also be described as
a free involution of 
 an invertible one-dimensional cellular automata.
 A longstanding question (recalled in  
\cite[p.492]{FiebigUperiodic1993})
 asks  for $n=2$ whether two such involutions
 must be conjugate in $\text{Aut}(\sigma_n)$. 
 This conjugacy in the group $\text{Aut}(\sigma_n)$
 is equivalent to topological conjugacy of the corresponding $G$-SFTs.
  By Proposition \ref{groupexfacts}, conjugacy in the group
   is equivalence to strong shift equivalence over $\Z_+G$ of
   presenting matrices.
   
 A necessary condition 
 for this conjugacy (a test) is that the $G$-SFTs
 be $G$-flow equivalent. (Equivalently,  the
 induced involutions of
 the mapping torus are conjugate in the mapping class group of $\sigma_n$.) 
 We prove next that this necessary
 condition is satisfied.

 For the proof, we assume the background and notation of
 Section \ref{backgroundsec}. 
For a matrix $A$ over a commutative ring,
the sequence $(\text{trace}(A^n))_{n\in \N}$
determines and is determined by the polynomial $\det (I-tA)$.
We have $\text{trace}(A^n) = |\text{Fix}(\sigma_A^n)|$ when 
$\sigma_A$  is an SFT 
 defined by
 a matrix $A$ over $\Z_+$. 
 We  assume in the proof some familiarity with the Bowen-Franks group
 invariant of flow equivalence. 
Finally, note that a full shift on an odd number of symbols has
an odd number of fixed points, and therefore admits 
no  free involution.  


\begin{theorem}\label{involutiontheorem} 
  Suppose  $G=\Z_2$, $k$ is a positive integer and
$\sigma_{2k}$ is the full shift  on
  $2k$ symbols. Then there is a  free involution 
  commuting with $\sigma_{2k}$. For any two such involutions
  $U,U'$ the  $G$-SFTs $(\sigma_{2k},U)$, $(\sigma_{2k},U')$ 
   are $G$-flow equivalent. 
\end{theorem}

\begin{proof}
  Let $G=\{e,g\}$. For the existence claim, choose a free
  order two permutation of the $2k$ symbols; this defines
  a one-block code which defines a free involution $U$
  of $\sigma_{2k}$.
  
  Now let $A$ be an $m\times m$  matrix over $\Z_+G$ presenting
  a skew product $G$-SFT isomorphic to one 
defined by $(U,\sigma_{2k})$. 
This  $A$ presents a skewing function on an SFT $\overline T$
which is a factor of $\sigma_{2k}$  under the map $\pi $ which
collapses 
$G$-orbits to points. 
Therefore (e.g. by \cite[Theorem A]{FiebigUperiodic1993}), 
$\det (I -t\overline A)=1-2kt$. So, for every $n$,
\[
|\text{Fix}(T^n)|=|\text{Fix}(\overline T^n)|=
(2k)^n\ .
\]

Let $\text{trace}(A^n)
= \alpha_n e + \beta_n g $.
From the structure of the skew product
construction, one can check that $\alpha_n$ is the number of
fixed points of $\overline T^n$ whose preimages are contained in
$\text{Fix}(T^n)$. 
As $\pi$ maps $\text{Fix}(T^n)$ 2-to-1 into
$\text{Fix}(\overline T^n)$,
we have $\alpha_n = (1/2)(2k)^n$, and therefore
$\beta_n = (1/2)(2k)^n$, for each $n$.
This forces $\det (I-tA)=\det(I-tB)$,
for the $1\times 1$ matrix  $B=(k(e+g))$, because $B^n= \big((1/2)(2k)^ne + (1/2)(2k)^ng\big)$. 
Therefore $\det(I-tA) =  e -t k(e+g)$. 

Because $G=\Z_2$ and
$1-2k$ is odd,
by \cite[Theorem 8.1]{BSullivan}
the matrix $I-A$ is $\text{El}(m,\ZG)$-equivalent
to a diagonal matrix, $D$, which must have determinant
$ e - k(e+g)$. It can happen, for some $k$,
that $ e - k(e+g)$  factors in $\ZG$.
Nevertheless, 
$D$ must be $(e-k(e +g) )\oplus I_{m-1}$.
This is because the Bowen-Franks group
for $\overline T$ 
is isomorphic to  $\Z_{2k-1}$ (because
$\det (I-tA) = 1-2kt$) and is also isomorphic 
to 
$\cok (I-\overline A)$.  
   So,   
 the $\text{El}(\ZG)$ class of $I-A$
 is the same for any choice of free involution $U$.
It follows from \cite[Theorem 6.4]{BSullivan}
(the case $\mathcal P=1$ of Theorem \ref{classification2}) that
all these $G$-SFTs fall in the same 
$G$-flow equivalence class. 
   \end{proof} 
  
  By definition, the involutions $U,U'$ are conjugate on periodic points
  if there is a bijection
  $\text{Per}(T) \to \text{Per}(T')$ which intertwines the
  actions of $(T,U)$ and $T,U'$.
A consequence of Ulf Fiebig's work 
\cite{FiebigUperiodic1993}
(or the arguments above) is that free involutions of
$\sigma_{2k}$ are conjugate on periodic points
(for $k=1$ this is explicitly contained in
\cite[Corollary 1.12]{FiebigUperiodic1993}).
Fiebig did much more, in particular for $G$ actions which
need not be free. In contrast to the free case, we do not
have a $\Z_+G$ matrix framework which handles the nonfree actions. 

\begin{problem}\label{nonfreeproblem}
  Develop a  useful matrix framework of matrices over $\Z_+G$ for 
  $G$-SFTs for which the $G$ action need not be free.
  The framework should in particular
  capture topological conjugacy and flow
  equivalence of the $G$-SFT.
    \end{problem} 
The papers \cite{FiebigUperiodic1993, SilverWilliams2005} are 
 relevant to Problem \ref{nonfreeproblem}. 
 A solution to Problem \ref{nonfreeproblem} would
   give a $\ZG$ matrix framework for 
 the entire vast collection of finite subgroups 
 of $\text{Aut}(\sigma_n)$. It would also naturally involve 
 general, reducible $G$-SFTs.

%
%
%

The discussions above involved $G$-SFTs which are irreducible as SFTs.
The advance of the current paper is in addressing $G$-flow equivalence
of $G$-SFTs which are reducible.
We expect the general reducible case to be meaningful to
related applications (especially, given a solution to Problem
\ref{nonfreeproblem}),  but 
for the most part we have not  developed reducible applications
related to the items above. 
However, our main result, Theorem  \ref{classification}, is
an essential tool for 
our classification in \cite{bce:sofic} of a large collection 
of irreducible  sofic shifts up to flow equivalence
(those which are \lq\lq point extension type\rq\rq , or PET).
We emphasize that our results on  $G$-FE for reducible $G$-SFTs is
used for the flow equivalence classification of sofic shifts which
are irreducible (or, equivalently for flow equivalence,  mixing). 

For simplicity, we will describe only a subclass  $\mathcal C$ of
the PET sofic shifts. Let $\mathcal C$ be the class of 
nontrivial irreducible strictly sofic  shifts
such that for the right Fischer cover $\pi: X\to Y$, the set 
$M= \{ x: |\pi^{-1}(\pi x)|>1\}  $ is a closed proper subset and satisfies 
the following condition for some finite group $G$:
there is a shift-commuting
embedding of $T$  to $M$ which takes $G$-orbits to the fibers
of $\pi$.
Given the easily computed flow equivalence class of the irreducible
SFT $X$, we show in \cite{bce:sofic} that the $G$-flow equivalence
class of $T$ is a complete classification invariant for the
flow equivalence class of the sofic shift $Y$.
For every $X$,
 every $G$-SFT $G$-flow equivalence class arises
in this construction.   This theorem appeals to the full
strength of our classification result Theorem  \ref{classification}
(as noted in \cite[Remark 7.14]{bce:sofic}). 

We  give next an example to indicate how this works. 

\begin{example} \label{feexample} 
Let $A= \left(\begin{smallmatrix} a&b&c&d \\
  b&a&d&c\\ 0&0&k&\ell \\ 0&0&\ell &k
\end{smallmatrix}\right) $ 
be a matrix in which the nonzero letters represent positive
integers. There is a free  involution $\gamma$ on the edge SFT $\sigma_A$
coming from a graph automorphism $\gamma$ corresponding to the 
involution  of vertices $1\leftrightarrow 2$, $3\leftrightarrow 4$.
(E.g., there are $d$ edges from vertex 1 to vertex 4, and
these are mapped bijectively to the $d$ edges from vertex 2 to vertex 3.)
The edge SFT $\sigma_A$, together with $\gamma$, is a
$G$-SFT, with $G=\Z_2$. 

Now take any $4\times 4$ matrix $B$ which is entrywise greater than $A$.
Define a one block code $\phi$ from $X_B$ which  maps two edges (the symbols
of $X_B$) to the same symbol if and only if they are edges of
$X_A$ paired by $\gamma$. The image shift is a mixing strictly sofic
shift. Now suppose
$A'= \left(\begin{smallmatrix} a&b&c'&d' \\
  b&a&d'&c'\\ 0&0&k&\ell\\ 0&0&\ell&k
\end{smallmatrix}\right) $,
again with letters representing positive integers, and again with
$B$ entrywise greater than $A'$. Construct $\phi'$ from $X_B$ just as
$\phi$ was constructed. Are the two image sofic shifts flow equivalent?
The result of \cite{bce:sofic} tell us they are if and only if
the $G$-SFTs on $X_A$ and $X_A'$ are $G$-flow equivalent.

Here, for any choices of the letters, we can translate to the
completely worked classification of Appendix 
\ref{specialcaseappendix}
and look up the answer.
For the translation, first $A$ becomes
$\left(\begin{smallmatrix} ae+bg & ce+dg \\ 0 & ke + \ell g
\end{smallmatrix}\right) $.
There is an isomorphism from $\ZG$ to the subring $R$ of $\Z^2$
consisting of the $(\alpha,\beta)$ such that $\alpha \equiv \beta \mod 2$,
given by $ae+bg\mapsto (a+b,a-b):=(\alpha,\beta) $. 
We apply  this map  entrywise to $A$ 
to arrive
at a matrix
$\left(\begin{smallmatrix} (\alpha_p,\beta_p) &(\alpha,\beta)
  \\ 0& (\alpha_q,\beta_q) 
\end{smallmatrix}\right) $,
where $(\alpha_p,\beta_p)=(a+b  , a-b  )$, 
$(\alpha_q,\beta_q)=(k+\ell  ,k-\ell )$ and $(\alpha, \beta ) =
(c+d, c-d)$. 
 For $A'$ we do the same, arriving at
the same matrix except that
$(\alpha',\beta')$ replaces $(\alpha,\beta)$.
From Proposition \ref{soeasy}
or Proposition \ref{doubled}, we determine the ideal $J$ to which
the classifying Theorem \ref{classifeasydecide} applies; and then we
can determine from Theorem \ref{classifeasydecide} whether the
$G$-flow equivalence holds. Example \ref{z2example} gives a complete
discussion of the classification for
the case $ae+bg = 54 - 42g$, $ke+ \ell g= 16e -8g$. 
\end{example}

\numberwithin{equation}{subsection}

\section{Background} \label{backgroundsec}

This section provides  a
bare-bones review of the background material,
assuming some familiarity with the subject.
For basic background on shifts of finite 
type, see   
\cite{bpk:sd,dlbm:isdc}. 
For a detailed presentation with proofs of the 
basic theory of $G$-SFTs and $G$-flow equivalence for finite 
$G$, see \cite{BSullivan}.
The basic ideas of skew product
constructions are of fundamental importance in various 
branches of dynamics; the exposition in \cite{BSullivan} 
is tailored to our topic and also includes facts specific 
to it.
See \cite{BoSc2} for further developments, and a correction 
  \cite[Appendix A]{BoSc2} to \cite{BSullivan}.

\subsection{{\bf Shifts of finite type and matrices over $\Z_+$.}} 
Given an $n\times n$  square matrix $A$ over $\ZZ_+=\{0,1,\dots\}$, let   $\mathcal G_A$ 
be a graph (in this paper, graph means directed graph) with vertex set 
$\{ 1, \dots , n\}$, edge set $\mathcal E=\mathcal E_A$ and adjacency 
matrix $A$. Define $X_A$ to be the subset of 
$\mathcal E^\Z$ realized by bi-infinite paths in $\mathcal G_A$. With the natural 
topology, $X_A$ is 
a zero-dimensional compact metrizable space. 
The homeomorphism $\sigma_A :X_A \to
X_A$ given by the {\em shift map} $\sigma_A$, defined by $ (\sigma_A (s))_i = s_{i+1}$, 
is the {\it edge SFT} defined by
$A$.  Every SFT is 
topologically conjugate to some edge SFT. 

\subsection{{\bf Matrices over $\Z_+G$.}}
Let $G$ be a finite group, let $\ZZ G$ be the integral group ring of $G$, and let 
$\ZZ_+ G$ be the subset containing the elements 
$\sum_{g\in G} n_g g$ with $n_g\geq 0$ for all $g$. 
Suppose $A$ is a  square matrix  over $\Z_+G$. 
Let $\overline A$ denote the standard \emph{augmentation} of $A$: 
the matrix over $\ZZ_+$ obtained 
by applying entrywise the standard augmentation map, 
 $\sum_{g\in G} n_g g \mapsto \sum_g n_g$.
 
By an \emph{irreducible }
matrix $A$ over $\Z G$ we mean a square matrix over $\Z_+ G$ 
whose augmentation 
$\overline A$ is an irreducible 
matrix.  An  \emph{irreducible component} of $A$ is a 
maximal  irreducible principal submatrix of $A$. A matrix $A$ is said to be \emph{essentially irreducible} if it has a unique irreducible component. If $A$ is essentially irreducible, then its unique irreducible component is called the \emph{irreducible core} of $A$.

An element $\sum_g n_g g$ of $\Z G$ is \emph{$G$-positive} 
when $n_g>0$ for all $g\in G$.{\footnote{\lq\lq $G$-positive\rq\rq\ 
  replaces the term \lq\lq very positive\rq\rq\ used in
  \cite{BSullivan}.}
A matrix $A$ over $\ZZ G$ is \emph{$G$-positive} if every entry is 
 $G$-positive.

\subsection{{\bf $G$-SFTs.}}\label{GSFTs}
In this paper, by a \emph{$G$-SFT} we mean an SFT together with a free continuous 
action on its domain by a finite group $G$ which commutes with the shift. 
 (In general, a ``$G$-SFT'' is not  
restricted to free actions 
or finite groups.)
Two $G$-SFTs are {\it $G$-conjugate} 
(isomorphic as $G$-SFTs)  if there is a topological 
conjugacy between them which intertwines their $G$ actions.  
For a {\it left $G$-SFT}, the $G$ action is from the left: 
$gh: y\mapsto g(hy)$ ($h$ acts first). 
For a {\it right $G$-SFT}, the $G$ action is from the right: 
$gh: y\mapsto (yg)h$ ($g$ acts first). 

\begin{convention} \label{leftconvention} 
Unless mentioned otherwise, 
in this paper a $G$-SFT is a left $G$-SFT (although we might 
sometimes repeat the declaration for clarity). This is the choice which
aligns with matrix invariants (see \cite[Appendix A]{BoSc2}).
(The $G$-SFTs of \cite{akmgroup} are implicitly
left $G$-SFTs;  the $G$-SFTs of  
\cite[p.493]{akmfactor} are right $G$-SFTs.)
\end{convention} 

Suppose $A$ is a  square matrix  over $\Z_+G$. 
Then $A$ can be interpreted as the adjacency 
matrix of a labeled graph $\mathcal G_A$, where the 
underlying graph is $\mathcal G_{\overline A}$, 
and the label of an edge of $\mathcal G_{\overline A}$ 
is the corresponding element of $G$ (so if the $(s,t)$ entry of $A$ is $\sum_{g\in G} n_g g$, then there is for each $g\in G$, $n_g$
is the number of edges from $s$ to $t$ with label $g$). 
The labeled graph defines a 
{\it skewing function} $\tau_A: X_{\overline{A}}\to G$  which sends 
$x$ to the label of $x_0$. The skew product construction then 
gives a homeomorphism 
$T_A:  X_{\overline{A}}\times G \to  X_{\overline{A}}\times G $ 
defined by $(x,g)\mapsto (\sigma_{\overline{A}}(x), g\tau_A (x))$, and $T_A$ is 
an SFT. (We consider every map topologically conjugate  to an edge SFT  
to be SFT.) 
The continuous free left $G$ 
action $g: (x,g')\mapsto (x,gg')$ commutes with $T_A$. 
Together with this action, $T_A$ is a $G$-SFT.
The map collapsing $G$-orbits to points is
  given by $(x,g)\mapsto x$; it defines 
a factor map from
  the SFT $T_A$ to the edge SFT defined by $\overline A$. 
Every $G$-SFT is isomorphic to one presented 
as a group extension in this 
way by some $A$ over $\ZZ_+ G$. 

\subsection{{\bf Cohomology.}} Continuous functions  
$\tau$ and $\rho$ from an SFT $(X,\sigma )$ into $G$ 
are 
 {\em cohomologous}  (written $\tau \sim \rho$) 
if there is another continuous function $\psi$ from $X$ into 
$G$ such that for all $x$ in $X$, 
$\tau (x) =  [\psi(x)]^{-1}\rho (x)\psi(\sigma x)$.
In this equation, the product on the right is 
a product in the group $G$. This is the form 
appropriate for our consideration of {\it left} $G$-SFTs 
(for which $\tau$ skews from the right).  
For {\it right} $G$-SFTs we would use instead the 
equation 
$\tau (x) =  [\psi(\sigma x)]\rho (x)\psi(x) ^{-1}$ for all $x$. 
For nonabelian $G$, these coboundary equations  are not equivalent.
The following 
result is fundamental  for us.

\begin{proposition} \cite[Proposition 2.7.1]{BSullivan} \label{groupexfacts}
Suppose $G$ is a finite group and  
$A, B$ are square matrices over $\Z_+G$. Then the following are 
equivalent. 
\begin{enumerate}
\item
$A$ and $B$ are strong shift equivalent over $\Z_+G$. 
\item 
There is a topological conjugacy 
$\varphi\colon X_{\overline A}\to X_{\overline B}$ 
such that \\ $\tau_B \sim \tau_A \circ \varphi$. 
\item 
The $G$-SFTs $T_A$ and  $T_B$ are $G$-conjugate.
\end{enumerate} 
\end{proposition}  
As seen in the Section \ref{motivationsec},
  $G$-SFTs may arise in some setting directly, or in terms of
  a  function from an SFT  into $G$, corresponding to (2) above. Both possibilities
  are addressed by the  matrix invariant  (1).

\subsection{{\bf Flow equivalence.}} 
Let $Y$ be a compact metrizable space. 
In this paper, a  {\em flow} on $Y$ is a continuous  $\RR$-action on
$Y$ with no fixed point. 
Two flows are {\em topologically conjugate}, or 
{\em conjugate}, if 
there is a homeomorphism intertwining their $\RR$-actions. 
Two flows are {\em equivalent} if there is a homeomorphism 
between their domains taking $\R$-orbits to $\RR$-orbits and 
preserving orientation (i.e., respecting the direction of the flow).
A {\em cross-section} to a flow $\gamma:Y\times\RR\to Y$ is a closed subset $C$ of $Y$ such that the restriction
of $\gamma$ to $C\times\RR$ is a surjective local homeomorphism onto $Y$. In that case, the {\em return time function} $\tau_C:C\to\RR$ given by $\tau_C(x)=\min\{t>0:\gamma(x,t)\in C\}$ is well defined and continuous. The map $r_C:C\to C$ given by $r_C(x)=\gamma(x,\tau_C(x))$ is call the {\em return map of $C$}. A {\em section} of a flow is the return map of a cross-section of the flow.

For $i=1,2$ suppose $S_i: X_i\to X_i$ is a homeomorphism 
of a compact metrizable space, 
and $Y_i$ is its mapping torus with the induced suspension flow. 
  The homeomorphisms  $S_1,S_2$ are 
  {\it flow equivalent} if they
are  topologically conjugate to sections of 
 a common flow; 
equivalently, after a 
continuous time  change, the flows on $Y_1$ and $Y_2$ 
become topologically conjugate;
equivalently, 
there is a homeomorphism $Y_1\to Y_2$ 
which on each $Y_1$ flow orbit is an orientation preserving 
homeomorphism to a $Y_2$ flow orbit.  
A {\it flow equivalence}
 $S_1\to S_2$ is such a homeomorphism. 
%

By a {\em G-flow} we mean a flow together with a continuous 
free left $G$-action which commutes with the flow. 
A free $G$ action commuting with a section lifts to a free $G$ 
action commuting with the flow. 
Two $G$-flows are {\em $G$-conjugate} if 
the flows are topologically conjugate by a map which 
intertwines the $G$-actions.
Two $G$-flows are {\em G-equivalent} if the flows are equivalent by a map 
which intertwines the 
$G$-actions (i.e., by a $G$-flow equivalence).  


The standard theory carries over to the $G$ setting. 
We call two $G$-homeomorphisms {\em G-flow equivalent} if they are 
conjugate to $G$-sections of the same $G$-flow. $G$-sections of two 
$G$-flows are $G$-flow equivalent if and only if the flows are 
$G$-equivalent. 

If $A$ and $B$ are square matrices over $\Z_+G$,
  then a $G$-flow equivalence $T_A \to T_B$ of their $G$-SFTs
  induces a flow equivalence of the SFTs defined by their
  standard augmentations $\overline A, \overline B$.

\subsection{{\bf Positive equivalence.}} 
Suppose $B,B',U,V$ are $n\times n$ matrices over $\Z G$ with $U,V$
 in $GL(n, \mathbb Z G)$.
We say $(U,V): B\to B'$ is  an
  equivalence if $UBV=B'$.
  If $\{U,V\}$ is contained in a subset
  $\mathcal M$ of $GL(n, \mathbb Z)$, then it is an
  $\mathcal M$-equivalence, and the matrices $B,B'$ are
  $\mathcal M$-equivalent.

A {\em basic elementary matrix}
is a matrix $E_{st}(x)$, which denotes a square matrix 
equal to the identity except for perhaps the off-diagonal 
$st$ entry (so, $s\neq t$),
which is equal to an element $x$ of $\ZG$.
Suppose  $E=E_{st}(g)$ and $A$ is a square matrix 
over $\Z_+G$  such that $g$ is a summand of $A(i,j)$
(i.e., $g\in G$ and the coefficient of $g$ in 
  $A(i,j)$ is positive).
Then we say that each of the  equivalences 
\begin{align*}
 (E,I)\colon \rep{I-A} \to E(I-A)\ , \qquad 
 &(E^{-1},I)\colon E(I-A) \to \rep{I-A}\ , \\
 (I,E)\colon (I-A) \to (I-A)E  \ , \qquad  
&(I,E^{-1})\colon (I-A)E \to \rep{I-A} 
\end{align*} 
is a {\em basic positive $\ZZ G$-equivalence}. 
Here the  equivalences $(E,I)$ and $(I,E)$ are {\em forward} 
and the other two are 
{\em backward}. 
An equivalence $(U,V):\rep{I-A}\to \rep{I-B}$ is a {\em positive $\ZZ G$-equivalence} 
if it is a composition of basic positive equivalences.

A basic positive equivalence $\rep{I-A}\to \rep{I-B}$ induces a $G$-flow
equivalence $T_A \to T_B$. Every 
$G$-flow equivalence $T_A \to T_B$ is induced (up to isotopy, see \cite[Section 6]{mb:fesftpf}) 
by a  {\em positive $\ZZ G$-equivalence}. For a justification of this claim, we refer to
\cite{BSullivan}; for more on its place in the positive K-theory 
classifications for symbolic dynamics, see \cite{Boylepositivek}.  

The elementary group
  $\text{El}(n,\Z G )$ is the group of $n\times n$ matrices which
  are products of basic elementary matrices. A positive equivalence
  $(I-A)\to (I-B)$ 
  through $n\times n$ matrices is an   $\text{El}(n,\Z G )$
  equivalence, but in 
   general, 
an $\textnormal{El}(n,\Z G)$   equivalence
     need not be a positive $\ZZ G$
equivalence, even if $\overline A$ is primitive  (see for instance \cite[Example 4.3]{BSullivan}). Therefore, we do not in general have that an equivalence $\rep{I-A}\to \rep{I-B}$ induces a $G$-flow
equivalence $T_A \to T_B$.
Still,  we will in Theorem \ref{theoremfactor}
  show that if $A$ and $B$ satisfy
  specified conditions, and the equivalence
  $(U,V):\rep{I-A}\to \rep{I-B}$ preserves {specified 
  structures (the poset structure, the cycle components
  (see later in this section) and the coset structure
  (see Section \ref{sec:coset})), then
it must be a positive $\ZZ G$-equivalence and thus
induce  a $G$-flow equivalence $T_A \to T_B$ (see Theorem \ref{classification}).

For the proofs in Appendices \ref{cohomologyappendix}
and \ref{permutation appendix}, we 
will use the graphical viewpoint described next (this description can also be
 found in \cite{BSullivan}).

\subsection{{\bf A row cut basic positive equivalence.}}\label{subsec:rowcut} 
Suppose $(E,I)\colon \rep{I-A} \to \rep{I-B}$ is a basic forward
positive equivalence, $E=E_{st}(g)$. Then $A$ and $B$ agree except perhaps 
in row $s$, where 
\begin{align*} 
B(s,r)\ &= \ A(s,r)+gA(t,r) \qquad 
\textnormal{if } r\neq t\ ,\qquad \quad \text{and}  \\
B(s,t)\ &= \ A(s,t) + gA(t,t)-g\ \ . 
\end{align*}
Consequently the labeled graph $\mathcal G_B$ associated to $B$ is 
constructed from the labeled graph $\mathcal G_A $ as follows.
An edge $e$ from $s$ to $t$ with label $g$ is deleted from $\mathcal G_A$. 
Then, for each $\mathcal G_A$-edge $f$ beginning at $t$, 
 an additional edge (called $[ef]$) from $s$ to $r$ 
with label $gh$
(where $h$ is the $G$-label  of $f$ and 
$r$ is the terminal vertex of $f$)
 is added in to form $\mathcal G_B$. 
 We refer to this type of positive equivalence as a
$(g,s,t)$ {\it row cut} (of the matrix $A$, or of
   an edge $e$ labeled $g$),
   or just a row cut. When $E(s,t)=p\leq A(s,t)$, 
   we may likewise
   refer to the positive equivalence implemented by $(E,I)$ as
 a $(p,s,t)$ row cut. 

See Figure~\ref{rowcut} for an example 
of a $(g,s,t)$ {\it row cut} of an edge from $s$ to $t$ labeled $g$.
For a matrix example corresponding to Figure~\ref{rowcut},
  with $(s,t)=(1,2)$, and $r$ in  Figure~\ref{rowcut} set to $r=3$,    
we use matrices   
 $E=E_{st}(g)$ and  
\[
A= \begin{pmatrix} 
p_{11} & g+p_{12} & p_{13} \\ 
0 & h' &h'' \\ 
p_{31}&p_{32}&p_{33} 
\end{pmatrix} 
\quad \textnormal{and} \quad 
B= \begin{pmatrix} 
p_{11} & gh'+p_{12} & gh''+ p_{13} \\ 
0 & h'& h'' \\ 
p_{31}&p_{32}&p_{33} 
\end{pmatrix} 
\]
in which the $p_{ij}$ are arbitrary 
elements of $\ZZ_+ G$, 
suppressed from the figure,  and row $2$ 
has just two entries for simplicity. 
The change from $\mathcal G_A$ to $\mathcal G_B$ is the 
replacement of the dashed edge of the left 
graph with the dashed edges of the right graph. 
On the left, $g,h',h''$ are labels of edges $e,f',f''$; 
on the right $gh', gh''$ label edges named $[ef'],[ef'']$. 

\begin{figure}[htb]
	\begin{center}  
\centerline{
\xymatrix{  
{ } 
& { } & { } 
& 
*+[o][F-]{r}
\\ 
*+[o][F-]{s}
\ar@{-->}[rr]_{g}
& { } 
& 
*+[o][F-]{t}
\ar[ur]_{h''} 
        \ar@(r,d)^<<<<<{h'}[]  
&{ }
}
\xymatrix{ 
 { } & { } & { } & { } \\ 
 { } &  \ar@{~>}[r]  & { } & { }
}
\xymatrix{  
{ } 
& { } & { } 
& 
*+[o][F-]{r}
\\ 
*+[o][F-]{s}
\ar@{-->}[rr]_{gh'}
\ar@{-->}[urrr]^{gh''}
 & { } 
& 
*+[o][F-]{t}
\ar[ur]_{h''} 
        \ar@(r,d)^<<<<<{h'}[]  
&{ }
}
}
	\end{center}
	\caption{A row cut of an edge from $s$ to $t$.}
	\label{rowcut}
\end{figure}
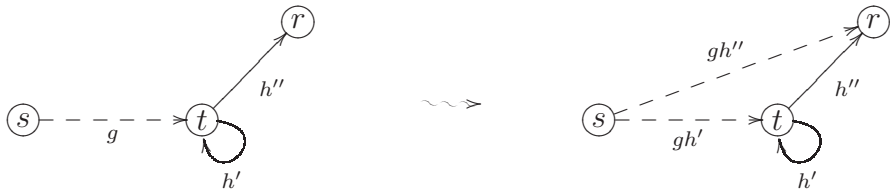

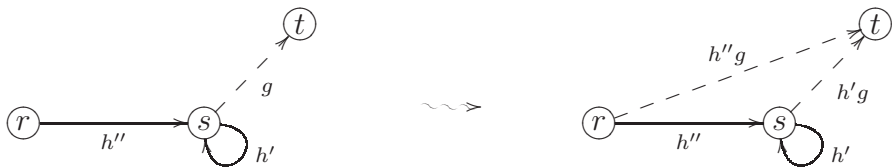
\begin{figure}[htb]
	\begin{center}  
\centerline{
\xymatrix{  
{ } 
& { } & { } 
& 
*+[o][F-]{t}
\\ 
*+[o][F-]{r}
\ar[rr]_{h''}
& { } 
& 
*+[o][F-]{s}
\ar@{-->}[ur]_{g} 
        \ar@(r,d)^<<<<<{h'}[]  
&{ }
}
\xymatrix{ 
 { } & { } & { } & { } \\ 
 { } &  \ar@{~>}[r]  & { } & { }
}
\xymatrix{  
{ } 
& { } & { } 
& 
*+[o][F-]{t}
\\ 
*+[o][F-]{r}
\ar[rr]_{h''}
\ar@{-->}[urrr]^{h''g}
 & { } 
& 
*+[o][F-]{s}
\ar@{-->}[ur]_{h'g} 
        \ar@(r,d)^<<<<<{h'}[]  
&{ }
}
}
	\end{center}
	\caption{A column-cut of an edge from $s$ to $t$.}
	\label{colcut}
\end{figure}

The correspondence of the graphs $\mathcal G_A, \mathcal G_B$
induces a bijection of 
$\sigma_A$-orbits and $\sigma_B$-orbits,  e.g. 
$$\dots\  b\, e\, f''\, c\, e\, f'\, f'\, d \ \dots 
\ \ \leftrightarrow \ \ 
\dots\  b\, [ef'']\, c\, [ef']\, f'\, d \dots \ \ .
$$ 
This bijection of orbits does not arise from a bijection of points for the 
SFTs, but it does correspond to a $G$-equivariant homeomorphism of their mapping tori (after 
changing time  by a factor of 2 over the clopen sets $\{ x: x_0=[ef]\}$, the 
new flow is conjugate to the old one), which lifts to a $G$-equivariant 
homeomorphism of the respective mapping tori.

\subsection{{\bf A column cut basic positive equivalence.}}\label{subsec:colcut}
The other type of 
basic forward
positive equivalence is   
$(I,E)\colon \rep{I-A} \to \rep{I-B}$, with $E=E_{st}(g)$. 
Then $A$ and $B$ agree except perhaps 
in column $t$, where 
\begin{align*} 
B(r,t)\ &= \ A(r,t)+A(r,s)g \qquad 
\textnormal{if } r\neq t\ ,\qquad \quad \text{and}  \\
B(s,t)\ &= \ A(s,t)+A(t,t)g -g\ \ . 
\end{align*}
The labeled graph $\mathcal G_B$ associated to $B$ is 
constructed from the labeled graph $\mathcal G_A $ as follows.
An edge $e$ from $s$ to $t$ with label $g$ is deleted from $\mathcal G_A$. 
Then, for each $\mathcal G_A$-edge $f$ ending at $s$, 
 an additional edge (called $[fe]$) from $r$ to $t$ 
with label $hg$
(where $h$ is the $G$-label of $f$ and 
$r$ is the initial vertex of $f$)
 is added in to form $\mathcal G_B$. 
We refer to this type of positive equivalence as a 
$(g,s,t)$ {\it column cut} (of the matrix $A$, or of
   an edge $e$ labeled $g$), or just a column cut.
Figure \ref{colcut} gives the column-cut analogue of Figure
\ref{rowcut}.

\begin{example}\label{cutandtrim}
When $A$ is a square matrix  over $\Z_+G$ with some diagonal
    entry $A_{tt}=0$, 
    row cuts may be applied to zero out each entry of 
    column $t$, and then column cuts to zero out each entry of 
    row $t$, giving a permutation matrix $P$ and a smaller matrix $M$ such that
    $P^{-1}AP = M \oplus 0$. 
    We call this type of operation a \emph{trim move}. Note
    $A$ and $M$ define $G$-SFTs which are $G$-flow equivalent. 
    For instance,
with $G$ the symmetric group $S_{  3}$,  
using $A_{22}=0$  
and the standard notation recalled in
  Definition \ref{ex_o}, 
the matrix 
\[
A=\begin{pmatrix}\perm{23}&\perm{12}\\\perm{13}&0\end{pmatrix}
  \]
  can be row cut, using 
  $(E_{12}(\perm{12}),I): (I-A) \to (I-B)$,
to the matrix 
\[
B= \begin{pmatrix}\perm{23}+\perm{132}&0\\\perm{13}&0\end{pmatrix}\ , 
\]
which can be column cut, using
  $(I,E_{21}(\perm{13}): (I-B) \to (I-C)$, to the matrix 
\[
C=\begin{pmatrix}\perm{23}+\perm{132}&0\\0&0\end{pmatrix}
  =(\perm{23}+\perm{132}) \oplus 0 
    \ . 
  \]
\end{example}

\subsection{{\bf Poset-blocked matrices.}} 
\label{posetblocksubsec} 
In order to handle general $G$-SFTs (having more than one irreducible
component), as for the case $G=\{e\}$ addressed in 
\cite{mb:fesftpf,mbdh:pbeim} 
we need to consider matrices with block structures 
corresponding to irreducible components and transitions between them. 
Throughout this paper, $\mathcal P = \{ 1, \dots , N \}$ 
is a poset (partially ordered set) with a partial order relation $\preceq $ 
chosen such that $i\preceq j \implies i\le j$. We will write $i\prec j$ if $i\preceq j$ and $i\ne j$. For a vector of positive integers $\mathbf n = (n_1, \dots , n_N)$, let $n=\sum_{j=1}^Nn_j$, and let $\mathcal I_i=\{(\sum_{j=1}^{i-1}n_j)+1,(\sum_{j=1}^{i-1}n_j)+2,\dots,\sum_{j=1}^in_j\}$ for each $i\in\mathcal P$. If $s\in\{1,2,\dots,n\}$, then we let $i(s)$ be the unique integer such that $s\in\mathcal{I}_{i(s)}$. For an $n\times n$ matrix $A$ and $i,j\in\mathcal P$, we let $A\{i,j\}$ denote the submatrix of $A$ obtained by deleting the rows corresponding to indices not belonging to $\mathcal I_i$ and columns corresponding to indices not belonging to $\mathcal I_j$. The
matrix $A$ is called an \emph{$( \mathbf n, \mathcal P )$-blocked matrix} if $A\{ i,j \} \neq 0 \implies i \preceq j$.
For $S$ a subset of $\Z G$, 
  we let $\mpns$ denote the set of $( \mathbf n, \mathcal P )$-blocked matrices with entries
  in $S$, and we
  let $\mps$ be the union over $\mathbf n$ of the sets
  $\mpns$.

The set $\mopnG$ is the set of matrices $A$
in $\mathcal M_{\mathcal P} (\mathbf n , \ZZ_+G)$ 
satisfying the following conditions: 
\begin{enumerate} 
\item
Each diagonal block $A\{i,i\}$ is essentially irreducible. 
\item 
If $i\prec j$, then there are $r>0$, an index $s$ corresponding to a row in the irreducible core of $A\{i,i\}$, and an index $t$ corresponding to a column in the irreducible core of $A\{j,j\}$ such that 
$A^r(s,t) \neq 0$. 
\end{enumerate}

For $A\in \mopnG$,  $i$ in $\mathcal P$ corresponds explicitly
to an irreducible component of the SFT defined by $X_{\overline A}$,
  with $i\prec j$ if and only there exists an orbit in $X_{\overline A}$
    backwardly asymptotic to component $i$ and forwardly asymptotic to
    component $j$.  If 
    $A\in \mopG$ and $A'\in \mopprimeG$, then a flow equivalence $T_A\to T_B$
    induces a poset isomorphism $\mathcal P \to \mathcal P'$. We say
    the flow equivalence \emph{respects the component order} if this isomorphism
    is $k\mapsto k$, $1\leq k\leq N$. 

We let $\mopG$ be the union over $\mathbf n$ of the sets 
$\mopnG$.

\subsection{{\bf Cycle components.}}\label{sec:cyclecomponents} 
For a matrix $A$ in $\mopnG$, 
a {\it cycle component} 
 is 
a component $i$ in $\mathcal P$ such that the 
irreducible core of $\overline{A\{i,i\}}$ is a cyclic 
permutation matrix. The cycle components contribute 
significantly to technical difficulties in the classification 
of $G$-SFTs up to $G$-flow equivalence. 
For $A$ in  $\mopnG$, 
$\mathcal C(A)$ denotes the set of its cycle components.  
For a subset 
$\mathcal C$ of $\mathcal P$, 
\[
\mopcnG \ := \ \{ A \in 
\mopnG : \mathcal C(A)=\mathcal C \}\ .
\]

We let $\moppcG$ be the union over $\mathbf n$ of the sets 
$\mopcnG$.

\subsection{Stabilizations} \label{subsec:stabilization}

{An unblocked 
matrix $A'$ is a {\it stabilization} (or {\it 0-stabilization}) of an 
$m\times n$ matrix $A$ if $A$ equals an upper left corner of $A'$, and
$A'$ is zero in all remaining entries. (I.e.,
$A'(s,t) =A(s,t) $ if $1\leq s\leq m$ and $ 1\leq t \leq n$, and otherwise
$A'(s,t)=0$.) 
  A matrix $A'$ in $\mathcal M_{\mathcal P} (\mathbf n' , \Z G)$ is a 
{\it stabilization} (or {\it 0-stabilization}) of 
a matrix $A$ in $\mathcal M_{\mathcal P} (\mathbf n , \Z G)$ 
if $\mathbf{n'} \geq \mathbf n$
(i.e., $n'_i\ge n_i$
for $1\leq i\leq N$)
and for
$1\leq i,j \leq n$, the $i,j$ block submatrix of  
$A'$ is a stabilization of the  $i,j$ block submatrix of  
$ A$.}

{A matrix $M'$ in $\mathcal M_{\mathcal P} (\mathbf n' , \Z G)$ is a 
{\it 1-stabilization}  of
a matrix $M$ in $\mathcal M_{\mathcal P} (\mathbf n , \Z G)$
if $M'-I$ is a 0-stabilization of $M-I$.  
When its entries come from $\Z_+G$, a  matrix 
$A$ and its 0-stabilizations define the same $G$-SFT,
but it is $I-A$ and its 1-stabilizations which will share the
algebraic invariants for $G$-flow equivalence.}

\numberwithin{equation}{section}
\section{$(G,\mathcal P )$ coset structures} \label{sec:coset}

The classification up to $G$-flow equivalence of 
 $G$-SFTs $T_A$ defined by irreducible matrices $A$ over $\Z_+G$ 
 required a reduction
to the case that the \lq\lq weights group\rq\rq\ of $T_A$ is
all of $G$. (This essentially amounts to reducing to the case that
$T_A$ is mixing as an SFT, as recalled in Appendix \ref{akmappendix}.)
For general $A$ over $\Z_+G$, we will need
an analogous reduction on
the irreducible components of $A$, and then we will 
capture invariants of transitions between components 
using  double coset conditions.
In this section we prepare the formal structure for this. 
We begin with the double coset conditions. 
 
Below, $G$ is the given finite group and $\mathcal P=\{1, \dots , N\}$ 
is the given finite
poset,
with a partial order
relation $\preceq$ satisfying $i\preceq j\implies i\le j$. 
Let $H_i$ and $H_j$ be subgroups of $G$.
An $(H_i,H_j)$ 
double coset is a nonempty set equal to $H_igH_j$ for some 
$g$ in $G$.

\begin{definition} \label{defn:cosetstructure} 
 A $(G,\mathcal P)$ coset structure $\mathcal H$
is a function which assigns to each pair $(i,j)$ in 
$\mathcal P \times \mathcal P$ such that $i\preceq j$ 
a nonempty subset $H_{ij}$ of $G$ such that 
\begin{equation} \label{cosetstructurecondition}
i\preceq j \preceq k \implies 
H_{ij}H_{jk} \subset H_{ik} \ .
\end{equation}
Consequently, $H_{ii}$ (also denoted $H_i$) is a 
subgroup of $G$ and for $i\prec j$ 
$H_{ij}$ is a nonempty union of 
$(H_i,H_j)$ double  cosets. 
\end{definition} 
If $H_i,H_j$ are subgroups of an abelian group $G$,
  then  $H_iH_j$ is a group, and a double coset
  $H_igH_j$ is a coset $gH_iH_j$. For general $G$, 
  $\mathcal H$ is actually
  a double coset structure; we use \lq\lq coset structure\rq\rq\
  for brevity.

  \begin{definition}\label{ex_o}
    (Notation) $\mathcal P_n $ denotes the
    poset $\{1, \dots , n\}$ with the linear order:
    $i\prec j$ iff $i< j$.
\end{definition}
    In the next two examples, we consider
      $\mathcal H$  a  coset
      structure  for $(G      , \mathcal P_3 )$, with $G$ abelian,
      using additive notation. 
Here the possibilities for $\mathcal H$ are as follows. 
    \begin{itemize}
    \item
      $H_1,H_2,H_3$ are arbitrary subgroups of $G$,
    \item
      $H_{12}$ is an arbitrary  nonempty union of cosets of  $H_1+H_2$,
          \item
            $H_{23}$ is an arbitrary  nonempty union of cosets of  $H_2+H_3$,
          \item
            $H_{13}$ is an arbitrary union of cosets of $H_1+H_3$ 
            containing  $H_{12}+H_{23}$; because $G$ is abelian,
            $H_{12}+H_{23}$ is a union
            of cosets of    $H_1+H_2+H_3$. 
    \end{itemize}
    \begin{example} 
      Let $G= \Z /27$, $H_1=H_3= 9G = \{0,9,18\}$,
      $H_2=H_{12}=H_{23}=3G = \{0,3,\dots , 24\}$,
      $H_{13} = \{1,10,19\} \cup 3G$. So, $H_1+H_2+H_3=3G$. 
      Now $H_{13} $ contains a coset of $H_1+H_2+H_3$, but
      $H_{13} $ is not a union of cosets  of $H_1+H_2+H_3$.
    \end{example}
    \begin{example} 
    Let  $G=\Z /30$, 
    $H_1=15G$,  $H_2=6G$ and
    $H_3= 10G$. 
    Then $H_1+H_2=3G$, $H_2+H_3=2G$,  $H_1+H_3=5G$ and
    $H_1+H_2+H_3=G$. Because $H_{13}$ contains a coset of
    $H_1+H_2+H_3$, $H_{13}$ must be $G$. 
\end{example} 
\begin{remark}       For       $G$
      not necessarily abelian, with subgroups $H_i, H_j$,
      let us recall some elementary facts about the double cosets
      $H_igH_j$, $g\in G$. 
      As in Example \ref{ex_i})
  $H_iH_j$
  need not be a group;
  $H_igH_j$ need not be a coset of a subgroup of $G$;
  $H_igH_j$ is a union of right cosets of $H_i$ (or left cosets of
  $H_j$),  but the union is not arbitrary. 
  A double coset  $H_igH_j$ is the orbit of $g$ under the
  (right) action on $G$ by
  $H_i \oplus H_j$  given by $(h,k): g\mapsto h^{-1}gk$.
  Consequently, two $(H_i,H_j)$ double  cosets are equal or disjoint.
  Thus, if there are exactly $r$ $(H_i,H_j)$ double cosets in $G$,
  there are exactly $2^r -1$  sets which are nonempty unions of
$(H_i,H_j)$   double cosets. 
For $g\in G$, define the subgroup 
 \begin{align*} 
 M_{ij}(g) &=  \{ h\in H_i : \exists k \in H_j,\ h^{-1}gk = g\} \\
& =  \{ h\in H_i : g^{-1}hg \in H_j\}
 = H_i \cap gH_jg^{-1}\ .
 \end{align*}
  Then the isotropy group of $g$ for this action
  of   $H_i \oplus H_j$ on $G$ is 
  \[
  \{ (h,k) \in H_i \oplus H_j : h\in M_{ij}(g), k = g^{-1}hg \}\ ,
  \]
  with cardinality $|M_{ij}(g)|$, and therefore 
 $|H_igH_j | = |H_i||H_j|/|M_{ij}(g)|$. In particular,
as the cardinality $|M_{ij}(g)|$ of the isotropy group might vary with $g$,
so might the cardinality of the double coset $|H_igH_j | $.
  \end{remark} 

\begin{definition}\label{ex_o}
 (Notation)
    $S_{  n}$ is the group of permutations of
    $\{\mathsf 1, \dots , \mathsf n\}$
    (to avoid confusion, we reserve  sans serif numbers  to indicate the elements on which $S_{  n}$ acts).
    $A_{  n}$ is the alternating group in
        $S_{  n}$. 
    We use cycle notation to
     denote  elements of $S_{  n}$.  
For example, $\perm{123}$ is the cyclic permutation
$\mathsf 1 \to \mathsf 2 \to \mathsf 3$;  $\perm{12}\perm{13} = \perm{132}$. We let $\perme$
 denote the identity. For a subset $K$ of a group,
$\gen{K}$ denotes the subgroup generated by $K$; for example,
in $S_{  n}$, $\gen{\perm{12}}=\{ \perm{12}, \perme \}$.  
\end{definition}

\begin{example}\label{ex_i}
Suppose $(G,\mathcal P)=(S_{  3},
    \mathcal P_2)$ with coset structure $\mathcal H$. If
    $H_{11}=S_{  3}$ or  $H_{22}=S_{  3}$,
then $H_{12}$ must be $S_{  3}$.
 At the opposite extreme,
 if $H_{11}=H_{22}=\{\perme\}$,
  then $H_{12}$ can be any nonempty subset of $S_{  3}$.
    Now suppose $H_1= \gen{\perm{12}}$ and $H_2= \gen{\perm{13}}$.
    For $g$ in $S_{  3}$,
    \[
    |H_1gH_2| =
    |H_1||H_2|/|H_1\cap gH_2g^{-1}| = 4/|H_1\cap gH_2g^{-1}| \ . 
    \]
    So,  $|H_1 gH_2| =2 $ if
    $gH_2g^{-1}  =H_1$, and otherwise
    $|H_1 gH_2| =4 $ . 
There are exactly
two double cosets $H_1gH_2$:   
$D_1=\{ \perm{23}, \perm{132}\}$ and its complement $D_2$ in 
$S_{  3}$. Neither double coset is a subgroup of $S_{  3}$.
$D_1$ is the only coset of $H_1$ which is a double coset. 
$D_2$ is neither a left coset nor a right coset of a subgroup of  $S_{  3}$.
$D_1$ is a right coset of $H_1$ and a left coset of $H_2$,
but it is neither a left coset of $H_1$ nor a right coset of $H_2$.  
\end{example}
\begin{example} 
    Suppose $(G,\mathcal P)=(S_{  3},\mathcal P_3)$
    with coset structure $\mathcal H$.
Let $H_1 =H_{12} = \gen{\perm{12}}$,
  $H_2=\{ \perme\}$, 
$H_3= \perm{13}$ with
$ H_{12}= H_{1}$  and  
    $ H_{23}=H_{3}\cup \perm{23}H_{3}$ . 
As seen in Example \ref{ex_i}, 
there are two $(H_1, H_3)$ double cosets: 
$\{ \perm{23}, \perm{132} \}$, and its complement in $S_{  3}$.
    By \eqref{cosetstructurecondition},
    $H_{13} $ contains $H_{12}H_{23}$,
    which here contains $\{ \perme, \perm{23}\}$. 
Therefore     $H_{13} $ 
    intersects both $(H_1, H_3)$ double cosets,
    and  must equal 
    $S_{  3}$.
  \end{example}

\begin{definition} \label{defn:cosetstucturecohomology}
Two $(G,\mathcal P)$ coset structures $\mathcal H$, 
$\mathcal H'$ are $G$-cohomologous if there exist 
elements $\gamma_1, \dots , \gamma_N$ in $G$ such that 
\[
i \preceq j \implies H_{ij}= \gamma_i^{-1}H'_{ij}\gamma_j \ .
\]
\end{definition} 

The ``$G$'' in ``$G$-cohomologous'' matters
(see Example \ref{Gmatters}).
Still,  because $(G,\mathcal P)$
is fixed, 
we sometimes write just ``coset structure'' in place of 
``$(G, \mathcal P)$ coset structure''.

\begin{example}\label{Gmatters} Let 
  $\mathcal P$ be  
  the trivial poset $\mathcal P_1$. Then $(G, \mathcal P)$
  coset structures $\mathcal H, \mathcal H'$
  are cohomologous iff the groups
  $H_{1}, H'_1$ are conjugate subgroups of $G$.
  Define order two subgroups of
$ S_{  4}$, $H_{1}=\gen{\perm{12}\perm{34}}$ and $ H'_1=\gen{\perm{13}\perm{14}}$. 
  Let $H$ (isomorphic to $\Z/2\oplus \Z/2$) be the
  subgroup generated by $H_1$ and $H'_1$. 
  Then $H_{1}, H'_1$ are conjugate as subgroups of $S_{  4}$,
  but not as subgroups of $H$. 
\end{example}

  \begin{example}\label{Gabeliancosetstructureclass}
    Suppose $G$ is abelian and $\mathcal H$, $\mathcal H'$ are
    $(G,\mathcal P_2)$ coset structures such that
    $\mathcal H_1=\mathcal H'_1$ and $\mathcal H_2 = \mathcal H'_2$.
    If $|\mathcal H_{12}|=1=|\mathcal H'_{12}|$,
    Then $\mathcal H$    and $\mathcal H'$ are $G$-cohomologous if
    and only if  $\mathcal H'_{12}= g+    \mathcal H_{12}$
    for some $g$ in $G$. This 
 always holds if 
    $|\mathcal H_{12}|=1=|\mathcal H'_{12}|$,
but need not hold if e.g. 
$|\mathcal H_{12}|=2=|\mathcal H'_{12}|$.
  \end{example}

\begin{example}\label{ex_ii} Let $\mathcal H$ be a $(S_{  3}, \mathcal P_2)$
  coset structure, with $H_{11}=H_{22}=\{\perme\}$.
  We noted in Example \ref{ex_i} that
  any nonempty subset of $S_{  3}$ may serve as $H_{12}$. When $H'_{12}$ is another such subset, of course it is necessary that $|H_{12}|=|H_{12}'|$ for the structures to be cohomologous. But this is not a sufficient condition; for instance $H_{12}=\{\perme,\perm{12}\}$ gives a system which is not cohomologous
  to that from $H'_{12}=\{\perme,\perm{123}\}$, as is seen by applying
  the sign function to the defining relation.
\end{example}

\begin{definition} \label{defn:weight}
For a matrix $A$ over $\Z_+G$, 
with $\tau_A$ the associated labeling of edges  
of $\mathcal G_A$, 
the {\em weight} of a path of edges 
$p=p_1p_2\cdots p_k $  in $\mathcal G_A$ 
is defined to be 
$\tau_A(p)=\tau_A(p_1)\tau_A(p_2)\cdots \tau_A(p_k)$.
\end{definition}

\begin{definition} \label{defn:cosetstructureofA}
Suppose 
$A\in \mopG$. Then a 
$(G,\mathcal P)$ coset structure $\mathcal H$ for $A$ 
is defined as follows.  
\begin{enumerate}
\item 
For each $i\in \mathcal P$, choose a vertex $v(i)$ 
from the irreducible core of the block $A\{ i,i\}$. 
\item 
For $i\preceq j$,  $H_{ij}$ is the set of 
weights of paths from $v(i) $ to $v(j)$. 
\end{enumerate} 
The group $H_{ii}$ (also denoted $H_i$) was 
called a {\it weights group} for $A_{ii}$ in 
  \cite{BSullivan}.  
\end{definition}

\begin{example}\label{ex_iii}
With $G=S_{  3}$, consider the matrices
\[
A_1=\begin{pmatrix}
\perm{12}&\perm{132}\\
\perm{123}&\perme
\end{pmatrix},  
A_2=\begin{pmatrix}
\perm{12}&\perm{13}\\
0&\perm{12}
\end{pmatrix},  
A_3=
  \begin{pmatrix}0
&\perm{12}& 0\\
\perm{12}&\perme&\perm{12}\\
0&0&\perme
\end{pmatrix} .
\]
For $A_1 $, $\mathcal P = \mathcal P_1$.  
  If $v(1)=1$, then 
  $H_{1}=\gen{\perm{12}}$;
  if 
  $v(1)=2$, then from $\perm{132} =\perm{123}^{-1} $ we have 
  $H_{1}=\perm{123}\gen{\perm{12}}\perm{123}^{-1} =
\gen{\perm{13}}$.
  For $A_2$ and $A_3$, $\mathcal P= \mathcal P_2$.
For $A_2$, 
$H_{1}=H_{2}=\gen{\perm{12}}$ and
$H_{12}= S_{  3} \setminus \gen{\perm{12}}$. 
For $A_3$,
$H_{1}=H_{2}=\{\perme\}$;  
if $v(1)=1$ then  $H_{12}=\{\perme\}$,
and if
$v(1)=2$ then  $H_{12}=\{\perm{12}\}$. 
\end{example}

Even though the example above shows that they may be
    different,
all coset structures for $A$ will
  be $G$-cohomologous. To see this, suppose 
$v_1,v_2$ are vertices in the irreducible core 
of $A\{ i,i\}$ and $\gamma_i$ is the weight of a path 
from $v_1$ to $v_2$. 
Replacing a 
choice $v(i)=v_2$ with the choice $v(i)=v_1$ has the effect 
of replacing $H_{i}$ with $H'_i:=\gamma_i H_{i}\gamma_i ^{-1}$, 
and replacing $H_{ij}$ with $H'_{ij}:=\gamma_iH_{ij}$ 
when $i\prec j$. Therefore $\mathcal H$ and $\mathcal H'$ 
are $G$-cohomologous.

\begin{definition} \label{structureclassdefn} 
The {\it $(G,\mathcal P)$ coset structure class 
of $A$} is 
the $G$-co\-ho\-mo\-lo\-gy class  of  a 
$(G,\mathcal P)$ coset structure for $A$. 
\end{definition}

\begin{example}
  Let $G=\Z/4\Z$ with generator $g$.
  By Example \ref{Gabeliancosetstructureclass},
  $\left(\begin{smallmatrix} 2&2\\0&2 
  \end{smallmatrix}\right)$ and
    $\left(\begin{smallmatrix} 2&g\\0&2 
  \end{smallmatrix}\right)$ have the same coset structure class,
  but 
  $\left(\begin{smallmatrix} 2&1+g\\0&2 
  \end{smallmatrix}\right)$ and
    $\left(\begin{smallmatrix} 2&1+g^2\\0&2 
  \end{smallmatrix}\right)$ do not have the same coset structure class.
  \end{example}

  \begin{example}\label{whyelpnh} 
    Suppose $G$ is a nontrivial group. Define
    $A=\left(\begin{smallmatrix} 1&1&1\\0&1&1 \\
      0&1&1    \end{smallmatrix}\right)$ and
    $A'=\left(\begin{smallmatrix} 1&1+g&1\\0&1&1 \\
      0&1&1    \end{smallmatrix}\right)$ .
 Coset structures
    $\mathcal H, \mathcal H'$ for
    $A,A'$ are given by $H_1=H_2=H_{12}=H'_1=H'_2=\{1\}$ and
    $H'_{12} = \{1,g\}$. 
Because $|H'_{12}| \neq |H_{12}|$, 
the coset structures $\mathcal H, \mathcal H'$
    are not $G$-cohomologous (and therefore, by our structure theorem,
    the matrices $A,A'$ define $G$-SFTs which are not $G$-flow equivalent).
    However, $(I-A)$ and $(I-A')$ are
    $\textnormal{El}(\mathbf n, \mathcal P_2,\Z_+G)$-equivalent,
    with $\mathbf n= (1,2)$, because
    \begin{align*}
      E_{13}(g)(I-A)
      &=
    \left(\begin{smallmatrix} 1&0&g\\0&1&0 \\
      0&0&1    \end{smallmatrix}\right)
    \left(\begin{smallmatrix} 0&-1&-1\\0&0&-1 \\
      0&-1&0    \end{smallmatrix}\right)
    =\left(\begin{smallmatrix} 0&-1-g&-1\\0&0&-1 \\
      0&-1&0    \end{smallmatrix}\right)
    =I-A'\ .
    \end{align*} 
     A smaller example is given
    by   
    $C=\left(\begin{smallmatrix} 1&2\\0&2
    \end{smallmatrix}\right)$ and
    $C'=\left(\begin{smallmatrix} 1&2+g\\0&2
    \end{smallmatrix}\right)$, which respectively define $G$-SFTs
    which respectively
    are $G$-flow equivalent to those defined by $A$ and $A'$.
   \end{example}

Example \ref{whyelpnh} shows that 
  $\text{El}(\mathcal P, \mathbf n, \Z G )$ 
equivalence, even by matrices with $i$th diagonal
  block entries in $\Z_+H_i$ for all $i$, does not 
  give an algebraic relation  capturing $G$-flow equivalence.
  This is why we are 
   led to introduce $\elpnh$
  equivalence later in this section.
\begin{example}
Notice that it can happen that not every coset structure in the
$(G,\mathcal P)$ coset structure class of a matrix $A$ is a
$(G,\mathcal P)$ coset structure for
$A$. If for example $(G,\mathcal P)=(S_{  3},\mathcal P_2)$ and 
\[
A=\begin{pmatrix}\perme&\perme\\0&\perme\end{pmatrix},
\]
 then $H_{11}=H_{12}=H_{22}=\{\perme\}$ is the only $(S_{  3},\mathcal P_2)$ coset structure for
    $A$; but $H'_{11}=H'_{22}=\{\perme\}$,
   $H'_{12}=\{\perm{12}\}$ is also in the $(S_{  3},\mathcal P_2)$ coset structure class
   of $A$.
\end{example}

The classes 
$\mathcal M_{\mathcal P} (\mathbf n , \Z G)$ and
$\mopcnG$ 
  used below were defined in
Subsections \ref{posetblocksubsec} and \ref{sec:cyclecomponents}.
\begin{definition}\label{defn:mopcnh}
Let   $\mathcal H$ be a $(G,\mathcal P)$ coset structure. 
  \begin{itemize}
    \item 
      $\mpnh$ is the set of matrices $A\in\mathcal M_{\mathcal P} (\mathbf n ,
      \Z G)$ such that for $i\preceq j$, the entries of $A\{i,j\}$
belong to $\Z H_{ij}$. 
\item 
 $\mopcnh$ is the set of  $A$ in 
$\mpnh\cap\mopcnG$
  such that
  $\mathcal H$ is a $(G,\mathcal P)$ coset structure for $A$.
\item
  $\mopch = \cup_{\mathbf n}\, 
  \mopcnh$.
  \item $\moph = \cup_{\mathcal C}\, \mopch$.
  \end{itemize} 
\end{definition}


\begin{definition}\label{conditionc1}  
A matrix in $\moppcG$
satisfies Condition $\mathcal C_1$ (or, is $\mathcal C_1$) 
if it belongs to $\mopcnG$ for $\mathbf n =(n_1,n_2,\dots,n_N)$ 
such that the following condition holds for all $i\in\mathcal P$: $n_i=1$ if and only if $i\in\mathcal C$.  
\end{definition}

A square matrix $A$ over $\ZZ_+ G$ is \emph{nondegenerate} if it has no zero row and no zero
column. Notice that this is equivalent to 
the graph $\mathcal{G}_{\overline{A}}$ being nondegenerate
(that is, every vertex of $\mathcal{G}_{\overline{A}}$ belongs to a bi-infinite path).
The next proposition applies in particular to a nondegenerate
  matrix $A$ (which cannot be
  nilpotent).

\begin{proposition}\label{pro:right form}
  Let $A$ be a 
nonnilpotent square matrix over $\ZZ_+ G$. Then there is an $N$, a
  partial order $\preceq$ on $\mathcal{P}:=\{1,\dots,N\}$ satisfying
  $i\preceq j\implies i\le j$, a subset $\mathcal C$ of $\mathcal P$,
  a $(G,\mathcal P)$ coset structure $\mathcal H$, and a 
  $\mathcal C_1$ matrix $B\in\mopch$ with each diagonal block irreducible such that $T_A$ and $T_B$ are
  $G$-flow equivalent. 
$B$ can be produced algorithmically
    from $A$. 
\end{proposition}
\begin{example}
The first two of the  
matrices
\[
\begin{pmatrix}
0&\perm{12}\\\perm{13}&0
\end{pmatrix},
\begin{pmatrix}
\perme+\perm{12}
\end{pmatrix},
\begin{pmatrix}
\perme&\perm{12}&0\\0&0&\perme\\
0&0&\perme
\end{pmatrix}
\]
are not $\mathcal C_1$, and the third does not satisfy the
diagonal block irreducibility
 condition of Proposition \ref{pro:right form}.
 The algorithm outlined in the proof below will give, respectively,
\[
\begin{pmatrix}
\perm{123}
\end{pmatrix},
\begin{pmatrix}
\perme&\perm{12}\\\perme&\perm{12}
\end{pmatrix},
\begin{pmatrix}
\perme&\perm{12}\\0&\perme
\end{pmatrix}
\]
which
 satisfy both conditions.
See Example \ref{cutandtrim} for details of the first example.

For an example of the step $B=DAD^{-1}$ in the proof
of Prop. \ref{pro:right form}, we use $\mathcal P = \mathcal P_2$,
$G= S_{  3}$, $H_1 = H_2 = H_{12} = A_{  3}$,
$\nu(1)=1$ and  $\nu(2)=4$. Then   
\begin{align*}
  B=   DAD^{-1} = &
\left(\begin{smallmatrix}
  \perme&\perme&\perme &\perm{123}&\perm{123} \\
  \perm{123}&\perm{213}&\perm{123}&\perme&\perm{123}& \\
  \perme&\perme&\perme&\perme&\perme& \\
  0&0&0&  \perme & \perme \\
  0&0&0&  \perm{132} & \perm{123}  
\end{smallmatrix}\right)  
 \\ 
= \left(\begin{smallmatrix}
  \perme&0&0 &0&0 \\ 0&\perm{12}&0 &0&0 \\ 0&0&\perme &0&0 \\
  0&0&0 &\perme&0
  \\ 0&0&0 &0&\perm{23}  \end{smallmatrix}\right)
&\left(\begin{smallmatrix}
  \perme&\perm{12}&\perme &\perm{123}&\perm{13} \\
  \perm{23}&\perm{123}&\perm{23}&\perm{12}&\perm{123}& \\
  \perme&\perm{12}&\perme&\perme&\perme& \\
  0&0&0&  \perme & \perm{23} \\
  0&0&0&  \perm{12} & \perm{132}  
\end{smallmatrix}\right)
\left(\begin{smallmatrix}
  \perme&0&0 &0&0 \\ 0&\perm{12}&0 &0&0 \\ 0&0&\perme &0&0 \\
  0&0&0 &\perme&0 \\ 0&0&0 &0&\perm{23}  \end{smallmatrix}\right) 
\end{align*}
\end{example}

\begin{proof}[Proof of Prop. \ref{pro:right form} ]
 If a diagonal entry of $A$ is zero, then we may apply the trimming
  move of Example \ref{cutandtrim} to produce a smaller matrix
  over $\Z_+G$, which defines a flow equivalent $G$-SFT,
  and which is also not nilpotent. 
  By iteration of this move, we may assume without     loss of generality
  that every diagonal entry of $A$ is nonzero.
Then 
    there is a
    poset $\mathcal P = \{1, \dots ,N\}$, a vector
 $\mathbf n =(n_1,n_2,\dots,n_N)$ and  
    a permutation matrix $P$ and such that $PAP^{-1}\in \mopnG$.
 We may assume $PAP^{-1}=A$. 

 Let $\mathcal C$ be the set of
    cycle components for $A$ in $\mathcal P$.
         For $1\leq i \leq N$, set $m_i = \sum_{h<i} n_h$ and
     $\mathcal I_i =\{m_i+1, \dots , m_i +n_i\}$.
     Let $A\{i,j\}$ be the submatrix of $A$ on
     indices $\mathcal I_i \times \mathcal I_j$.
     For $i\in \mathcal C$,
     $\mathcal I_i$ must now be a      singleton. If 
       $i\notin \mathcal C $ and $\mathcal I_i$ is a singleton,
  then there is a state splitting $A\to A'$ which increases
  $n_i$ from 1 to 2 such that  every entry of $A'\{i,i\}$ is nonzero.
  On iteration of this process, passing to 
  a new $A$ and a new $\mathbf n$,
  we have $A\in \mopnG$,  with each  $A\{i,i\} $  irreducible,  
  such that  $i\in \mathcal C \iff n_i = 1$.

We now turn to the coset structure. For each $i$ in $\mathcal P$, pick an 
index $v(i)$ in $\mathcal I_i$, 
and with these choices define a 
$(G,\mathcal P)$ coset structure $\mathcal H$ 
for $A$ as in Definition \ref{defn:cosetstructureofA}. 
Next, for each $i\in \mathcal P$ and 
each index $s$ in $\mathcal I_i$,  
because $\mathcal I_i = \mathcal I^*_i$ we may
choose $d_s$ the weight of some path 
from $v(i)$ to $s$ and $b_s$ the weight 
of some path from $s$ to $v(i)$. Let 
$D$ be the diagonal matrix with 
$D(s,s)=d_s$.  Define $B=DAD^{-1}$. 
There is a positive $\Z G$-equivalence 
from $I-A$ to $I-B$ (Proposition \ref{cohomologyaspositive}), 
so $T_B$ is $G$-flow equivalent to $T_{A}$.

Suppose $s\in \mathcal I_i$ and 
$t\in \mathcal I_j$.  We claim that 
$B(s,t)\in \Z H_{ij}$. 
To prove this claim, note that  $b_td_t$ is 
the weight of a path from
$t$ to $t$. Pick $k>0$ such that 
$(b_td_t)^k=e$. Then 
$(b_td_t)^{k-1}b_t=(d_t)^{-1}$, 
and therefore  $(d_t)^{-1}$ is the weight of a 
path from $t$ to $v(j)$. 
Therefore 
$B(s,t) = d_sA(s,t)(d_t)^{-1} $ is 
the weight of a path from $v(i)$ to 
 $v(j)$, and therefore is in $ H_{ij}$.   

Because $I-A$ and $I-B$ are positive $\Z G$  
equivalent, a coset structure for $B$ 
must be $G$-cohomologous to 
the coset structure $\mathcal H$ of $A$. 
By construction, a coset structure $\mathcal H'$ for 
$B$ defined from the vertex choices $v(i)$
has $\mathcal H'_{ij} \subset \mathcal H_{ij}$ if $i\preceq j$. 
By the $G$-cohomology, this containment 
must be equality, so 
$\mathcal H$ is a coset structure for $B$, 
and $B\in \mopch$. 
\end{proof}

\begin{remark}\label{correspondenceremark} 
Suppose  $A\in \mpppnh$ and $T_A$ is the $G$-SFT defined by $A$
  in Section \ref{GSFTs}. 
Let $\pi \colon X_{\overline A} \times G \to X_{\overline A} $ be the map
collapsing $G$-orbits. For $i\in \mathcal P$,
let $\overline T_i : \overline X_i \to \overline X_i$ be the 
mixing  $\Z$-SFT defined by $\overline{A\{i,i\}}$.
Let $X_i = \overline X_i \times H_i$, and let $T_i$ be the restriction of
$T_A$ to $X_i$. Then $H_i$ is the isotropy group of $X_i$, and 
$\pi^{-1}(\overline X_i) $ is the disjoint union
of $ |G/H_i|$ mixing $\Z$-SFTs;
the mixing SFTs $g_1X_i, g_2X_i$ are equal if and only if
$g_1H_i =g_2H_i$. 
Let us write $g_1X_i \to g_2X_j$ if there exists a point
backwardly asymptotic to $g_1X_i $ and forwardly asymptotic to
$g_2X_j$. Then $g_1X_i \to g_2X_j$  if and only if
$X_i \to (g_1)^{-1}g_2X_j$. From the  definition of $T_A$, we
see $X_i \to gX_j$ iff $g\in H_{ij}$, and this holds for
every element or no element of a 
double coset $H_igH_j$. 
\end{remark}

We  now give terminology for the equivalences fundamental
  to our results.
For a positive vector $\mathbf n=(n_1,\dots,n_N)$, let $\elpnh$ be the group of matrices generated by the basic 
elementary matrices in $\mpnh$. 
We define an \emph{$\elpnh$-equivalence} to be an equivalence 
$(U,V): \rep{I-A}\to \rep{I-B}$ with $U,V$ in $\elpnh$ 
and $A,B$ in $\mpnh$.

A \emph{basic positive $\elpnh$-equivalence} is a basic 
positive $\Z G$-equivalence which is also an
 $\elpnh$-equivalence. 
A  \emph{positive $\elpnh$-equivalence} is 
defined to be a composition of 
basic positive $\elpnh$-equivalences.

\emph{An (positive) $\elph$-equivalence} 
from $I-A$ to $I-B$ is defined to be 
any (positive) $\elpnh$-equivalence $(U,V): \rep{I-A'}\to \rep{I-B'}$ 
such that $A',B'$ are stabilizations of $A,B$. It is easy to check that if there is a (positive) $\elph$-equivalence 
from $I-A$ to $I-B$, and a (positive) $\elph$-equivalence 
from $I-B$ to $I-C$, then there is a (positive) $\elph$-equivalence 
from $I-A$ to $I-C$.
%

\section{The main results} \label{sec:main theorem}

  We can now state the primary result of the paper.
 The $\mathcal C_1$ condition in the statement was given 
in Definition \ref{conditionc1}.
Given 
$\gamma= (\gamma_1, \dots , \gamma_N)\in G^N$,
we 
define  $D^{\mathbf n}_{\gamma} \in \mpnG$ 
to be the diagonal matrix $D$ such that $D(s,s) = \gamma_{i(s)}$
  (i.e., for $1\leq i\leq N$, the  diagonal block
  $D\{i,i\}$  is $\gamma_i I$).

\begin{theorem} \label{classification2} 
  Suppose $G$ is a finite group; $\mathcal P=\{1,\dots,N\}$ is a poset; 
  $\mathcal H$ 
and $\mathcal H'$
 are $(G, \mathcal P)$ coset structures; 
 $A$ and $B$ are nondegenerate  $\mathcal C_1$ matrices in
$ \mopcnh$ 
 and  $\mopcnprimehprime$, respectively. 
   Then  the following are equivalent. 
\begin{enumerate} 
\item 
  There is a $G$-flow equivalence of
  the $G$-SFTs $T_A$ and $T_B$ which respects the component ordering. 
\item 
For some   $\gamma \in G^N$, for 
 $C= (D^{\mathbf n}_{\gamma})^{-1}BD^{\mathbf n}_{\gamma}$  
  there exist $\mathbf m$ and $\mathcal C_1$  stabilizations
 $A^{<0>},C^{<0>}$ 
of $A,C$ in $\mopcmh$ such that
the matrices 
$I-A^{<0>}$ and $I-C^{<0>}$ 
are $\elpmh$-equivalent. 
\end{enumerate} 
\end{theorem}

By  condition (2) of Theorem \ref{classification2},
  there is an immediate corollary.
\begin{corollary} 
The $(G,\mathcal P)$ coset structure class is
  an invariant of component-order-respecting $G$-flow equivalence.
\end{corollary} 
We will give a more complicated statement next for a 
  flow equivalence which need not respect component order.   
  By Proposition \ref{pro:right form},
  $G$-SFTs can be presented by matrices in the form addressed by
  Theorem \ref{classification2}.  Therefore, 
  Theorem \ref{classification2} (as elaborated in Theorem
  \ref{classification}) gives
  a classification of $G$-SFTs up to $G$-flow equivalence. 

If $\mathcal P=\{1,\dots,N\}$ and $\mathcal P'=\{1,\dots,N\}$ are finite posets given by partial order relations $\preceq$ satisfying $i\preceq j\implies i\le j$, $\alpha:\mathcal{P}\to\mathcal{P'}$ is a poset isomorphism, and $\mathbf n=(n_1,\dots,n_N)$ is a positive vector, then we denote by $\alpha^*(\mathbf n)$ the vector $(m_1,\dots,m_{N})$ with $m_i=n_{\alpha^{-1}(i)}$, and we let $Q^{\mathbf n}_\alpha$ be the $n\times n$ matrix (where $n=\sum_{j=1}^Nn_j=\sum_{i=1}^{N}m_i$) with entries $q_{rs}=1$ if $r=l+\sum_{j=1}^{\alpha(i)-1}m_j$ and $s=l+\sum_{k=1}^{i-1}n_k$ for some $i\in\mathcal P$ and some $l\in\{1,\dots,n_i\}$, and 0 otherwise. Then $(Q^{\mathbf n}_\alpha)^{-1}AQ^{\mathbf n}_\alpha\in\mathcal M_{\mathcal P}(\mathbf n,\Z_+G)$ whenever $A\in\mathcal M_{\mathcal P'}(\alpha^*(\mathbf n),\Z_+G)$.

\begin{theorem} \label{classification} 
Let $G$ be a finite group, let $\mathcal P=\{1,\dots,N\}$ and $\mathcal P'=\{1,\dots,N'\}$ be finite posets given by partial order relations $\preceq$ satisfying $i\preceq j\implies i\le j$, let $\mathcal C$ and $\mathcal C'$ be subsets of $\mathcal P$ and $\mathcal P'$ respectively, and let $\mathcal H=\{H_{ij}\}_{i,j\in\mathcal P}$ and $\mathcal H'=\{H'_{ij}\}_{i,j\in\mathcal P'}$ be $(G,\mathcal P)$ and $(G,\mathcal P')$ coset structures.	
	
Suppose $A\in\mathcal{M}_{\mathcal{P}}^o(\mathcal{C},\mathbf n,
\mathcal{H})$ and $B\in\mathcal{M}_{\mathcal{P'}}^o(\mathcal{C'},\mathbf n',\mathcal{H'})$ are
$\mathcal C_1$ 
stabilizations of nondegenerate 
matrices.
Then  the following are equivalent. 
\begin{enumerate} 
\item 
The $G$-SFTs $T_A$ and $T_B$ 
are $G$-flow equivalent. 
\item 
  There exists a poset isomorphism $\alpha:\mathcal{P}\to\mathcal{P'}$ such that $\alpha(\mathcal{C})=\mathcal{C'}$, and there exists
$\gamma\in G^N$
  such that $H_{ij}=\gamma_i^{-1}H'_{\alpha(i)\alpha(j)}\gamma_j$ for $i,j\in\mathcal{P}$ with $i\preceq j$, and such that 
for the matrix $$C= (D^{\mathbf n}_{\gamma})^{-1} (Q^{\mathbf n}_{\alpha})^{-1} B Q^{\mathbf n}_{\alpha} D^{\mathbf n}_{\gamma}$$
the following holds: 
there exist 
 $\mathbf m$ and $\mathcal C_1$ 
stabilizations $A^{<0>}$, $C^{<0>}$ 
in $\mopcmh$ of $A,C$ 
such that the following holds: \\ 
the matrices 
$I-A^{<0>}$ and $I-C^{<0>}$ 
are $\elpmh$-equivalent. 
\end{enumerate} 
\end{theorem} 

In Theorem \ref{classification}, 
the implication $(1) \implies (2)$ is a part of 
the more general result Theorem \ref{necessary}. 
The implication $(2) \implies (1)$
is a consequence 
of a much stronger constructive statement, 
the Factorization Theorem  \ref{theoremfactor}.  
We therefore postpone the proof of Theorem \ref{classification} to
Section \ref{proof}.

After finite reductions,
  it is now clear that there is a procedure for determining $G$-flow
  equivalence of $T_A$ and $T_B$ (hence of $G$-SFTs) if
  there is a procedure for answering the following.
\begin{question} \label{decisionquestion}
  Let $G$ be a finite group.
  Given a $(G,\mathcal P)$
  coset structure $\mathcal H$, $\mathcal C \subset \mathcal P$
  and $\mathcal C_1$ matrices $A,C$ in
  $\mopcnh$, does there exist a procedure to decide
  whether the following holds: 
there exists  $\mathbf m$, 
   and $\mathcal C_1$ 
stabilizations $A^{<0>},C^{<0>}$ 
in $\mopcmh$ of $A,C$,  
such that 
the matrices 
$I-A^{<0>}$ and $I-C^{<0>}$ 
are $\elpmh$-equivalent.
\end{question}

We will give a more algebraically phrased
  question next for future reference. 
  An answer yes to Question \ref{decisionquestion2} gives
  an answer yes to Question \ref{decisionquestion}
(using $L= I-A$, $M = I-B$). 
  \begin{question} \label{decisionquestion2}
  Let $G$ be a finite group.
  Given a $(G,\mathcal P)$
  coset structure $\mathcal H$, $\mathcal C \subset \mathcal P$,
and  $\mathbf n$ such that $n_i=1$ for $i\in \mathcal C$, 
  and matrices $L,M$ in
  $\mpnh$, 
 is there  a procedure to decide
  whether the following holds: 
  there exist  $\mathbf m$, 
  with $m_i=1$ for $i\in \mathcal C$,
  and 
1-stabilizations $L^{<1>},M^{<1>}$ of $L,M$ 
in $\mpmh$,  
such that 
the matrices 
$L^{<0>}$ and $M^{<0>}$ 
are $\elpmh$-equivalent.
\end{question}

There is} an affirmative answer to Question \ref{decisionquestion2}
for $G=\{e\}$ (see \cite{bs:decide}),
   and perhaps for all
   $G$.
  In some cases the invariants 
for SFTs (which we can regard as $G$-SFTs with $G=\{e\}$)
have allowed practical 
computation of examples and subclasses (see e.g. 
\cite{dh:fersft,mbdh:pbeim}). 
We note that 
if $G$ is abelian, then there
are only finitely many $G$-flow equivalence classes 
of $G$-SFTs defined by matrices $A$ 
for which the diagonal block determinants of $I-A$ are 
prescribed regular elements (i.e.,
are not zero-divisors) in $\Z G$ (Theorem \ref{nonzerodet}).

For the case $G=\{e\}$, there is 
 a complicated diagram of homomorphisms 
of finitely generated abelian groups 
(the {\it reduced  K-web} of 
 \cite{mbdh:pbeim}, useful for applications to 
Cuntz-Krieger algebras as explained in \cite{seagrer}} which 
(with regard to an appropriate notion of diagram isomorphism) 
is a complete invariant 
for the $\elph$-equivalence.
For general $G$, one can define a $K$-web invariant 
in the same way, using $\ZZ G$-modules and module homomorphisms 
in place of abelian groups and homomorphisms of abelian groups. 
However, this invariant is
no longer complete, because new obstructions arise to passing from
diagram isomorphism to the elementary matrix equivalence. 
(We thank Takeshi Katsura for showing us examples 
of this.) 
Developing a complete invariant 
from the  $\ZZ G$ $K$-web by characterizing the allowed diagram 
isomorphisms is a nontrivial but perhaps accessible
problem. 

We use the $\mathcal C_1$ condition 
in Theorem \ref{classification} 
to get a precise
characterization  in terms of matrix equivalence. 
To see that Theorem \ref{classification} 
would be false if the $\mathcal C_1$ condition were dropped, 
consider the following example: 
\begin{align} 
(I-A) \, U & = I-B \\ \notag 
\begin{pmatrix} 
1-g & -e & -e & 0 \\
0     & 0 & 0   & -e \\
0     & 0 & 0   & -e  \\
0     & 0 & 0  & 1-h  
\end{pmatrix} 
\begin{pmatrix} 
1     & 0 & 0 & 0 \\
0     & 2 & 1 & 0 \\
0     & 1 & 1 & 0  \\
0     & 0 & 0 & 1  
\end{pmatrix} 
&= 
\begin{pmatrix} 
1-g & -3e & -2e & 0 \\
0     & 0 & 0   & -e \\
0     & 0 & 0   & -e  \\
0     & 0 & 0  & 1-h  
\end{pmatrix}  \ . 
\end{align}
Here $e$ is the identity in $G$ ($e=1$ in $\Z G$) 
and $U\in \elph$, for a coset structure
$\mathcal H$ for $A$. But $A$ and $B$  present skewing functions 
on SFTs with 2 and 5 infinite orbits,
respectively,  so $T_A$ and $T_B$ are certainly not $G$-flow
equivalent.  

In the  case $G=\{e\}$ 
we
understand by  \cite[Theorem 3.3]{mb:fesftpf}
 exactly    when an 
$\elph$-equivalence is a positive  
$\elph$-equivalence (i.e., arises from a 
$G$-flow equivalence). We do not know how to do the same for nontrivial $G$,
if a block $A\{i,i\}$ defining a cycle component is allowed to be
  larger than $1\times 1$.
We need to require the $\mathcal C_1$ condition
to be able to prove in  
the Factorization Theorem \ref{theoremfactor} below, 
that every $\elph$-equivalence is a positive  
$\elph$-equivalence.

Partly because of complications arising from 
cycle components, we avoid the 
infinite matrices used in \cite{mb:fesftpf} and \cite{BSullivan}
to describe stabilizations.

\section{From $G$-flow equivalence to positive $\elph$ 
equivalence} \label{sec:1to2}

We will now present and prove Theorem \ref{necessary}, from which we directly get the implication $(1) \implies (2)$ in Theorem \ref{classification}.

First we introduce Condition $\mathcal C_{1+}$, which we shall use in the proof of Theorem \ref{necessary}.

\begin{definition} \label{conditionc1+}
Let $G$ be a finite group, let $\mathcal P=\{1,\dots,N\}$ be a finite poset given by a partial order relation $\preceq$ satisfying $i\preceq j\implies i\le j$, and let $\mathbf n = (n_1, \dots , n_N)$ be a vector of positive integers.
 
Given $\mathcal C \subset \mathcal P$, 
a matrix $ M$ in $\mathcal M_{\mathcal P}  (\mathbf n, \Z G) $  satisfies  
Condition $\mathcal C_{1+}$ if 
for every $i$ in $\mathcal C$ there exists $s_i$ in $\mathcal I_i$ 
 such that the following hold. 
\begin{enumerate} 
\item 
If $\{s,t\}\subset \mathcal I_i$ and $M(s,t)\neq 0$, 
then $(s,t)=(s_i,s_i)$. 
\item 
If $s\in \mathcal I_i$, with 
$M(s,t)\neq 0$ or $M(t,s)\neq 0$, then  
 $s= s_i$. 
\end{enumerate} 
\end{definition} 
Notice that if $M$ is a stabilization of a $\mathcal C_1$ matrix, then $M$ satisfies condition $\mathcal C_{1+}$.

\begin{theorem} \label{necessary} 
Let $G$ be a finite group, let $\mathcal P=\{1,\dots,N\}$ and $\mathcal P'=\{1,\dots,N'\}$ be finite posets given by partial order relations $\preceq$ satisfying $i\preceq j\implies i\le j$, let $\mathcal C$ and $\mathcal C'$ be subsets of $\mathcal P$ and $\mathcal P'$ respectively, and let $\mathcal H=\{H_{ij}\}_{i,j\in\mathcal P}$ and $\mathcal H'=\{H'_{ij}\}_{i,j\in\mathcal P'}$ be $(G,\mathcal P)$ and $(G,\mathcal P')$ coset structures.	
	
Suppose $A\in\mathcal{M}_{\mathcal{P}}^o(\mathcal{C},\mathbf n,\mathcal{H})$ and $B\in\mathcal{M}_{\mathcal{P'}}^o(\mathcal{C'},\mathbf n',\mathcal{H'})$ are stabilizations of nondegenerate matrices, and $T_A$ and $T_B$ are $G$-flow equivalent.
Then there exist a poset isomorphism $\alpha:\mathcal{P}\to\mathcal{P'}$ such that $\alpha(\mathcal{C})=\mathcal{C'}$ and  $\gamma=(\gamma_1,\dots,\gamma_N)\in G^N$ (where $N$ is the number of elements of $\mathcal{P}$) such that $H_{ij}=\gamma_i^{-1}H'_{\alpha(i)\alpha(j)}\gamma_j$ for $i,j\in\mathcal{P}$ with $i\preceq j$, and such that 
for the matrix $$C= (D^{\mathbf n}_{\gamma})^{-1} (Q^{\mathbf n}_{\alpha})^{-1} B Q^{\mathbf n}_{\alpha} D^{\mathbf n}_{\gamma}$$
the following holds: 
there exist 
 $\mathbf m$ and 
stabilizations $A^{<0>},C^{<0>}$ 
in $\mopcmh$ of $A,C$ 
with a positive $\elpmh$-equivalence  
$(U,V): \rep{I-A^{<0>}}  \to 
\rep{I-C^{<0>}}$. 

Moreover, if  $A$ and $B$ satisfy Condition 
$\mathcal C_1$, 
 then the matrices $A^{<0>}$, $C^{<0>}$ can be 
chosen to satisfy Condition $\mathcal C_1$. 
\end{theorem}

 Before we give the proof, which  is a nontrivial elaboration of 
the proof for the case that $A,B$ are essentially 
irreducible \cite [Proposition 4.7]{BSullivan}, we outline the argument.

A discrete cross-section for a homeomorphism $T:X\to X$ of a compact zero-dimensional metrizable space $X$ is a clopen subset $K\subset X$ such that every point of $X$ is mapped into $K$ by some positive power of $T$. In this case, for $x\in K$ there is a smallest positive integer $\rho_K (x)$ such that $T^{\rho_K(x)}$ is in $K$ and the return map $R_K: K\to K$ is then the map $ x\to T^{\rho_K(x)} (x)$ (see for example \cite{bce:fei} for details). If $T$ is an SFT, then $R_K$ is again SFT. If $K$ is a $G$-invariant discrete cross-section for $T_A$, then there is a (unique) discrete cross-section $C$ for $\sigma_{\overline{A}}$ such that $K=C\times G$ and $$R_K((x,g))=(R_C(x),g\tau_A(x)\tau_A(\sigma_{\overline{A}}(x))\dots\tau_A(\sigma_{\overline{A}}^{\rho_C(x)-1}(x)))$$ for $x\in C$ and $g\in G$.

The Parry-Sullivan argument \cite{parrysullivan} shows that 
any  flow equivalence of mapping tori of SFTs 
is isotopic to one which is 
induced by a conjugacy of return 
maps to discrete cross-sections (again, see for example \cite{bce:fei} for details).  
It follows that since $T_A$ and $T_B$ are $G$-flow equivalent, there exist $G$-invariant discrete cross-sections $K_A$ and $K_B$ for $T_A$ and $T_B$ such that the return maps $R_{K_A}$ and $R_{K_B}$ are $G$-conjugate. Let $C_A$ and $C_B$ be discrete cross-sections for $\sigma_{\overline{A}}$ and $\sigma_{\overline{B}}$ such that $K_A=C_A\times G$ and $K_B=C_B\times G$.

Our strategy is to first construct matrices $A^{<1>},A^{<2>}\in\mopch$ such that $A^{<2>}$ presents the $G$-SFT $R_{K_A}$, and such that there are  positive $\elph$-equivalences from $I- A$ to $I- A^{<1>}$ and from $I- A^{<1>}$ to $I- A^{<2>}$ (this is done in Step 1 and Step 2). Similarly, we get matrices $B^{<1>},B^{<2>}\in\mathcal{M}_{\mathcal{P'}}^o(\mathcal{C'},\mathcal{H'})$ such that there is a positive $\elphp$-equivalence $\rep{I-B}\to \rep{I-B^{<1>}}$ and a positive $\elphp$-equivalence $\rep{I-B^{<1>}}\to \rep{I-B^{<2>}}$, and such that $T_{B^{<2>}}$ is $G$-conjugate to $R_{K_B}$. Since $R_{K_A}$ and $R_{K_B}$ are $G$-conjugate, it follows that $T_{A^{<2>}}$ and $T_{B^{<2>}}$ are $G$-conjugate. We use this in Step 3 to construct matrices $A^{<3>}\in\mathcal{M}_{\mathcal{P}}^o(\mathcal{C}, \mathbf r,\mathcal{H})$, $B^{<3>}\in\mathcal{M}_{\mathcal{P'}}^o(\mathcal{C'}, \mathbf r',\mathcal{H'})$ and a poset isomorphism $\alpha:\mathcal{P}\to\mathcal{P'}$ such that $\alpha^*(\mathbf r)=\mathbf r'$ and $\alpha(\mathcal{C})=\mathcal{C'}$, and such that there is a positive $\elph$-equivalence from $I-A^{<2>}$ to $I-A^{<3>}$, a positive $\elphp$-equivalence from $I-B^{<2>}$ to $I-B^{<3>}$, and such that $\overline{A^{<3>}}=(Q^{\mathbf r}_\alpha)^{-1}\overline{B^{<3>}}Q^{\mathbf r}_\alpha$ and on this common domain the matrices $A^{<3>}$ and $(Q^{\mathbf r}_\alpha)^{-1}B^{<3>}Q^{\mathbf r}_\alpha$ define skewing functions which are cohomologous. In Step 4 we then find a vector $\gamma \in G^N$ such that $H_{ij}=\gamma_i^{-1}H'_{\alpha(i)\alpha(j)}\gamma_j$ for $i,j\in\mathcal{P}$ with $i\preceq j$, and such that there are positive $\elph$-equivalences
\begin{align*}
 & \rep{I-A^{<3>}} \ \to\ 
  \rep{I- (D_{\gamma}^{\mathbf r})^{-1}(Q^{\mathbf r}_\alpha)^{-1}B^{<3>}Q^{\mathbf r}_\alpha D_{\gamma}^{\mathbf r}} \ \ \ \text{and } \\
  &  \rep{I- (D_{\gamma}^{\mathbf r})^{-1}(Q^{\mathbf r}_\alpha)^{-1}B^{<3>}Q^{\mathbf r}_\alpha D_{\gamma}^{\mathbf r}} \ \to \ \rep{I-C} \ ,
  \end{align*}
where $C= (D^{\mathbf n}_{\gamma})^{-1} (Q^{\mathbf n}_{\alpha})^{-1} B Q^{\mathbf n}_{\alpha} D^{\mathbf n}_{\gamma}$. This completes the proof of the first half of the theorem. 

To show that the matrices $A^{<0>},C^{<0>}$ can be 
chosen to satisfy Condition $\mathcal C_1$ if $A$ and $B$ satisfy Condition 
$\mathcal C_1$, we refine the  construction of Steps 1--4 in order to obtain stabilizations $\widetilde A^{<0>}, \widetilde C^{<0>}\in\mathcal M^o_{\mathcal P}(\mathcal C,\mathcal H)$ of $A$ and $C$ and a positive $\textnormal{El}_{\mathcal P}(\mathcal  H)$-equivalence
\[
(\widetilde U,\widetilde V):( I-\widetilde A^{<0>})
\to ( I-\widetilde C^{<0>})
\]such that for every $i\in \mathcal C$ the matrices
$\widetilde U\{i,i\}$, $\widetilde V\{i,i\}$ are the identity matrix. We then get $\mathcal C_1$ stabilizations $A^{<0>},C^{<0>}$ in $\mopcmh$ of $A,C$ and with a positive $\elpmh$-equivalence
\[
(U,V): \rep{I-A^{<0>}}  \to \rep{I-C^{<0>}}
\]
as wanted by letting $A^{<0>}$, $C^{<0>}$, $U$, $W$ be principal submatrices of $\widetilde A^{<0>}$, $\widetilde C^{<0>}$, $\widetilde U$, $\widetilde V$. This is done in Steps 5--8.

\begin{proof}
{\bf Step 1: Higher block presentation.} We begin by constructing a higher block presentation $A^{<1>}$ of $T_A$ such that the discrete cross-section $C_A$ corresponds to a union of vertices in the graph $\mathcal G_{\overline{A^{<1>}}}$, and such that there is a positive $\elph$-equivalence from $I-A$ to $I-A^{<1>}$.

There is a $k$ and a subset $S$ of the $2k+1$-blocks of $X_{\overline{A}}$ such that $C_A=\{x\in X_{\overline{A}}:x[-k,k]\in S\}$. Let $P_{2k+1}$ be the set of paths in $\mathcal G_A$ of length $2k+1$. For $p\in P_{2k+1}$, let $s(p)$ be the initial vertex of the middle edge of $p$. Index the elements of $P_{2k+1}=\{p_1,p_2,\dots,p_n\}$ such that if $s(p_i)<s(p_j)$, then $i<j$. We will now construct an $n\times n$ matrix $A^{<1>}$ over $\Z_+ G$ (actually it will be a matrix over $G$). Let $1\le s,t\le n$. If there exist edges $e,f$ in $\mathcal G_{\overline{A}}$ such that $p_se=fp_t$, then the $(s,t)$ entry of $A^{<1>}$ is the label of the middle edge of $p_s$. If there are no edges $e,f$ in $\mathcal G_{\overline{A}}$ such that $p_se=fp_t$, then the $(s,t)$ entry of $A^{<1>}$ is 0. It follows from Proposition \ref{labellift} that $A^{<1>}\in\mopch$ and that there is a positive $\elph$-equivalence from $I-A$ to $I-A^{<1>}$. Since $A$ is a stabilization of a nondegenerate matrix, it follows that $A^{<1>}$ is nondegenerate.



{\bf Step 2: Discrete cross-section.} 

In this step we produce a nondegenerate matrix $A^{<2>}\in\mopch$ which 
presents the $G$-SFT $R_{K_A}$,  
and explain that there is a
positive $\elph$-equivalence from $I- A^{<1>}$ 
to $I- A^{<2>}$. 

The matrix $A^{<2>}$ is the adjacency matrix of the labelled graph which has vertex set $\{s\in\{1,2,\dots,n\}:p_s\in S\}$ and where there for each path $p$ in $\mathcal G_{A^{<1>}}$ which starts and ends in vertices $s$ and $t$ for which $p_s,p_t\in S$, but which otherwise go through vertices $v$ for which $p_v\notin S$, is an edge from $s$ to $t$ with label equal to the weight $\tau_{A^{<1>}}(p)$ of $p$ (so in particular, if $e$ is an edge in $\mathcal G_{A^{<1>}}$ which starts and ends in vertices $s$ and $t$ for which $p_s,p_t\in S$, then there is an edge in $\mathcal G_{A^{<2>}}$ from $s$ to $t$ with the same label as $e$). We then have that $T_{A^{<2>}}$ is $G$-conjugate to $R_{K_A}$. 

We will now construct a
positive $\elph$-equivalence from $I- A^{<1>}$ 
to $I- A^{<2>}$. This is accomplished by iterating a 
matrix 
move. Given 
a matrix $M$ 
in $\mopcmh$ 
 and a vertex $s$  such that $M(s,s)=0$, 
the move produces a positive $\elpmh$-equivalence 
 $\rep{I-M}\to \rep{I-M_s}$ 
for    a related matrix $M_{s}$
in $\mopcmh$, where $M_{s}$  has row $s$ and column $s$ zero.    

For a description of this move, 
let $M(r,s) =p\neq 0$, let $E_r$ be the basic elementary 
matrix $E_{r,s}(p)$. 
Let $U$ be the product of these $E_r$ 
 (so, $U(r,s)=M(r,s)$ if $r\neq s$, and in other entries 
$U=I$) 
and 
let  $M'$ be the matrix  such that 
$U(I-M)=I-M'$.  Then 
$(U,I): \rep{I-M} \to \rep{I-M'}$,  as a composition of the equivalences 
$(E_r,I)$,  is a positive
$\elpmh$  equivalence.  Column $s$ of $M'$  is zero. 
Row $s$ of $M'$ equals row $s$ of $M$. 

Next, we zero out row $s$ of $M'$. 
Define $V(s,t) = M(s,t) $ if $s\neq t$ and $V=I$ 
otherwise.  Let $M_{(s)}$
be the matrix such that $(I-M')V = I- M_{(s)}$. Then 
$(I,V): \rep{I-M'}\to \rep{I-M_{(s)}}$ is a positive $\elpmh$-equivalence. 
We have 
\begin{equation*} 
  M_{(s)} (r,t)  = 
  \begin{cases}
	  M(r,t) + M(r,s)M(s,t) &\text{ if }r\neq s\neq t \\
		0 &\text{ if }r=s \text{ or }s=t.
  \end{cases}
\end{equation*} 
The matrix $M_{(s)}$ presents a skewing function 
into $G$ induced by the return map 
to the clopen set of points $x$ for which the 
initial and terminal vertices of $x_0$ do not 
equal $s$.  

Altogether, for $\{1,\dots,n\}\setminus S=\{s(1), \dots ,s(k)\}$, 
set $M_0=A^{<1>}$, and set 
$M_{i} = (M_{i-1})_{(s(i))}$ for $1\leq i \leq k$. Then we have a positive $\elph$-equivalence from $I- A^{<1>}=1-M_0$ to $I-M_k=I- A^{<2>}$.

{\bf Step 3: The resolving tower and  matrix cohomology. } 
 
Similar to how we constructed $A^{<1>}$ and $A^{<2>}$, we can construct nondegenerate matrices $B^{<1>},B^{<2>}\in\mathcal{M}_{\mathcal{P'}}^o(\mathcal{C'},\mathcal{H'})$ such that there is a positive $\elphp$-equivalence $\rep{I-B}\to \rep{I-B^{<1>}}$ and a positive $\elphp$-equivalence $\rep{I-B^{<1>}}\to \rep{I-B^{<2>}}$, and such that $T_{B^{<2>}}$ is $G$-conjugate to $R_{K_B}$. Since $R_{K_A}$ and $R_{K_B}$ are $G$-conjugate, it follows that $T_{A^{<2>}}$ and $T_{B^{<2>}}$ are $G$-conjugate. It therefore follows from
Proposition \ref{groupexfacts} that there is a topological 
conjugacy 
\[
\phi : X_{\overline{A^{<2>}} } \to X_{\overline{B^{<2>}} }
\]
which takes the skewing function $\tau_{A^{<2>}}$ to 
a function cohomologous to $\tau_{B^{<2>}}$. 
In this step we will construct matrices $A^{<3>}\in\mathcal{M}_{\mathcal{P}}^o(\mathcal{C}, \mathbf r,\mathcal{H})$, $B^{<3>}\in\mathcal{M}_{\mathcal{P'}}^o(\mathcal{C'}, \mathbf r',\mathcal{H'})$ and a poset isomorphism $\alpha:\mathcal{P}\to\mathcal{P'}$ such that $\alpha^*(\mathbf r)=\mathbf r'$ and $\alpha(\mathcal{C})=\mathcal{C'}$, and such that there is a positive $\elphp$-equivalence from $I-A^{<2>}$ to $I-A^{<3>}$, a positive $\elphp$-equivalence from $I-B^{<2>}$ to $I-B^{<3>}$, and such that $\overline{A^{<3>}}=(Q^{\mathbf r}_\alpha)^{-1}\overline{B^{<3>}}Q^{\mathbf r}_\alpha$ and on this common domain the matrices $A^{<3>}$ and $(Q^{\mathbf r}_\alpha)^{-1}B^{<3>}Q^{\mathbf r}_\alpha$ define skewing functions which are cohomologous.

There is a standard decomposition for 
the  topological conjugacy $\phi$ \cite{parrynotes} (see also \cite[Theorem 7.1.2]{dlbm:isdc}). 
It follows from this that there is a matrix $M$ over $\Z_+$, with  
one-block conjugacies (given by graph 
homomorphisms)  $\phi_1: \sigma_M\to \sigma_{\overline{A^{<2>}}}$ 
and $\phi_2: \sigma_M \to \sigma_{\overline{B^{<2>}}}$ such that 
\begin{enumerate} 
\item 
$\phi$ is $\phi_1^{-1}$ followed by 
$\phi_2$, 
\item 
$\phi_1$ is left resolving, 
with  $\phi_1^{-1} $ a composition of 
conjugacies given by row splittings, 
\item  
$\phi_2$ is right resolving, with 
$\phi_2^{-1} $  a composition of 
conjugacies given by column splittings. 
\end{enumerate} 
Since $\phi_1^{-1} $ a composition of conjugacies given by row splittings, it follows that there is a $\Z G$-matrix $A^{<3>}$ such that  $\phi_1^{-1} $ can be lifted to a $G$-conjugacy $\psi_A:T_{A^{<2>}}\to T_{A^{<3>}}$, also given by row splittings, and a permutation matrix $P_A$ such that $P_A^{-1}\overline{A^{<3>}}P_A=M$ and $\psi_A((x,g))=(\eta_{P_A}(\phi_1^{-1}(x)),g)$ for $(x,g)\in X_{\overline{A^{<2>}}}\times G$, where $\eta_{P_A}:\sigma_M\to \sigma_{\overline{A^{<3>}}}$ is the conjugacy given by $P_A$. It follows from Proposition \ref{labellift} that $A^{<3>}\in\mopch$ and that there is a positive $\elph$-equivalence from $I-A^{<2>}$ to $I-A^{<3>}$. Since $A^{<2>}$ is nondegenerate, so is $A^{<3>}$. Choose $\mathbf r$ such that $A^{<3>}\in\mathcal{M}_{\mathcal{P}}^o(\mathcal{C}, \mathbf r,\mathcal{H})$.

Similarly, there is a vector $\mathbf r'$, a nondegenerate $\Z G$-matrix $B^{<2.5>}\in\mathcal{M}_{\mathcal{P'}}^o(\mathcal{C'}, \mathbf r',\mathcal{H'})$, a $G$-conjugacy $\psi_B:T_{B^{<2>}}\to T_{B^{<2.5>}}$, a permutation matrix $P_B$ such that $P_B^{-1}\overline{B^{<2.5>}}P_B=M$ and $\psi_B((x,g))=(\eta_{P_B}(\phi_2^{-1}(x)),g)$ for $(x,g)\in X_{\overline{B^{<2>}}}\times G$, where $\eta_{P_B}:\sigma_M\to \sigma_{\overline{B^{<2.5>}}}$ is the conjugacy given by $P_B$, and a positive $\textnormal{El}_{\mathcal P'}(\mathcal  H')$-equivalence from $I-B^{<2>}$ to $I-B^{<2.5>}$. Let $P=P_BP_A^{-1}$. Then $\overline{A^{<3>}}=P^{-1}\overline{B^{<2.5>}}P$. It follows that there is a poset isomorphism $\alpha:\mathcal{P}\to\mathcal{P'}$ such that $\alpha^*(\mathbf r)=\mathbf r'$ and $\alpha(\mathcal{C})=\mathcal{C'}$, and a permutation matrix $P'\in\mathcal{M}_{\mathcal{P'}}(\mathbf r',\Z_+)$ such that $P=P'Q^{\mathbf r'}_\alpha$. Let $B^{<3>}=(P')^{-1}B^{<2.5>}P'\in\mathcal{M}_{\mathcal{P'}}^o(\mathcal{C'}, \mathbf r',\mathcal{H'})$. It then follows from Proposition \ref{permutationprop} that there is a positive $\textnormal{El}_{\mathcal P'}(\mathbf r',\mathcal  H')$-equivalence $\rep{I-B^{<2.5>}}\to \rep{I-B^{<3>}}$. We furthermore have that
$(Q^{\mathbf r}_\alpha)^{-1}B^{<3>}Q^{\mathbf r}_\alpha\in\mathcal{M}_{\mathcal{P}}^o(\mathcal{C}, \mathbf r,\mathcal{H})$, $(Q^{\mathbf r}_\alpha)^{-1}\overline{B^{<3>}}Q^{\mathbf r}_\alpha=\overline{A^{<3>}}$, and $\tau_{(Q^{\mathbf r}_\alpha)^{-1}B^{<3>}Q^{\mathbf r}_\alpha}$ and $\tau_{A^{<3>}}$ are cohomologous.

{\bf Step 4: $\elph$-equivalence. } 

In this step we complete the proof apart from 
the (nontrivial) ``moreover'' statement. 
We continue with the notation of the last step. 

Let $\psi$ be the continuous function from 
$X_{\overline{A^{<3>}}}$ into $G$ such that 
\[
\tau_{A^{<3>}}(x)= (\psi(x))^{-1}\tau_{(Q^{\mathbf r}_\alpha)^{-1}B^{<3>}Q^{\mathbf r}_\alpha}(x)\psi(\sigma_{\overline{A^{<3>}}}(x) )
\]
for all $x\in X_{\overline{A^{<3>}}}$. 
Then 
proof of Parry for 
\cite[Lemma 9.1]{parrylivsic}  
(translated from his vertex 
SFTs to our edge SFTs) 
shows that if $x\in X_{\overline{A^{<3>}}}$, 
then $\psi (x)$ is determined by the initial vertex of the edge $x_0$. 
Because $A^{<3>}$ and $(Q^{\mathbf r}_\alpha)^{-1}B^{<3>}Q^{\mathbf r}_\alpha$ are nondegenerate, this implies that there is  
%
a diagonal matrix $D$, with each diagonal 
element an element of $G$, such that 
\begin{equation}\label{parrydiagonal} 
D^{-1}A^{<3>}D = (Q^{\mathbf r}_\alpha)^{-1}B^{<3>}Q^{\mathbf r}_\alpha.
\end{equation}

We let $M_i$ denote a diagonal 
block $M\{i , i\}$ of a $\mathcal P$-blocked matrix. 
The matrices $A^{<3>}_i$ and 
$((Q^{\mathbf r}_\alpha)^{-1}B^{<3>}Q^{\mathbf r}_\alpha)_i$ are essentially irreducible and 
the group $H_i$ is a weights 
group for $A^{<3>}_i$ and 
$((Q^{\mathbf r}_\alpha)^{-1}B^{<3>}Q^{\mathbf r}_\alpha)_i$. 
 It then follows from the proof of 
Theorem 4.7 of \cite{BSullivan}  that 
there exists $\gamma \in G^N$
 such that for each $i$ the 
 diagonal matrix 
$(DD_\gamma^{\mathbf r})_i$ has every entry in $H_i$. 
Then 
\begin{align*}
(D_{\gamma}^{\mathbf r})^{-1}(Q^{\mathbf m'}_\alpha)^{-1}B^{<3>}Q^{\mathbf m'}_\alpha D_{\gamma}^{\mathbf r} &=   
(D_{\gamma}^{\mathbf r})^{-1}(D^{-1}A^{<3>}D) D_{\gamma}^{\mathbf r} \\ &=  (D D_{\gamma}^{\mathbf r} )^{-1} A^{<3>}(D D_{\gamma}^{\mathbf r} ).
\end{align*} 
Applying Proposition \ref{cohomologyaspositive},  
we have $(D_{\gamma}^{\mathbf r})^{-1}(Q^{\mathbf r}_\alpha)^{-1}B^{<3>}Q^{\mathbf r}_\alpha D_{\gamma}^{\mathbf r} \in \moph$,
with  a positive $\textnormal{El}_{\mathcal P}(\mathcal  H)$-equivalence 
$$\rep{I-A^{<3>}} 
\to \rep{I- (D_{\gamma}^{\mathbf r})^{-1}(Q^{\mathbf r}_\alpha)^{-1}B^{<3>}Q^{\mathbf r}_\alpha D_{\gamma}^{\mathbf r}}.$$   
So, $H_{ij}=\gamma_i^{-1}H'_{\alpha(i)\alpha(j)}\gamma_j$ for $i,j\in\mathcal{P}$ with $i\preceq j$, and we have positive 
$\elph$-equivalences 
\[
\rep{I-A}\ \to\  \rep{I-A^{<3>}}\ \to \ 
\rep{I- (D_{\gamma}^{\mathbf r})^{-1}(Q^{\mathbf r}_\alpha)^{-1}B^{<3>}Q^{\mathbf r}_\alpha D_{\gamma}^{\mathbf r}}. 
\] 
Let $C= (D^{\mathbf n}_{\gamma})^{-1} (Q^{\mathbf n}_{\alpha})^{-1} B Q^{\mathbf n}_{\alpha} D^{\mathbf n}_{\gamma}$. Then $C\in\moph$ because $B\in\mathcal{M}_{\mathcal{P'}}^o(\mathcal{C'},\mathcal{H'})$, $\alpha$ is a poset isomorphism from $\mathcal P$ to $\mathcal P'$ such that $\alpha(\mathcal{C})=\mathcal{C'}$ and $H_{ij}=\gamma_i^{-1}H'_{\alpha(i)\alpha(j)}\gamma_j$ for $i,j\in\mathcal{P}$ with $i\preceq j$.

It remains 
to show that there is a positive 
$\elph$-equivalence 
from $I-C$ to $I- (D_{\gamma}^{\mathbf r})^{-1}(Q^{\mathbf r}_\alpha)^{-1}B^{<3>}Q^{\mathbf r}_\alpha D_{\gamma}^{\mathbf r}$.  
We have proved that there is
a positive $\textnormal{El}_{\mathcal P'}(\mathcal  H')$-equivalence from $I-B$ to 
$I-B^{<3>}$  given by a path 
\[
\rep{I-B'}\xrightarrow{(E_1,F_1)} \cdot 
\xrightarrow{(E_2,F_2)} \cdots  
\xrightarrow{(E_T,F_T)} 
\rep{I-(B^{<3>})'} \  
\]
in which $B'\in\mathcal M^o_{\mathcal P'}(\mathbf m',\mathcal H')$ is a stabilization of $B$, $(B^{<3>})'\in\mathcal M^o_{\mathcal P'}(\mathbf m',\mathcal H')$ is a stabilization of $B^{<3>}$, and the $E_t$ and $F_t$ are basic elementary 
matrices in $\textnormal{El}_{\mathcal P'}(\mathbf m',\mathcal  H')$. But then 
$E_t':=(D_{\gamma}^{\mathbf p'})^{-1}(Q^{\mathbf p'}_\alpha)^{-1}E_tQ^{\mathbf p'}_\alpha D_{\gamma}^{\mathbf p'}$ and $F_t':=(D_{\gamma}^{\mathbf p'})^{-1}(Q^{\mathbf p'}_\alpha)^{-1}F_tQ^{\mathbf p'}_\alpha D_{\gamma}^{\mathbf p'}$ are  
basic elementary 
matrices in $\textnormal{El}_{\mathcal P}(\mathcal  H)$, and  we have a positive $\textnormal{El}_{\mathcal P}(\mathcal  H)$-equivalence
\begin{multline*}
\rep{I-(D_{\gamma}^{\mathbf p'})^{-1}(Q^{\mathbf p'}_\alpha)^{-1}B'Q^{\mathbf p'}_\alpha D_{\gamma}^{\mathbf p'}}\xrightarrow{(E'_1,F'_1)} \cdot 
\xrightarrow{(E'_2,F'_2)}\\ \cdots  
\xrightarrow{(E'_T,F'_T)} 
\rep{I-(D_{\gamma}^{\mathbf p'})^{-1}(Q^{\mathbf p'}_\alpha)^{-1}(B^{<3>})'Q^{\mathbf p'}_\alpha D_{\gamma}^{\mathbf p'}}.  
\end{multline*}
Since $(D_{\gamma}^{\mathbf p'})^{-1}(Q^{\mathbf p'}_\alpha)^{-1}B'Q^{\mathbf p'}_\alpha D_{\gamma}^{\mathbf p'}$ is a stabilization of $C$ and $$(D_{\gamma}^{\mathbf p'})^{-1}(Q^{\mathbf p'}_\alpha)^{-1}(B^{<3>})'Q^{\mathbf p'}_\alpha D_{\gamma}^{\mathbf p'}$$ is a stabilization of $(D_{\gamma}^{\mathbf r})^{-1}(Q^{\mathbf r}_\alpha)^{-1}B^{<3>}Q^{\mathbf r}_\alpha D_{\gamma}^{\mathbf r}$, this shows that there is a positive 
$\elph$-equivalence 
from $I-C$ to $I- (D_{\gamma}^{\mathbf r})^{-1}(Q^{\mathbf r}_\alpha)^{-1}B^{<3>}Q^{\mathbf r}_\alpha D_{\gamma}^{\mathbf r}$ and thus that there is a positive 
$\elph$-equivalence 
from $I-A$ to $I-C$.

{\bf ``Moreover''.} 
For the rest of the proof, we assume that $A$ and $B$ satisfy Condition 
$\mathcal C_1$ . It remains to show that we can find stabilizations $A^{<0>},C^{<0>}\in\mathcal M^o_{\mathcal P}(\mathcal C,\mathbf m,\mathcal H)$ of $A$, $C$ with a positive $\elpmh$-equivalence  
$(U,V): \rep{I-A^{<0>}}  \to \rep{I-C^{<0>}}$. 

For this we refine the  construction of Steps 1-4 in order to obtain stabilizations $\widetilde A^{<0>}, \widetilde C^{<0>}\in\mathcal M^o_{\mathcal P}(\mathcal C,\mathbf t,\mathcal H)$ of $A$ and $C$ and a positive $\textnormal{El}_{\mathcal P}(\mathbf t,\mathcal  H)$-equivalence  $(\widetilde U,\widetilde V):( I-\widetilde A^{<0>}) \to ( I-\widetilde C^{<0>})$ such that for every $i\in \mathcal C$ the matrices
$\widetilde U\{i,i\}$, $\widetilde V\{i,i\}$ are the identity matrix. We then get $\mathcal C_1$ stabilizations $A^{<0>},C^{<0>}$ in $\mopcmh$ of $A,C$ and with a positive $\elpmh$-equivalence $(U,V): \rep{I-A^{<0>}}  \to \rep{I-C^{<0>}}$ as wanted by letting $A^{<0>}$, $C^{<0>}$, $U$, $W$ be principal submatrices of $\widetilde A^{<0>}$, $\widetilde C^{<0>}$, $\widetilde U$, $\widetilde V$.

{\bf Step 5: Getting a $\mathcal C_u$-equivalence 
$\rep{I-A^{<0>}}\to \rep{I-A^{<3>}}$.}

In this step we will show that the positive equivalence 
$I-A \to I-A^{<3>}$ of Steps 1-3 can be chosen  
such that 
there are stabilizations $A''',{A^{<3>}}'''\in\mathcal M^o_{\mathcal P}(\mathcal C,\mathbf s,\mathcal H)$ of $A$ and $A^{<3>}$ and a positive 
$\textnormal{El}_{\mathcal P}(\mathbf s,\mathcal  H)$-equivalence  $(U_A,V_A):( I-A'''') \to ( I-{A^{<3>}}'''')$ such that for every $i\in \mathcal C$ the matrices
$U_A\{i,i\}$, $V_A\{i,i\}$ are unipotent 
upper triangular. 

Recall that the positive $\elph$-equivalence from $I-A$ to $I-A^{<1>}$ is the composition of positive $\elph$-equivalences $\rep{I-A_t}\to \rep{I-A_{t+1}}$ obtained by applying Proposition \ref{labellift}.
At each stage the $\mathcal P$ blocking of 
$A_{t+1}$ is the lift of the $\mathcal P$ blocking 
of $A_t$. 
At step $t$, there is an index 
$s_t$ such that either 
row $s_t$ of $A_t$ is split into two rows, or column $s_t$ is 
split into two columns.  
We will choose $s_t, 1+s_t$ to be the indices associated 
to the splitting (so, if an index $j$ of $A_t$ is greater than 
$s$, then it corresponds to index $j+1$ of $A_{t+1}$).  
In the case that $s_t$ is an index in a cycle component 
we place additional conditions as follows.  
If $A_t(s_t,s_t) =g\neq 0$,  then we require that 
\[ 
A_{t+1}(s_t +1,s_t+1) = g\quad \text{ and } A_{t+1}(s_t ,s_t) = 0 
\] 
in the case $A_t \to A_{t+1}$ is a column splitting, and 
we require 
\[ 
A_{t+1}(s_t +1,s_t+1) = 0\quad \text{ and } A_{t+1}(s_t ,s_t) = g 
\] 
in the case $A_t \to A_{t+1}$ is a row splitting.

When $A_t$ is upper triangular in its cycle component 
diagonal blocks, it follows 
from Proposition \ref{labellift} that $A_{t+1}$ is as well, 
and that there are unipotent 
upper triangular matrices $U_t,V_t$ such that
\[
(U_t,V_t): 
\rep{I-A_t}\to \rep{I-A_{t+1}}
\]
is a positive $\elph$-equivalence.  
By induction, each cycle component block of 
$A^{<1>}$ is upper triangular, and there is a positive $\elph$-equivalence
\[
(U_1,V_1): \rep{I-A'}\to \rep{I-{A^{<1>}}'}
\]
where $A'$ is a stabilization of $A$, ${A^{<1>}}'$ is a stabilization of $A^{<1>}$, and for every $i\in \mathcal C$ the matrices
$U_1\{i,i\}$, $V_1\{i,i\}$ are unipotent 
upper triangular. 

Next consider the restriction of the Step 2 move $M\to M_{(s)}$ 
to a diagonal block $M\{ i,i\}$ with $i\in \mathcal C$.  
Clearly, if  $M\{ i,i\}$ 
is upper triangular, then 
the trimming matrices 
implementing the positive equivalence $\rep{I-M }\to \rep{I-M_s}$
are unipotent upper triangular, and 
$M_{(s)}\{ i,i\}$ is upper triangular.  
Because $A^{<1>}\{ i, i\}$ is  
 upper triangular, it follows by induction that 
$A^{<2>}\{ i, i\}$ is  
 upper triangular and that there is a positive $\elph$-equivalence
 \[
 (U_2,V_2): \rep{I-{A^{<1>}}''}\to \rep{I-{A^{<2>}}''}
 \]
 where ${A^{<1>}}''$ is a stabilization of $A^{<1>}$, ${A^{<2>}}''$ is a stabilization of $A^{<2>}$, and for every $i\in \mathcal C$ the matrices
$U_2\{i,i\}$, $V_2\{i,i\}$ are unipotent 
upper triangular. 
By the argument for Step 1, 
we will likewise have that 
$A^{<3>}$ upper triangular in each cycle component block and that there is a positive $\elph$-equivalence
\[
(U_3,V_3): \rep{I-{A^{<2>}}'''}\to \rep{I-{A^{<3>}}'''}
\]
where ${A^{<2>}}'''$ is a stabilization of $A^{<2>}$, ${A^{<3>}}'''$ is a stabilization of $A^{<3>}$, and for every $i\in \mathcal C$ the matrices
$U_3\{i,i\}$, $V_3\{i,i\}$ are unipotent 
upper triangular. Putting this together we get there is a 
stabilization $A''''\in\mathcal{M}_{\mathcal{P}}^o(\mathcal{C}, \mathbf s,\mathcal{H})$ of $A$ and a stabilization ${A^{<3>}}''''\in\mathcal{M}_{\mathcal{P}}^o(\mathcal{C}, \mathbf s,\mathcal{H})$ of $A^{<3>}$ such that there is a positive 
$\textnormal{El}_{\mathcal P}(\mathbf s,\mathcal  H)$-equivalence
\begin{equation}\label{steppingup}
(U_A,V_A):( I-A'''') \to ( I-{A^{<3>}}'''')
\end{equation} 
such that for every $i\in \mathcal C$ the matrices
$U_A\{i,i\}$, $V_A\{i,i\}$ are unipotent 
upper triangular. 


By the same argument, 
there are
stabilizations $B''''$, ${B^{<2.5>}}''''\in\mathcal{M}_{\mathcal{P'}}^o(\mathcal{C'}, \mathbf s',\mathcal{H}')$ of $B,B^{<2.5>}$ and a positive 
$\textnormal{El}_{\mathcal P'}(\mathbf s',\mathcal  H')$-equivalence  $$(U_B,V_B):( I-B'''') \to ( I-{B^{<2.5>}}'''')$$ such that for every $i\in \mathcal C'$ the matrices
$U_B\{i,i\}$, $V_B\{i,i\}$ are unipotent 
upper triangular.

{\bf Step 6: Clearing out diagonal $\mathcal C$ blocks in $(U,V)$.}

From Step 5 we have 
a stabilizations $A'''',{A^{<3>}}''''\in\mathcal{M}_{\mathcal{P}}^o(\mathcal{C}, \mathbf s,\mathcal{H})$ of $A,A^{<3>}$ and
the positive 
$\textnormal{El}_{\mathcal P}(\mathbf s,\mathcal  H)$-equivalence
\eqref{steppingup}
such that for every $i\in \mathcal C$ the matrices
$U_A\{i,i\}$, $V_A\{i,i\}$ are unipotent 
upper triangular. In this step we will 
show there is a positive 
$\textnormal{El}_{\mathcal P}(\mathbf s,\mathcal  H)$-equivalence 
$\rep{I-A''''} \to \rep{I-A''''}$ such that precomposing $(U,V)$ with this 
equivalence produces an equivalence
\[
(\widetilde U_A, \widetilde V_A):( I-A'''') \to ( I-{A^{<3>}}'''')
\]
such 
that for all $i$ in $\mathcal C$, the diagonal blocks 
$\widetilde U_A\{i,i\}$ and $\widetilde V_A\{i,i\}$ are the identity matrix. 

So, consider $i\in \mathcal C$. Restricted to the block 
$\rep{I-A''''}\{i,i\}:=\rep{I-M}$, our equivalence $U_A(I-A'''')V_A =\rep{I-{A^{<3>}}''''}$  has
the following  block 
triangular form, with central block $1\times 1$:  
\[
\bpmat 
U_{11} & U_{12} & U_{13} \\ 
0       & U_{22} & U_{23} \\ 
0       & 0        & U_{33}  
\epmat
\bpmat 
I       & 0         & 0 \\ 
0       & 1-g    & 0 \\ 
0       & 0        & I
\epmat
\bpmat 
V_{11} & V_{12} & V_{13} \\  
0       & V_{22} & V_{23} \\ 
0       & 0        & V_{33}  
\epmat 
= 
\bpmat 
I       & 0         & 0 \\  
0       & 1-g    & 0 \\ 
0       & 0        & I
\epmat \ .
\] 
The form is determined by placing the unique entry $1-g$ 
as the central block. 
Suppose $\{ s,t \}\subset \mathcal I_i$, $s<t$, 
$s\neq s_i\neq t$ and $E$ is a basic elementary 
matrix of size matching $A''''$ 
with $E(s,t)= \pm h$ 
for some $h$ in $G$. 
(For $E$ in $\elph$, $h$ must be $g^k$ for 
some $k$.)   Then, because $A$ satisfies condition 
$\mathcal C_1$, $(E,E^{-1}): \rep{I-A'''' }\to \rep{I-A'''' }$ is a 
positive $\elpnh$-equivalence. After precomposing 
$(U,V)$ with a suitable composition of these, 
we may assume $U_{11}=I$, $U_{33}=I$ and $U_{13}=0$.
Our matrix equivalence now has the following form 
\begin{align} \label{leftequiv} 
\bpmat 
I      & x        & 0 \\ 
0       & 1       & U_{23} \\ 
0       & 0        & I  
\epmat
\bpmat 
I       & 0         & 0 \\ 
0       & 1-g    & 0 \\ 
0       & 0        & I
\epmat
\bpmat 
V_{11} & V_{12} & V_{13} \\  
0       & 1       & y \\ 
0       & 0        & V_{33}  
\epmat 
&= 
\bpmat 
I       & 0         & 0 \\  
0       & 1-g    & 0 \\ 
0       & 0        & I
\epmat 
\end{align}
which multiplies out to give 
\[
\bpmat 
V_{11} &  V_{12}+x(1-g) & V_{13}+x(1-g)y \\  
0       & 1-g       & (1-g)y + U_{23}V_{33}\\ 
0       & 0        & V_{33}  
\epmat 
= 
\bpmat 
I       & 0         & 0 \\  
0       & 1-g    & 0 \\ 
0       & 0        & I
\epmat  \ .
\]
Consequently we can rewrite the left side of \eqref{leftequiv} as 
\begin{align*}
& \bpmat 
I      & x        & 0 \\ 
0       & 1       & (g-1)y \\ 
0       & 0        & I  
\epmat
\bpmat 
I       & 0         & 0 \\ 
0       & 1-g    & 0 \\ 
0       & 0        & I
\epmat
\bpmat 
I & x(g-1) & x(g-1)y\\  
0       & 1       & y \\ 
0       & 0        & I
\epmat  \ .
\end{align*}
This equivalence $(U,V): \rep{I-M}\to \rep{I-M}$ is a composition
of two equivalences,  $(U_1,V_1)$ followed by 
$(U_2,V_2)$, where 
%
\begin{align*} 
U_1 & = 
\bpmat 
I        & x   & 0\\  
0       & 1       &0  \\ 
0       & 0        & I
\epmat 
\quad \quad \quad \quad \quad 
V_1  = 
\bpmat 
I & x(g-1) & 0\\  
0       & 1       & 0 \\ 
0       & 0        & I
\epmat 
\\ 
U_2& = 
\bpmat 
I & 0& 0\\  
0       & 1       & (g-1)y \\ 
0       & 0        & I
\epmat 
\quad \quad 
V_2 = 
\bpmat 
I        & 0       & 0\\  
0       & 1       & y \\ 
0       & 0        & I
\epmat
 \ .
\end{align*} 
We will see how  these equivalences are related 
to 
positive equivalences $\rep{I-M}\to \rep{I-M}$. 

Consider a term 
$-h$ ($h\in G$) which is part of an entry of $y$
in $V_2$, say  the 
$(s_i,t)$ entry of $V$. Recall $E_{s,t}(\delta )$ denotes 
a basic elementary matrix with off-diagonal entry 
$\delta$ in position $(s,t)$. We define now 
$n\times n$ matrices $E_1,\dots , E_4$. 
$E_1(r,t) = -M(r,s_i)h$ if $r\notin \{s_i,t\}$;
 in other entries, $E_1=I$.
$E_2=E_{s_i,t}(-gh)$; $E_3=E_{s_i,t}(-h)$; 
$E_4= E_{s_i,t}(h)$. Then 
$ E_4  E_2E_1(I-M)E_3 = \rep{I-M}$,  
and applying the multiplications in the order 
indexed gives a positive $\elph$-equivalence. 
It may be easiest to see the argument by 
restricting to a $4\times 4$ principal submatrix. 
For this, suppose $2=s_i,3=t$ and $M(1,2)\neq 0 \neq M(2,4)$. 
Then these principal submatrices (with names 
unchanged for simplicity) have the forms 
\begin{align*} 
I-M &= 
\bpmat 
1-x & -w& 0 &-u \\ 
0 & 1-g & 0 & -z \\ 
0&0&       1  & 0 \\ 
0&0&0&1 
\epmat 
\quad 
E_1 = 
\bpmat 
1 & 0& -wh &0 \\ 
0 & 1 & 0 &0 \\ 
0&0&    1  & 0 \\ 
0&0&0&1 
\epmat 
\\
E_2 &= 
\bpmat 
1 & 0& 0&0 \\ 
0 & 1 & -gh &0 \\ 
0&0&    1  & 0 \\ 
0&0&0&1 
\epmat 
\quad  
E_3 = 
\bpmat 
1 & 0& 0&0 \\ 
0 & 1 & -h &0 \\ 
0&0&    1  & 0 \\ 
0&0&0&1 
\epmat 
=(E_4)^{-1}\ .
\end{align*}  
For the $n\times n$ matrices, we have 
$(E_4E_2E_1,E_3) =(E_4E_{s_i,t}((g-1)h), E_{s_i,t}(-h))$. 

For the case the term is $h$, there is similarly a positive 
equivalence $    F_4 F_3F_1(I-M)F_2 =\rep{I-M}$, 
in which $F_1=E_3$, $F_2=(E_3)^{-1}$, $F_3=E_2^{-1}$ and 
$F_4=E_1^{-1}$. In the $4\times 4$ sample, this 
has the form 
\begin{align*} 
F_1 & = 
\bpmat 
1 & 0& 0&0 \\ 
0 & 1 & -h &0 \\ 
0&0&    1  & 0 \\ 
0&0&0&1 
\epmat =F_2^{-1}
\\ 
F_3 &= 
\bpmat 
1 & 0& 0&0 \\ 
0 & 1 & gh &0 \\ 
0&0&    1  & 0 \\ 
0&0&0&1 
\epmat 
\quad  \ \,
F_4=
\bpmat 
1 & 0& wh &0 \\ 
0 & 1 & 0 &0 \\ 
0&0&    1  & 0 \\ 
0&0&0&1 
\epmat 
\end{align*} 
Here $(  F_4 F_3F_1, F_2 ) =(F_4E_{s_i,t}((g-1)h), E_{s_i,t}(h))$.

A suitable  composition of the 
above equivalences is an equivalence 
which in the $\{ i,i\}$ block 
matches $(U_1,V_1)$. 
Precomposing $(U_A,V_A)$ with the inverse of this 
composition gives the required matrix 
$(\widetilde U_A, \widetilde V_A)$. 

Similarly, there is a positive 
$\textnormal{El}_{\mathcal P'}(\mathbf s',\mathcal  H')$-equivalence
\[
(\widetilde U_B,\widetilde V_B):( I-B'''') \to ( I-{B^{<2.5>}}'''')
\]
such that for every $i\in \mathcal C'$ the matrices
$\widetilde U_B\{i,i\}$, $\widetilde V_B\{i,i\}$ are the identity matrix.

{\bf Step 7: Cohomology.}



In this step we will show that there are stabilizations $B^{<2.5>}_1,B^{<3>}_1\in\mathcal{M}_{\mathcal{P'}}^o(\mathcal{C'}, \mathbf s'_1,\mathcal{H'})$ of $B^{<2.5>}$, $B^{<3>}$ and a positive 
$\textnormal{El}_{\mathcal P'}(\mathbf s_1',\mathcal  H')$-equivalence
\[
(\widetilde U_1',\widetilde V_1'):( I-B^{<2.5>}_1) \to ( I-B^{<3>}_1)
\]
such that for every $i\in \mathcal C'$ the matrices
$\widetilde U_1'\{i,i\}$, $\widetilde V_1'\{i,i\}$ are the identity matrix, and we will show that there are stabilizations $A^{<3>}_1, M_1\in\mathcal{M}_{\mathcal{P}}^o(\mathcal{C}, \mathbf s_1,\mathcal{H})$ of $A^{<3>}$, $(D_{\gamma}^{\mathbf r})^{-1}(Q^{\mathbf r}_\alpha)^{-1}B^{<3>}Q^{\mathbf r}_\alpha D_{\gamma}^{\mathbf r}$ and a positive $\textnormal{El}_{\mathcal P}(\mathbf s_1,\mathcal  H)$-equivalence
\[
(\widetilde U_1,\widetilde V_1):( I-A^{<3>}_1) \to ( I-M_1)
\]
such that for every $i\in \mathcal C$ the matrices
$\widetilde U_1\{i,i\}$, $\widetilde V_1\{i,i\}$ are the identity matrix.

Recall that $B^{<3>}=(P')^{-1}B^{<2.5>}P'$ where $P'\in\mathcal{M}_{\mathcal{P'}}(\mathbf r',\Z_+)$ is a permutation matrix. 
Since $B$ satisfies condition $\mathcal C_1$,
\[
(\widetilde U_B,\widetilde V_B):( I-B'''') \to ( I-{B^{<2.5>}}'''')
\]
is a positive 
$\textnormal{El}_{\mathcal P'}(\mathbf s',\mathcal  H')$-equivalence such that for every $i\in \mathcal C'$ the matrices
$\widetilde U_B\{i,i\}$, $\widetilde V_B\{i,i\}$ are the identity matrix, and $B^{<2.5>}$ is nondegenerate, it follows that $B^{<2.5>}$ satisfies condition $\mathcal C_1$, and thus that $P'\{i,i\}=1$ for every $i\in \mathcal C'$. It therefore follows from Proposition \ref{permutationprop} that there are stabilizations $B^{<2.5>}_1$, $B^{<3>}_1$ of $B^{<2.5>}$, $B^{<3>}$ and a positive 
$\textnormal{El}_{\mathcal P'}(\mathbf s_1',\mathcal  H')$-equivalence
\[
(\widetilde U_1',\widetilde V_1'):( I-B^{<2.5>}_1) \to ( I-B^{<3>}_1)
\]
such that for every $i\in \mathcal C'$ the matrices
$\widetilde U_1'\{i,i\}$, $\widetilde V_1'\{i,i\}$ are the identity matrix.

Recall that $(D D_{\gamma}^{\mathbf r} )^{-1} A^{<3>}(D D_{\gamma}^{\mathbf r} )=(D_{\gamma}^{\mathbf r})^{-1}(Q^{\mathbf m'}_\alpha)^{-1}B^{<3>}Q^{\mathbf m'}_\alpha D_{\gamma}^{\mathbf r}$. Since $A$ satisfies condition $\mathcal C_1$,
\[
(\widetilde U_A,\widetilde V_A):( I-A'''') \to ( I-{A^{<3>}}'''')
\]
is a positive 
$\textnormal{El}_{\mathcal P}(\mathbf s,\mathcal  H)$-equivalence such that for every $i\in \mathcal C$ the matrices
$\widetilde U_A\{i,i\}$, $\widetilde V_A\{i,i\}$ are the identity matrix, and $A^{<3>}$ is nondegenerate, it follows that $A^{<3>}$ satisfies condition $\mathcal C_1$. It then follows from the proof of 
Theorem 4.7 of \cite{BSullivan} that $\gamma \in G^N$ can be chosen such that for each $i\in\mathcal C$ the diagonal matrix 
$(DD_\gamma^{\mathbf r})_i$ is $1$. Applying Proposition \ref{cohomologyaspositive}, we get stabilizations $A^{<3>}_1$, $M_1$ of $A^{<3>}$, $(D_{\gamma}^{\mathbf r})^{-1}(Q^{\mathbf r}_\alpha)^{-1}B^{<3>}Q^{\mathbf r}_\alpha D_{\gamma}^{\mathbf r}$ and a positive $\textnormal{El}_{\mathcal P}(\mathbf s_1,\mathcal  H)$-equivalence
\[
(\widetilde U_1,\widetilde V_1):( I-A^{<3>}_1) \to ( I-M_1)
\]
such that for every $i\in \mathcal C$ the matrices
$\widetilde U_1\{i,i\}$, $\widetilde V_1\{i,i\}$ are the identity matrix.

{\bf Step 8: Conclusion.} 

By composing the equivalences $(\widetilde U_B,\widetilde V_B):( I-B'''') \to ( I-{B^{<2.5>}}'''')$ and $(\widetilde U_1',\widetilde V_1'):( I-B^{<2.5>}_1) \to ( I-B^{<3>}_1)$, we get stabilizations $B_2,B^{<3>}_2\in\mathcal{M}_{\mathcal{P'}}^o(\mathcal{C'}, \mathbf s'_2,\mathcal{H'})$ of $B$, $B^{<3>}$ and a positive 
$\textnormal{El}_{\mathcal P'}(\mathbf s_2',\mathcal  H')$-equivalence
\[
(\widetilde U_2',\widetilde V_2'):( I-B_2) \to ( I-B^{<3>}_2)
\]
such that for every $i\in \mathcal C'$ the matrices
$\widetilde U_2'\{i,i\}$, $\widetilde V_2'\{i,i\}$ are the identity matrix. By multiplying this equivalence with $(D_{\gamma}^{\mathbf s'_2})^{-1}(Q^{\mathbf s'_2}_\alpha)^{-1}$ on the left and $Q^{\mathbf s'_2}_\alpha D_{\gamma}^{\mathbf s'_2}$ on the right we get stabilizations $C_2,M_2\in\mathcal{M}_{\mathcal{P}}^o(\mathcal{C}, \mathbf s_2,\mathcal{H})$ of $C$, $(D_{\gamma}^{\mathbf s_2})^{-1}(Q^{\mathbf s_2}_\alpha)^{-1}B^{<3>}Q^{\mathbf s_2}_\alpha D_{\gamma}^{\mathbf r}$ (where $\alpha^*(\mathbf s_2)=\mathbf s'_2$) and a positive 
$\textnormal{El}_{\mathcal P}(\mathbf s_2,\mathcal  H)$-equivalence
\[
(\widetilde U_2,\widetilde V_2):( I-C_2) \to (I-M_2)
\]
such that for every $i\in \mathcal C$ the matrices
$\widetilde U_2\{i,i\}$, $\widetilde V_2\{i,i\}$ are the identity matrix.
By composing the inverse of this equivalence with the equivalences $(\widetilde U_A, \widetilde V_A):( I-A'''') \to ( I-{A^{<3>}}'''')$ and  $(\widetilde U_1,\widetilde V_1):( I-A^{<3>}_1) \to ( I-M_1)$, we get stabilizations $\widetilde A^{<0>},\widetilde C^{<0>}\in\mathcal{M}_{\mathcal{P}}^o(\mathcal{C}, \mathbf t,\mathcal{H})$ of $A$, $C$ and a positive 
$\textnormal{El}_{\mathcal P}(\mathbf S_{  3},\mathcal  H)$-equivalence
\[
(\widetilde U,\widetilde V):( I-\widetilde A^{<0>}) \to ( I-\widetilde C^{<0>})
\]
such that for every $i\in \mathcal C$ the matrices
$\widetilde U\{i,i\}$, $\widetilde V\{i,i\}$ are the identity matrix.

It remains to obtain the equivalence in 
$\mathcal C_1$ form.
Let $\mathcal I$ be the index set of $\widetilde A^{<0>}$ (and $\widetilde C^{<0>}$), and let $\mathcal I_{\textnormal{sec}}$ be the set of elements $s\in\mathcal I$ such that $i(s)\in\mathcal C$ and $\widetilde A^{<0>}(s,t)=\widetilde A^{<0>}(t,s)=0$ for all $t\in\mathcal I$. Since $\widetilde U(I-\widetilde A^{<0>})\widetilde V=I-\widetilde C^{<0>}$, $\widetilde U\{i,i\}$ and $\widetilde V\{i,i\}$ are the identity matrix for every $i\in \mathcal C$, and $\widetilde A^{<0>}$ and $\widetilde C^{<0>}$ satisfy condition $\mathcal C_{1+}$ (because $A$ and $C$ satisfy condition $\mathcal C_1$), it follows that $\mathcal I_{\textnormal{sec}}$ is equal to the set of $s\in\mathcal I$ such that $i(s)\in\mathcal C$ and $\widetilde C^{<0>}(s,t)=\widetilde C^{<0>}(t,s)=0$ for all $t\in\mathcal I$. Let $\mathcal I_{\textnormal{prim}}$ be the complement in $\mathcal I$ of the $\mathcal I_{\textnormal{sec}}$.
Let $\widetilde W=\widetilde V^{-1}$ and write the equivalence in the 
form 
\begin{equation} \label{Wform}
\widetilde U (I-\widetilde A^{<0>}) = (I-\widetilde C^{<0>})\widetilde W.
\end{equation} 

Let $A^{<0>}$, $C^{<0>}$, $U$, $W$ be the principal submatrices of $\widetilde A^{<0>}$, $\widetilde C^{<0>}$, $\widetilde U$, $\widetilde V$ with index set $\mathcal I_{\textnormal{prim}}$. Then $A^{<0>},C^{<0>}$ are $\mathcal C_1$ stabilizations in $\mopcmh$ of $A,C$.
If we can show 
$U (I-A^{<0>}) = 
(I-C^{<0>}) W$, then we have 
the required 
$\mathcal C_1$-equivalence. For a verification, 
suppose $t,u\in \mathcal I_{\textnormal{prim}}$. Then 
\begin{align*} 
\big( U(I-A^{<0>}) \big) (t,u) \ 
&=\  U(t,u)-(UA^{<0>})(t,u) \\ 
&=\  U(t,u)-\sum_{s\in \mathcal I_{ \textnormal{prim}}} U(t,s)A^{<0>}(s,u) \\
&=\  \widetilde U(t,u)-\sum_{s\in \mathcal I_{ \textnormal{prim}}} \widetilde U(t,s)\widetilde A^{<0>}(s,u) \\ 
& =\  \widetilde U(t,u)-\sum_{ s\in \mathcal I}
\widetilde U(t,s)\widetilde A^{<0>}(s,u) \\ 
&=\ \big( \widetilde U (I-\widetilde A^{<0>})\big) (t,u). 
\end{align*}
Likewise, 
\[
\Big ( (I-C^{<0>})W \Big) (t,u) 
=
\Big ( (I-\widetilde C^{<0>})\widetilde W \Big) (t,u).
\] 
The required equality now follows from \eqref{Wform}.

%
%
%
%
%
%
%
%
%

\end{proof} 


\section{The Factorization Theorem: setting}
\label{sec:factorizationtheoremsetting}

In this section we present  the Factorization Theorem
\ref{theoremfactor}, 
which we shall use to prove the implication $(2) \implies (1)$
in Theorem \ref{classification}, and establish the setting
  for the proof of Theorem \ref{theoremfactor}.

As before, $G$ is a finite group, $\mathcal P=\{1,\dots,N\}$ is a finite poset given by a partial order relation $\preceq$ satisfying $i\preceq j\implies i\le j$, $\mathcal C$ is a subset of $\mathcal P$, and $\mathcal H=\{H_{ij}\}_{i,j\in\mathcal P}$ is a $(G,\mathcal P)$ coset structure.
For the class $\mopcnh$, recall Definition \ref{defn:mopcnh}.

\begin{definition} \label{conditionsc2} 
A matrix $A$ in $\mopcnh$ satisfies condition $\mathcal C_2$ 
if the following holds: if $i\in \mathcal P $ and $i$ is not a cycle component of $A$, 
then there are matrices $U_i,V_i$ in $\elniHi$ 
such that 
$U_i(\myId-A_i)V_i$ is a block diagonal matrix 
with one summand a $2\times 2$ 
identity matrix.  
\end{definition}

\begin{theorem}[Factorization Theorem] \label{theoremfactor} 
Suppose $A$ and $A'$ are  matrices 
in $\mopcnh$  
which satisfy conditions $\mathcal C_1$ and $\mathcal C_2$. 
Then the following are equivalent. 
\begin{enumerate}
\item 
$(U,V): \rep{\myId-A} \rightarrow \rep{\myId-A'}$ 
is an $\elpnh$-equivalence.  
\item 
$(U,V): \rep{\myId-A} \rightarrow \rep{\myId-A'}$ 
is a positive $\elpnh$-equivalence. 
\end{enumerate} 
\end{theorem}

We do not have a sharp statement 
as to which general $\elph$-equivalences are positive $\elph$
equivalences. However, the restriction above to matrices satisfying 
$\mathcal C_1$ and $\mathcal C_2$ is rather mild. 
Condition $\mathcal C_2$
is a harmless technical 
condition (achievable by replacing $A$ with a larger
  stabilization) 
  which is
needed below to apply the 
Factorization Theorem proved in \cite{BSullivan} for the case that the 
presenting matrix $A$ over $\ZZ_+G$ is essentially
irreducible. 
Any $G$-SFT can be presented by 
a matrix in some $\mopcnh$
(by Proposition \ref{pro:right form}) 
which satisfies $\mathcal C_1$
 and $\mathcal C_2$
(by Proposition \ref{gettingpositive}).
Before turning to the setting of Theorem \ref{theoremfactor},
  we remark on  a possible future application.

\begin{remark}\label{BFrepn} Let $S$ be a nontrivial mixing SFT
    defined by a matrix
    $A$ over $\Z_+$. There is a ``Bowen-Franks'' representation, a 
    homomorphism $\beta$ from the mapping class group of $S$ to
    the group of automorphisms of $\cok(I-A)$.
    This map $\beta$ is fundamental to
    understanding the mapping class
    group; for example, quite possibly its kernel is a  simple group. 
    That $\beta $ is surjective 
is     a corollary of the Factorization
Theorem of \cite{mb:fesftpf} (as proved in \cite{mb:fesftpf}),
because every automorphism of
    $\cok(I-A)$ is induced by an elementary equivalence over $\Z$.

The mapping class group of a $G$-SFT, an SFT $T$ with
    $G$ acting by automorphisms, is naturally defined as the centralizer
    in the mapping class group of $T$ of the subgroup of elements induced by the $G$-action. 
        For a mixing $G$-SFT we likewise have a map $\beta$
from this centralizer into  the automorphism group of the $\ZG$-module $\cok(I-A)$.
    The Factorization Theorem tells us that the automorphisms
    in the range of $\beta $ are those induced by elementary $\ZG$ 
    equivalences. The problem of characterizing these has not
    to our knowledge been addressed.  Whatever analogous
    algebraic structures are developed for the case of reducible
    $G$-SFTs,  Theorem \ref{theoremfactor}
    will be a tool for
    characterizing the range of invariants.
    \end{remark}

It will be  convenient to have a setting
  in which we work with
   equivalences $(A-I)\to U(A-I)V$ (rather than $(I-A)\to U(I-A)V$),
   within
   a class of 
    matrices  $A-I$  which are  ``as positive as possible''. We
    develop the apparatus for this next.

\begin{definition} \label{rdefinition}  
Given $\mathcal H$
and $i\prec j$, $\mathcal D_{ij}=\mathcal D_{ij} (\mathcal H)$ is the 
set of $(H_i,H_j)$ double cosets contained in 
$H_{ij}$, and 
$\mathcal R_{ij}=\mathcal R_{ij} (\mathcal H)$ 
is the set of $D \in \mathcal D_{ij}$ such that 
\begin{equation}\label{nonext}
 i\prec k \prec j 
\implies H_{ik}H_{kj}\cap D = \emptyset \ . 
\end{equation}
\noindent
We also  define 
\begin{align} 
\mathcal R^{\mathcal C} &= 
\{ (i,j,D) : D\in\mathcal R_{ij},  i\in \mathcal C
\text{ and } j\in \mathcal C
\} \ .
\end{align}
\end{definition}
\noindent
When $\mathcal H$ is the coset structure of a matrix $A$
  in $\mopnh$, and  $g\in D\in \mathcal R_{ij}$,
  then $g$ cannot be the weight of a path
  which goes from $\mathcal I_i$ to $\mathcal I_j$
  without passing some other $\mathcal I_k$.

 \begin{example}\label{twoDs}
With $G=S_{  3}$, for the $G$-SFT  given by 
\[
 \begin{pmatrix}
\perme&\perme&\perme+\perm{12}\\
0&\perme&\perme\\
0&0&\perm{123}
\end{pmatrix}
\] 
we have
$\mathcal C=\{1,2,3\}$, $H_1=H_2=\{\perme\}$, $H_3=A_{  3}$, $H_{12}=\{\perme\}$, $H_{13}=S_{  3}$,  $H_{23}=A_{  3}$. 
$H_{13}$ contains both 
$(H_1,H_3)$ double cosets, namely $D_1=A_{  3}$ and $D_2=S_{  3}\backslash A_{  3}$. We have 
$(1,3,D_1)\notin \mathcal R^{\mathcal C}$ and $(1,3,D_2)\in \mathcal R^{\mathcal C}$.
 \end{example}

If $D$ is a nonempty subset of $G$, then $\pi_D$ is the projection 
\[
\pi_D : \sum_{g\in G} n_g g\mapsto \sum_{g\in D} n_g g \  . 
\] 
An element $\sum_{g\in G} n_g g$ is $D$-positive if $n_g  \geq 0$ for any $g$ and $n_g>0$ precisely when $g\in D$.   The terms are used for matrices when  the conditions hold entrywise.

\begin{definition}  \label{defnmpppcnh}
$\mpppcnh$ 
  is the set of matrices $M$ in
$\mpnh$
whose blocks $M\{i,j\}$ satisfy the following 
conditions: 
\begin{enumerate}
\item
  $M+I \in \mopcnh$.

\item 
$M\{ i,i\}$ is $H_i$-positive if $i\notin \mathcal C$ . 
\item \label{dijcondition}
  If $i\prec j$ and $D\in \mathcal D_{ij}$, then 
\begin{align*} 
(i)\ \quad (i,j,D) \notin \rc &\implies \pi_D M\{ i,j \} \text{ is }
D\text{-positive} \\
(ii) \quad (i,j,D) \in \rc &\implies \pi_D M\{ i,j \} > 0 \ .
\end{align*}
 \end{enumerate} 
By definition, the condition 
$\pi_D M\{ i,j \}> 0$ means that every entry of 
$\pi_DM\{ i,j \}$ is nonnegative and nonzero.
Note, Condition 1 above implies
  that $M\{ i,i \}$ has the form $(g_i-1)$ for some $g_i\in G$ if $i\in \mathcal C$. 
\end{definition}

\begin{definitions} \textnormal{
An \emph{elementary positive equivalence} in 
$\mpppcnh$ 
is an  
$\elpnh$ 
 equivalence $(U,V)\colon B\to B'=UBV$ such that 
the following hold: 
\begin{enumerate} 
\item 
$B,B' \in 
\mpppcnh$;
\item 
one of $U,V$ equals $\myId$; 
\item 
one of $U,V$
 is a basic elementary matrix. 
\end{enumerate} 
A \emph{positive equivalence} in 
$\mpppcnh$ 
is a composition of 
 elementary positive equivalences in 
$\mpppcnh$. 
For such an equivalence, we  use notations such as 
\[
\begin{CD}
(U,V)\colon B @>>+> B'  
\end{CD}
\quad \textnormal{or} \quad 
\begin{CD}
B @>(U,V)>+> B'  
\end{CD}
\quad \textnormal{or} \quad 
\begin{CD}
B @>>+> B'  
\end{CD}
\ . 
\]
}
\end{definitions}

\begin{observation}\label{observation}
Suppose 
$A,A'$ are in $\mopcnh$;  
$B=A-\myId$, $B'=A'-\myId$;  and 
\[
\begin{CD}
(U,V)\colon B @>>+> B'  
\end{CD}
\ . 
\]
Then $(U,V)\colon \myId-A\to \myId-A'$ is a positive 
$\elpnh$-equivalence. 
\end{observation}

\begin{prop}\label{gettingpositive}   
Suppose  $A\in \mopcnh$.
    \begin{enumerate}
    \item Suppose $A^{<0>}$ is the stabilization of $A$ in
$\mopckh$, where 
      $k_i= n_i+2$ if $i\notin \mathcal C$ and $k_i=n_i$ if $i\in \mathcal C$. 
Define $\mathbf m= (k_1, \dots , k_n)$
by  $m_i=k_i$ if $i\notin \mathcal C$, $m_i=1$ if $i\in \mathcal C$. 
Then there is a matrix $A'$ in $\mopcmh$, satisfying 
$\mathcal C_1$ and $\mathcal C_2$,
such that 
the $1$-stabilization in $\mopckh$ of $I-A'$ is positive $\elpkh$-equivalent to 
    $I-A^{<0>}$.

\item
Suppose $A$  satisfies  $\mathcal C_1$ and $\mathcal C_2$.      
Then there is a matrix $A'$, satisfying $\mathcal C_1$ and $\mathcal C_2$,
such that
$A'-I \in \mpppcnh$ and there is a 
positive $\elpnh$-equivalence from 
$\myId-A$   to $\myId-A'$.\newline 
\end{enumerate} 
\end{prop}
We postpone the proof of Proposition~\ref{gettingpositive} to later in this section.

Proposition \ref{gettingpositive}, along with 
  Observation \ref{observation},
  tells us that (after accepting  that
  we might need to pass to slightly larger matrices to 
  satisfy $\mathcal C_2$, and then to smaller matrices to satisfy 
  $\mathcal C_1$) we have reduced the problem of showing 
  every  $\elpnh$-equivalence of matrices $I-A$ (with $A$
  from $\mopcnh$) is positive 
to the  problem of showing  
every $\elpnh$-equivalence of matrices
$M$ (with $M$ from $\mpppcnh$)
is a positive equivalence in   $\mpppcnh$.

Unfortunately, as we see in  Example \ref{badRexample} below, 
  Proposition \ref{gettingpositive} would not be true
  if in the definition of $\mpppcnh$
    the condition \lq\lq $\pi_DM\{i,j\} >0$\rq\rq\  
    in 3(ii)    were strengthened to 
    $D$-positivity.\begin{footnote}
      {From its proof, one sees that Prop. \ref{gettingpositive}
        can be strengthened such that,  given $k\in \N$, 
for  $D\in \mathcal D_{ij}\setminus \mathcal R^{\mathcal C}$,
entrywise $A'\{i,j\}\geq k\sum_{g\in D} g$. Such flexibility is not
available 
 for $D$ in $\mathcal R^{\mathcal C}$.}
      \end{footnote}
  This complication  accounts for
  a good deal of the difficulty  of  arguments to come.

\begin{example}\label{badRexample}
With
   $G= S_{  3}$, define matrices over $\Z_+G$,
\[
A_1=
\begin{pmatrix}
\perm{12}&\perme\\
0&\perm{12}
\end{pmatrix}
\quad \quad \text{and} 
  \quad \quad
A_2=  
\begin{pmatrix}
  \perme&\perme&\perme+\perm{12}\\
0&\perme&\perme\\
0&0&\perm{123}
\end{pmatrix} \ . 
\]

For $A_1$, $(G,\mathcal P)=( S_{  3},\mathcal P_2)$; 
  $\mathcal H$ is given by $\gen{\perm{12}}=
  H_1=H_2=  H_{12}$; 
  $\mathcal C=\{1,2\}$; and
  $ D=H_{12}$ is an $(H_1, H_2)$ double coset,  
with   $(1,2, D)\in \mathcal R^{\mathcal C}$. 
 A basic positive
  $\elpnh$-equivalence $(I-A)\to (I-B)$ 
  must be given by right or left multiplication
  of $(I-A)$ by one of four matrices,
  $\left( \begin{smallmatrix} 1&\pm g\\0&1 \end{smallmatrix} \right)$ 
  with $g\in S_{  3}$. The orbit of
  $I-A_1$ under such equivalences consists of
  just $I-A_1$ and $I-A_3$, where $A_3=
  \left( \begin{smallmatrix}   \perm{12}&\perm{12}\\ 0&\perm{12}
  \end{smallmatrix} \right)$. Neither $A_1-I$ nor $A_3-I$ is
$\mathcal D$-positive. (Alternately, we can see that
a matrix  $ B=\left( \begin{smallmatrix}
  \perm{12}&x\\ 0&\perm{12}
\end{smallmatrix} \right)$ defining a $G$-SFT which is $G$-flow equivalent
to $T_{A_1}$, with $x= n_1 \perme + n_2 \perm{12} \in \Z_+G$, must
satisfy $n_1+n_2=1$, because the augmentations
$\overline{A_1}=
\left( \begin{smallmatrix}
  1&1\\ 0&1 
\end{smallmatrix} \right)
$ and $\overline{B}=
\left( \begin{smallmatrix}
  1&n_1+n_2\\ 0&1 
\end{smallmatrix} \right)
$ must define flow equivalent
SFTs.) 

For $A_2$, continuing from  Example \ref{twoDs} 
we have 
$(1,3,D_2)\in \mathcal R^{\mathcal C}$.
A 
 nonidentity basic positive
$\elpnh$-equivalence can only be given by multiplication by
a matrix $E_{st}(\pm g)$, with $g \in S_{  3}$ and $s<t$.
It is straightforward to check that if $I-B$ is in the orbit of
$I-A_1$ under such positive equivalences, then $B(s,t)\in \Z_+A_3$ if
$(s,t)\neq (1,3)$, and for $B(1,3) = \sum_{g\in G} n_g g$ we
have $\sum_{g\notin A_3} n_g =1$.  Thus, $B-I$ cannot be
$\mathcal D_2$-positive.
\end{example}

Before giving the proof of
Proposition
\ref{gettingpositive}, we introduce further notation. 

\begin{definition} \label{deltadefn}
We define 
\begin{align*} 
\delta_{ij} = & \sum_{g\in H_{ij} } g
\quad \quad \quad \text{if } i\prec j \\ 
\delta_{i} = & \sum_{g\in H_{i} } g
\end{align*} 
\end{definition} 
If $i\in \mathcal C$, then 
$\delta_i =  \sum_{m=0}^{ \kappa (g_i) -1 } (g_i)^m $,
where $\kappa (g) $ is the order of $g$ in $G$. 
\begin{definition} \label{rhoandS}
Let $\mathcal S =\{ (i,j) \in \mathcal P \times \mathcal P: 
i\prec j \}$. Define $\mathcal S_0=\emptyset$. Inductively, given 
$\mathcal S_m$, define $\mathcal S_{m+1}$ to be the set of 
$(i,j)$ in $\mathcal S \setminus \cup_{r\leq m} \mathcal S_r$
such that
\[
i\prec k\prec j \implies 
\{(i,k),(k,j) \}\subset \bigcup_{k=0}^m \mathcal S_k \ . 
\]
Define $\rho (i,j) =m$ if $(i,j)\in \mathcal S_m$.
\end{definition}
To make the sets $\mathcal S_m$ more concrete, we prove the next proposition.
For this, we define a path of length $\ell$ in $\mathcal P$ from $i_0$ to
$i_{\ell}$ to be a string $(i_0,i_1, \dots , i_{\ell})$ such that
 $i_{t-1}\prec i_t$ for $1\leq t \leq \ell$. Such a path is {\it maximal} 
if  $(i_{t-1},i_t) \in \mathcal S_1$ for $1\leq t\leq \ell$. For
$i\prec j$, define $\rho (i,j)$ to be the maximum length of a
path from $i$ to $j$.
\begin{proposition} $(i,j)\in \mathcal S_m$ if and only if $\rho (i,j)=m$.
\end{proposition}
\begin{proof}
  The proof is by induction on $m$. The basis step $m=1$ is straightforward.
  Assume the claim is true for $m$.
  If $\rho (i,j) = m+1$, then clearly $(i,j) \in \cup_{t\leq m+1}\mathcal S_t$;
  by the induction hypothesis, $(i,j) \notin \cup_{t\leq m}\mathcal S_t$,
  and therefore $(i,j) \in \mathcal S_{m+1}$. Now suppose $(i,j) \in
  \mathcal S_{m+1}$. Then for every maximal path
  $(i,i_1, \dots, i_k,j)$,
  we have by the induction hypothesis that $\rho(i,i_k)\leq m$, and
  therefore 
  $k\leq m$ and the length of $(i,i_1, \dots, i_k,j)$ is
  at most $m+1$. Therefore $\rho (i,j) \leq m+1$. By the induction
  hypothesis, $\rho(i,j) > m$, and therefore $\rho (i,j) = m+1$.
  \end{proof} 
Note, $i\prec k \prec j$ implies
$\max \{ \rho (i,k), \rho(k,j)\} < \rho(i,j)$.
This is a key point for the proof by induction below.

\begin{example}
With $\mathcal P = \mathcal P_4$, we have
 $\mathcal S$ partitioned into
  $\mathcal S_1=\{(1,2),(2,3),(3,4)\}$,
  $\mathcal S_2=\{(1,3), (2,4)\}$ and
    $\mathcal S_3=\{(1,4)\}$.
\end{example}

\begin{proof}[Proof of Prop. \ref{gettingpositive}]

  (1) Clearly $A^{<0>}$ satisfies $\mathcal C_2$, as does any matrix
    $\elpkh$-equivalent to $A^{<0>}$.  Now apply the initial part
    of the proof of Proposition \ref{pro:right form} to $A^{<0>}$,
    iterating the trimming move of Example \ref{cutandtrim} to produce
    a matrix $A_1$ in which a diagonal entry is zero iff the row and column
    through that entry are zero. By induction, these
    trimming moves are positive 
    $\elpkh$-equivalences in $\mopcnh$, because $A_1 \in \mopcnh$.
    If $i \in \mathcal C$, then the nonzero diagonal entry of
    $A_1\{i,i\}$ must be unique. It follows that $A_1$ is the stabilization of
    a matrix $A'$ as claimed.

  For the proof of (2), suppose $A \in \mopcnh$, satisfying $\mathcal C_1$
    and $\mathcal C_2$. Without loss of generality, suppose also that 
a diagonal entry in $A$ is zero iff the row and column
through that entry are zero. Given  $i\notin \mathcal C$,
let $C_i=A\{i,i\}$. 
Steps 0,3,4,5,6 of \cite[Lemma 6.6]{BSullivan} give an explicit
string of basic positive $\Z H_{i}$-equivalences which send
$I-C_i $ to a matrix $I-B_i$ such that $B_i-I$ is $H_i$-positive. These 
basic positive equivalences are given by row and column cuts
(recall subsections 
\ref{subsec:rowcut},\ref{subsec:colcut})
and therefore the elementary row and column
operations defining them give a string of basic positive $\elpnh$-equivalences
from $I-A$ to some matrix $I-B$ in $\mopcnh$ such that 
$(B\{i,i\}-I)$ is $H_i$-positive. After applying such
equivalences for every $i\notin \mathcal C$, and renaming, 
we may assume without loss of generality that
$A$
satisfies all conditions 
of 
Definition \ref{defnmpppcnh} 
except perhaps condition   (\ref{dijcondition}).

To arrange for condition (3), we consider 
double cosets $D\in \mathcal D_{ij}$ by induction, addressing
the sets $D_{ij}$ in an order with $\rho (i,j)$ nondecreasing. 
The positive equivalences below are  compositions of 
 $\elpnh$-equivalences, of the form
$(I-B) \to (I-C)$ with $B\leq C$. 
So, all the matrices $C$ continue to satisfy
Conditions 1 and 2. Also, if Condition 3 is satisfied by
$B$ for  $D\in\mathcal D_{i'j'}$, then this remains true for $C$.
In particular, when considering $D\in D_{ij}$, if e.g. $\rho (i,k) < \rho (i,j)$
then we may assume Condition 3 holds for $D$ in $D_{ik}$.

We begin with a claim.

{\bf Claim 1.} Suppose $i\notin \mathcal C$ or $j\notin \mathcal C$; 
$(s,t) \in \mathcal I_i \times \mathcal I_j$; $ g\in \mathcal H_{ij}$;  and
$A(s,t)\geq g \in \mathcal H_{ij}$. Let $D= H_igH_j$. Then there is a
positive equivalence $(I-A)\to (I-A')$ where
$A'\geq A$ and $A'\{i,j\}$ is $\mathcal D$-positive. 

\begin{proof}[Proof of Claim 1] We assume $j\notin \mathcal C$ (the
  proof for the case
$i\notin \mathcal C$ is similar).  Let $E= E_{st}(g)$.
  There is a positive equivalence $(I-A) \to (I-B)$ 
(a $(g,s,t)$ row cut equivalence, as in 
Sec \ref{subsec:rowcut}) 
implemented 
by $I-B = E(I-B)$. Here row $s$ of $B\{ i, j\}$ is $gH_j$-positive,
because
\begin{align*}
  B(s,t) &= A(s,t) -g +gA(t,t) \geq g\delta_j \ , \\ 
  B(s,t') &= A(s,t') +gA(t,t') \geq g\delta_j \ \ 
  \text{if } t' \in \mathcal I_j \text{ and } t'\neq  t\ , 
\end{align*}
and the inequalities hold because $(A-I)\{j,j\} $ is $H_j$-positive.
For use in Subcase 2,
after performing a second $(g,s,t)$ row cut equivalence (if necessary),
we obtain $B'$ with 
\begin{align*}
  B'(s,t') &= B(s,t') +gB(t,t') \geq A(s,t') + 2g\delta_j \ \ 
  \text{if } t' \in \mathcal I_j  \ .  
\end{align*}
There are two cases for the rest of the proof.

Subcase 1: $i\notin \mathcal C$. Suppose $t' \in \mathcal I_j$; 
$B(s,t') \geq h$; and $E=E_{st'}$.  Choose $gh\in gH_j$. Then 
there is a positive equivalence to $(I-B) \to (I-C)$
(a $(gh,s,t')$ column cut equivalence, as in 
Sec \ref{subsec:colcut}) implemented by 
$(I-C) = (I-B)E$. Here column $t'$ of $C\{ i,j \}$ is $H_igH_j$
positive, because 
\begin{align*}
  C(s,t') &= B(s,t') -gh + B(s,s) A(s,t') \geq \delta_i g\delta_j\ , \quad
 \\
 C(s',t') &= B(s',t')  + B(s',s)A(s,t') \geq \delta_i g \delta_j  \ ,
   \text{if } s'\in \mathcal I_i \text{ and }s'\neq s \ , 
\end{align*}
where the inequalities hold because $(B-I)\{i,i\}$ is $H_i$-positive.
So, after implementing $(h,s,t')$ column cuts for each $t'$ in $\mathcal I_j$,
we pass to $I-A$ where $A\{i,j\}$ is $\mathcal D$-positive.

Subcase 2: $i\in \mathcal C$. 
Now $\mathcal I_i = \{ s\}$.
Let $B(s,s) = g_i$.
Suppose $t'\in \mathcal I_j$.
Let $(g_i)^{\ell} =1$.
Starting from $B'$, apply $((g_i)^k,s,t')$ column cuts, for
$k=0,1,\dots ,\ell -1$, to produce $C_0, \dots , C_{\ell-1}$.
Because $B'(s,t') \geq 2g\delta_j$, we have 
\begin{align*}
  C_0(s,t') &= B'(s,t') -g +  g_iB'(s,t')  
  > g\delta_j +  g_ig 2\delta_j  \\
  C_1(s,t') &= C_0(s,t') -g_ig +  g_iC_0(s,t') 
   > g\delta_j  + g_ig\delta_j  + 
 (g_i)^22 g\delta_j \\
& \quad \ \ \ \   \dots \\
  C_{\ell -1}(s,t') &  > g\delta_j  + g_ig\delta_j 
  + \dots + (g_{i-1})^{\ell -1} (\delta_j -1)
  = \delta_i g \delta_j \ . 
\end{align*}

\end{proof}

Now we consider $D\in \mathcal D_{ij}$. If $D\in \mathcal R_{ij}$,
then there must be some $g$ in $D$ and some $(s,t)$ in $
\mathcal I_s \times \mathcal I_t$ such that
$A(s,t)\geq g$. If $\{i,j\} \subset \mathcal C$,
then $\mathcal A\{ i,j \}$ has only one entry, and
therefore Condition 3(ii) holds for $D$. 
If $\{i,j\}$ is not contained in $\mathcal C$, then
by  Claim 1 there is a positive equivalence to
some $A'$ with $A'\{i,j\}$ $D$-positive.

From here,  suppose   $(i,j,D)\notin \mathcal R$.
Because $\mathcal H$ is a coset structure for $A$, we deduce that
there exists $k$ such that $i \prec k \prec j$ and
$D\subset H_{ik}H_{kj}$. 
So, given $i\prec k \prec j$, we will finish by
producing a  positive equivalence
replacing $A$ with a matrix $A'$, $A'\geq A$,  
such that $A'\{i,j\}$ is $H_{ik}H_{kj}$-positive.
This argument goes by cases. For an element $i$ of $\mathcal C$, we
let $g_i$ be the unique entry of $A\{i,i\}$. 

{\bf Case 1:} $k \notin \mathcal C$, or $\{ i,j\} \cap \mathcal C =
\emptyset$.
Suppose $(s,r) \in \mathcal I_i \times \mathcal I_k$.
Perform an $(A(s,r),s,r)$ row cut equivalence
on $A$ to produce a matrix $B$ such that
\begin{align*} 
  B(s,r) &= A(s,r)B(r,r) \geq \delta_{ik} \\
  B(s,t) &= A(s,t) +A(s,r) B(r,t) \geq \delta_{ik} \delta_{kj} \quad 
  \text{if } t\in \mathcal I_j\ ,  
  \end{align*}
where the inequalities hold because, by the induction hypothesis,
$(A-I)\{i,k\}$ is $H_{jk}$ positive and 
$(A-I)\{k,j\}$ is $H_{kj}$ positive. Thus row $s$ of
$B\{i,j\}$ is now $H_{ik}H_{kj}$ positive. Repeat as needed for all
$s$ in $\mathcal I_i$ to obtain $A'$ such that  $A'\{ i,j\}$
is $H_{ik}H_{kj}$ positive.

{\bf Case 2:} $\{i,k,j\}\subset \mathcal C$. 
Let $\mathcal I_i \times \mathcal I_k \times I_j= \{(s,r,t)\}$.

Looking only at the principal submatrices
on indices $\{s,r,t\}$, let
$A=\left( \begin{smallmatrix}
  g_i & b & c \\ 0 & g_k & d \\ 0 & 0& g_j
   \end{smallmatrix}\right)
$.
By the induction hypothesis, for $D\in H_{ik} $
we have $\pi_D(b) >0$, and for $D\in H_{kj} $
we have $\pi_D(d) >0$. Therefore
$H_{ik}H_{kj} \subset H_ibH_kdH_j$.
So, it suffices to  give a string of positive equivalences
producing $A'$ with $A'(s,t) \geq \delta_ib\delta_kd\delta_j$. 

Let $\ell = \kappa (g_i)$.
We apply in order $(b(g_k)^t,s,r)$ row cuts, for $t=0, 1, \dots ,
\ell -1$, starting with
$A$.
The first equivalence ($t=0$) 
is  implemented by 
\begin{align*}
\begin{pmatrix} 
1&b&0 \\
0&1&0\\
0&0&1 
\end{pmatrix} 
\begin{pmatrix} 
1-g_i&-b&-c \\
0&1-g_k&-d\\
0&0&1-g_j 
\end{pmatrix} 
&= 
\begin{pmatrix} 
1-g_i&-bg_k&-c -bd\\
0&1-g_k&-d\\
0&0&1-g_j 
\end{pmatrix} 
\end{align*}
After the last ($t=\ell -1$) equivalence we reach 
\[
I-B =
\begin{pmatrix} 
1-g_i&-b(g_k)^{\ell} &-c'\\
0&1-g_k&-d\\
0&0&1-g_j
\end{pmatrix} 
\]
where  $B\geq A$ (because $b(g_k)^{\ell} = b$) and
$c' = -A(s,t) -b(1 + g_k + \dots +(g_k)^{\ell -1}d
 =A(s,t) + b\delta_kd$. 
We follow this with 
 a column cut equivalence, 
\begin{align*}
\begin{pmatrix} 
1-g_i&-b&-c' \\
0&1-g_k&-d\\
0&0&1-g_j 
\end{pmatrix}
\begin{pmatrix} 
1&0&b\delta_k \\
0&1&0\\
0&0&1 
\end{pmatrix} 
&= 
\begin{pmatrix} 
1-g_i&-b&-c''\\
0&1-g_k&-d\\
0&0&1-g_j 
\end{pmatrix} 
\end{align*}
and then repeat the first move to reach a matrix $I-C$, $C\geq A$, 
with $C(s,t) = A(s,t) + (1 +g_i)b \delta_k d$. 
Iterating this process leads to a matrix $I-D$ such that
$D(s,t) = A(s,t)  + \delta_i b\delta_kd :=d'$ .
Setting $\delta''= \delta_i b\delta_kd $, 
 we then apply
the equivalence 
\begin{align*}
\begin{pmatrix} 
1&0&\delta'' \\
0&1&0\\
0&0&1 
\end{pmatrix} 
  \begin{pmatrix} 
1-g_i&-b&-d' \\
0&1-g_k&-d\\
0&0&1-g_j 
\end{pmatrix}
&= 
\begin{pmatrix} 
1-g_i&-b&-d''\\
0&1-g_k&-d\\
0&0&1-g_j 
\end{pmatrix} 
\end{align*}
to reach a matrix $I-E$ for which
$E(s,t) = A(s,t) + \delta_a b \delta_k d g_j$.
Iterating the move that produced $E$ from $A$, we
arrive at $A'$ with
$A'(s,t) = A(s,t) + \delta_a b \delta_k d \delta_j$.

{\bf Case 3:} $\{i,k\}\subset \mathcal C$,  $j\notin \mathcal C$.
Suppose $t\in \mathcal I_j$. Use the Case 2 moves which led
the matrix $D$ with
$D(s,t) = A(s,t) + \delta_a b \delta_k d $. By the induction
hypothesis, $d$ is $H_{kj}$ positive, and therefore
$E(s,t)$ is $H_{ik}H_{kj}$ positive. Iterating this move over
$t\in \mathcal I_j$, we reach $A'$ such that $A'\{ i, j\}$ is
$H_{ik}H_{kj}$ positive.

{\bf Case 4:} $\{j,k\}\subset \mathcal C$,  $i\notin \mathcal C$.
The proof here is similar to the proof for Case 3.

\end{proof}

\section{The Factorization Theorem: proof} \label{sec:factorizationtheoremproof}

This section is devoted to the proof of
  the Factorization Theorem \ref{theoremfactor}. 
We  use and generalize proof techniques 
from  \cite{BSullivan} and \cite{mb:fesftpf}.
We will prove three lemmas, 
and then use them 
  in a short argument to finish the proof of 
  Theorem \ref{theoremfactor}.

Let $\upnh$ be the set of matrices $M$ in 
$\elpnh$ such that every diagonal block 
$M\{ i,i\}$ is the identity matrix. 
We will address equivalences 
$(U,V)$ for matrices $U,V$ in $\upnh$. 

In Lemma \ref{2by2}, we will consider $2\times 2$ 
upper triangular matrices. 
Recall $\mathcal P_{2}=\{ 1,2\}$ with $1\prec 2$.

\begin{definition} \label{2by2defn}
Suppose 
$\mathbf m=(m_1,m_2)$; $\mathcal H$ is  a
$(G,\mathcal P_{2})$ 
coset structure; and 
$B,B'$ are in 
$\mathcal M^{++}_{\mathcal P_{2}}(\mathbf m,\mathcal H)$.
A string of basic elementary 
 positive $\mathcal U_{\mathcal P_{2}}(\mathcal H)$-equivalences 
 \[ 
B\xrightarrow[+]{(E_1,F_1)} \mbox{} 
\xrightarrow[+]{(E_2,F_2)} \dots
\xrightarrow[+]{(E_t,F_t)} 
B'
\] 
is {\it extendable} if 
the matrix products
$E_i\cdots E_1$  
and $F_1\cdots F_i$ are nonnegative,  $1\leq i \leq
t$ . 
In this case, 
with $(U,V)=( E_t\cdots E_2E_1, F_1F_2\cdots F_t)$, 
 we also say  that 
$(U,V):B\to C$ is an  
{\it extendable} positive equivalence.  
\end{definition} 

Our interest in extendable equivalences  is the following. 
Suppose $M,M'$ are in 
$\mpppcnh$ with $2\times 2$ principal 
submatrices $B,B'$ on the same coordinate 
indices $s,t$ contained in a block $\{i,j\}$ 
with $i\prec j$. Then 
an extendable positive equivalence of 
$B,B'$ 
(with respect to the restriction of $\mathcal H$)  
will give (by the same elementary 
operations) a positive equivalence from $M$ to $M'$ 
in $\mpppcnh$.  

\begin{lemma} \label{2by2}
Given matrices $B,B'$ in $\mpppcnh$, 
suppose $(U,V): B\to B'$ is a $\uph$-equivalence which only differs from the identity at blocks $U\{i,j\}$ and $V\{i,j\}$, and that every entry of $ U\{i,j\}$ and $V\{i,j\}$ lies in $\Z_+G$. 
Then $(U,V): B\to B'$ is an extendable positive 
$\elph$
equivalence, and consequently  
\[
\begin{CD}
(U,V)\colon B @>>+> B'  
\end{CD}
\]
\end{lemma}

\begin{proof}  
\mbox{}\\
\noindent
    {\bf Case 1 
      $i\not\in\mathcal C,j\not\in\mathcal C$:}\\ Since $B,B'$ have positive entries in all relevant diagonal blocks we can simply decompose $U$ and $V$ one entry at a time, thus obtaining an extendable positive $H_{ij}$-equivalence at every step.

\noindent
    {\bf Case 2 $i\in \mathcal C,j\not\in \mathcal C$:}\\
As in Case 1, $(U,I)\colon  B \xrightarrow[+] UB$. So, 
without loss of generality, we can assume
$(U,V)=(I,V)$. Considering compositions, it is enough to
address the case that $V$ has a single nonzero offdiagonal
entry, say, $V(1,2)=s \neq 0$. Then, in the principal submatrices
on indices $\{1,2\}$, the equivalence
$(I,V)\colon B\to BV$ has the form
$B = \left(\begin{smallmatrix} g-1 & r \\ 0 & h \end{smallmatrix}\right)
\to \left(\begin{smallmatrix} g-1 & r+(g-1)s \\ 0 & h \end{smallmatrix}\right)$.
Choose $p$ in $\Z_+H_{ij}$ such that $p>s$; then $ph>s$, because
$h$ is $H_j$-positive. 
Applying the equivalences $(E_{12}(p),I)$, $(I,V)$, $(E_{12}(-p))$ produces
\[
B\to 
\left(\begin{smallmatrix} g-1 & r+ph \\ 0 & h \end{smallmatrix}\right)
\to 
\left(\begin{smallmatrix} g-1 & r+ph +(g-1)s \\ 0 & h \end{smallmatrix}\right)
\to
\left(\begin{smallmatrix} g-1 & r +(g-1)s\\ 0 & h \end{smallmatrix}\right)\ .
\]
Let $E_{12}(p)=E_k\dots E_1$, $E_{12}(s)=F_1\dots F_k$ be factorizations
into nonnegative matrices with offdiagonal entries in $\mathcal H_{ij}$
Then $(I,V)$ is the composition 
$(E_1,I), \cdots ,(E_k,I),(I,F_1), \cdots ,(I,F_k),
(E_k^{-1},I), \cdots ,(E_1^{-1},I)$ 
and therefore $(I,V)$ is extendable.

\noindent
{\bf Case  3 $i\not\in \mathcal C,j\in \mathcal C$:}\\ 
The proof here is essentially as for Case 2.

\noindent
    {\bf Case 4
      $i\in \mathcal C,j\in\mathcal C$:}
    \\ Note first that in this case, the nontrivial matrices $U\{i,j\}$ and $V\{i,j\}$ are $1\times 1$. Let $p$ be the entry of $U\{i,j\}$ and  $s$ the entry of $V\{i,j\}$; we have assumed that $p,s\in \Z_+G$.

The proof is by induction on $K
=\overline{p+s}$, and the  lemma is true for $K=0$. 
Suppose $\overline{p+s}=K>0$ and the lemma holds 
if $\overline{p+s} <K$. 

Here  the  submatrix of the equivalence  $UBV=B'$ containing any change has the 
form 
\begin{equation} \label{case1} 
\begin{pmatrix} 
1&p\\ 0& 1
\end{pmatrix}  
\begin{pmatrix} 
g-1&r\\ 0& h-1
\end{pmatrix} 
\begin{pmatrix} 
1&s\\ 0& 1
\end{pmatrix} 
=
\begin{pmatrix} 
g-1&r'\\ 0& h-1
\end{pmatrix} \ . 
\end{equation} 
where 
\begin{equation}\label{sestar}
r'=r+p(h-1)+(g-1)s
\end{equation}
We use $E(x)$ to denote a 
matrix 
 $
\left(\begin{smallmatrix} 1&x\\0&1
\end{smallmatrix}\right)$; e.g., $U=E(p)$. 
For any $x,y,z$, 
\begin{equation}
\begin{pmatrix} 
1&x\\ 0& 1
\end{pmatrix}  
\begin{pmatrix} 
g-1&y\\ 0& h-1
\end{pmatrix} 
\begin{pmatrix} 
1&z\\ 0& 1
\end{pmatrix} 
=
\begin{pmatrix} 
g-1&
y +x(h-1) + (g-1)z
\\ 0& h-1
\end{pmatrix}  
\end{equation} 
so here the pair $(E(x),E(z))$ acts by adding 
$x(h-1) + (g-1)z$ to the $(1,2)$ entry. 
The equivalence given by $(U,V)$ is a composition 
of  basic elementary equivalences,  
given by $(\myId, E(w))$ or $(E(w),\myId)$, with 
$w\in G$ a summand of 
$p$ or $s$. Such an equivalence acts by adding 
a term $w'-w$ to the $1,2$ position, where 
$w'$ is $wh$ or $gw$. 

\noindent
{\bf Case 4(i):}\\   
Assume $r\neq r'$. Let $r=\sum_w n_ww$ and $r'=\sum_w n'_ww$. 
The images  
$\overline r$ and $\overline{r'}$ under the augmentation 
must be equal. So, there must be some $w\in G$ such 
that $n_w> n'_w$. Therefore $w$ must be a summand of 
$p$ or $s$, and $(\myId,E(w))$ or $(E(w),\myId)$ applied to 
$B$ is a positive equivalence in 
$\mpppcnh$.  Now the equivalence given by $(U,V)$ 
is this positive equivalence followed by one satisfying 
the induction hypothesis. A composition of extendable 
equivalences is extendable. This completes the 
inductive step if $r\neq r'$. 

\noindent
{\bf Case 4(ii):}\\
 Assume $r=r'$ and note that in this case 
\begin{equation}\label{sebullet}
p+s=ph+gs
\end{equation}
according to \eqref{sestar}.
 
Suppose $w_0$ is a summand of $p+s$. 
Then $w_0\in H_{12}$, since $(U,V):B\to B'$ is a 
$\uph$-equivalence. 
Because $B\in \mpppcnh$,  
there must be  a summand 
$x$ of $r$ and $i,j$ such that  $w_0=g^ixh^j$. 
We then have a 
positive equivalence $(E,F): B\to B_0$ defined by 
 \begin{align*} 
&\begin{pmatrix} 
g-1&r\\ 0& h-1
\end{pmatrix} 
\xrightarrow[+]{(\myId,E(x) )} \mbox{}
 \xrightarrow[+]{(\myId,E(gx) )} \cdots
\xrightarrow[+]{(\myId,E(g^{i-1}x) )} 
\\ 
&\quad \quad  
\mbox{}
\xrightarrow[+]{(E(g^{i}x),\myId )} \mbox{}
\xrightarrow[+]{(E(g^{i}xh),\myId )} \cdots 
\xrightarrow[+]{(E(g^{i}xh^{j-1}),\myId )} 
\begin{pmatrix} 
g-1&r+g^ixh^j -x
\\ 0& h-1
\end{pmatrix} 
\end{align*} 
Let $(E_t,F_t)$ be the $t$th of these basic positive equivalences, 
so, $(E,F)=(E_{i+j}\cdots E_1,F_1\cdots F_{i+j})$. 
Define 
\[
(E'_1,F'_1)=\begin{cases}(E(w_0),\myId)&\text{if }w_0\text{ is a summand of }p\\
 (\myId,E(w_0)) &\text{otherwise .}\end{cases}
 \]
  Now $(E'_1,F'_1): B_0
\xrightarrow[+]{ }
E'_1B_0F'_1 =:B_1$, with $B_1(1,2)=B_0(1,2)+ w_1-w_0$, where 
$w_1=w_0h$ if $E'_1=E(w_0)$
and 
$w_1=gw_0$ if $F'_1=E(w_0)$. 
In either case, according to \eqref{sebullet} and the definition of $(E'_1,F'_1)$, $w_1$ is a summand of $p+s$.
Since $B_1(1,2)=B_0(1,2)+ w_1-w_0$, if $w_1 \neq w_0$ then  
$w_1$ must be a summand of $p+s-w_0$, and we may construct 
a positive equivalence $(E'_2,F'_2): B_1\to B_2$ as before, with $w_1$ in 
place of $w_0$  and
$B_2 (1,2)=B_1(1,2) + w_2- w_1
  =B_0(1,2) + w_2- w_0$. Since $p+s$ is finite
and $r=r'$, this process 
must  reach some  $w_m=w_0$, and we set
$(E',F')=(E'_m\cdots E'_1,F'_1\cdots F'_m)$.  Because $B_m=B_0$, 
we have by composition a positive equivalence 
\[ 
(E',F'):
 B \xrightarrow[+]{(E,F)} 
B_0  \xrightarrow[+]{(E',F')} 
B_0 \xrightarrow[+]{(E^{-1},F^{-1})} 
B \ . 
\]
The equivalence 
$ B \xrightarrow{(E,F)} 
B_0  \xrightarrow{(E',F')} 
B_0 $ is extendable because the matrices 
$E_t,F_t,E'_t,F'_t$ are nonnegative. 
Extendability through the remaining basic equivalences 
holds because 
\begin{align*}
&\Big (E_t^{-1}\cdots E_1^{-1} E' E,\ F F'  F_1^{-1}\cdots
F_t^{-1} \Big)  \\ 
=\ &\Big ( E' E_t^{-1}\cdots E_1^{-1} E,\ F F_1^{-1}\cdots F_t^{-1} F' \Big) \\
= \ &\Big (E'E_{t+1}\cdots E_{i+j},\ F_{t+1} \cdots F_{i+j} F'\Big) \  
\end{align*}
and all the matrices in the last line are nonnegative. 
The equivalence 
$(U(E')^{-1},(F')^{-1}V): B\to B$
has the form $(E(p'),E(s'))$ with
  $\overline{p'+s'} = \overline{p+s}-m < \overline{p+s}$,
  and therefore is extendable by the 
induction hypothesis. This finishes the inductive step 
for Case 2(ii). 
\end{proof}

\begin{lemma} \label{unipotentlemma} Suppose 
$U$ and $V$ are matrices in $\upnh$, 
$B$ and $B'$ are in $\mpppcnh$, and 
$UBV= B'$. Then 
$(U,V): B   \xrightarrow[+]{ \ } B'$. 
\end{lemma} 
\begin{proof} 

Recalling Definition \ref{rhoandS}, let 
$(i_1,j_1), \dots , (i_r,j_r)$ 
be an enumeration of 
elements of $\mathcal S$  
such that 
$t\leq s \implies \rho (i_t,j_t) \leq \rho (i_s,j_s)$.

We will define various matrices by induction, beginning 
with $B_0=B, B'_0=B',U_0=U, V_0=V$. For $1\leq s \leq r$, 
given $B_{s-1}, B'_{s-1}, U_{s-1},V_{s-1}$ with 
$B_{s-1}, B'_{s-1}\in \mpppcnh$ and 
$U_{s-1},V_{s-1}\in
 \upnh $, we  choose matrices 
$P_s,Q_s$ in $\upnh$, equal to $\myId$ outside 
block $\{i_s,j_s\} $, such that the following 
Positivity Conditions hold: 
\begin{enumerate} 
\item 
For some nonnegative integer $M_s$, 
every entry of $P_s\{i_s,j_s\}$ and 
every entry of $Q_s\{ i_s,j_s\}$ equals 
$M_{s} \delta_i \delta_{ij} \delta_j$.  
\item 
The blocks 
$(P_sU_{s-1})\{i_s,j_s\}$ and 
$(V_{s-1}Q_s)\{ i_s,j_s\}$ have all entries in $\Z_+H_{ij}$. 
\end{enumerate} 
We note that  by taking $M_s$ large in (1), we can achieve (2). 
We then define matrices $W_s, X_s$ in 
$\upnh $, equal to $\myId$ outside block $\{i_s,j_s\}$, 
by setting 
\begin{align*} 
W_s \{i_s,j_s\} &= (P_sU_{s-1}) \{i_s,j_s\} \\ 
X_s \{i_s,j_s\} &= (V_{s-1}Q_s) \{i_s,j_s\}  \ . 
\end{align*} 
Finally we define 
\begin{align*} 
U_s &= P_s U_{s-1} W_s^{-1}
\qquad  
B_s = W_sB_{s-1}X_s 
\\
V_s &= X_s^{-1} V_{s-1} Q_s 
\qquad  
B'_s = P_sB'_{s-1}Q_s 
\end{align*} 
Then $U_s,V_s\in\upnh$ and 
$B_s, B'_s\in \mpppcnh$,  $0\leq s\leq r$. 

For $1\leq s \leq r$, we will verify the following 
claims by induction. 
\begin{align*} 
 (a) &\ \  U_s B_sV_s  = B'_s  \\
 (b) &\ \  
\begin{CD}
B'_{s-1} @>(P_s,Q_s)>+> B'_s  
\end{CD}
    \\
 (c) &\ \      
\begin{CD}
B_{s-1} @>(W_s,X_s)>+> B_s  
\end{CD}
\\
 (d) &\ \     
U_s = P_s \cdots P_1 U W_1^{-1}\cdots W_s^{-1} \quad \text{ and } \\ 
&\ \ 
V_s = X_s^{-1}\cdots X_1^{-1}V Q_1 \cdots Q_s 
 \\
 (e) &\ \      
U_s\{i_t,j_t\} = 0= V_s\{i_t,j_t\} \quad 
\text{ if } 1\leq t \leq s \  . 
\end{align*} 

Before proving $(a)$-$(e)$, suppose all these claims hold. 
Define $P=P_r P_{r-1}\cdots P_1$ and 
$Q= Q_1Q_2 \cdots Q_r$. 
From $(b)$, we have 
\[
\begin{CD}
B'= B'_{0} @>(P_1,Q_1)>+> B'_1  
@>(P_2,Q_2)>+> B'_2 
\cdots    
@>(P_r,Q_r)>+> B'_r = PB'Q  
\end{CD}
\]
and therefore 
$
\begin{CD}
B' @>(P,Q)>+> PB'Q  
\end{CD} 
$. 
Similarly, let $W=W_r\cdots W_1$ and $X= X_1\cdots X_r$. 
From $(c)$ we have 
\[
\begin{CD}
B= B_{0} @>(W_1,X_1)>+> B_1  
@>(W_2,X_2)>+> B_2 
\cdots    
@>(W_r,X_r)>+> B_r  =WBX 
\end{CD}
\]
and therefore 
$
\begin{CD}
B @>(W,X)>+> WBX   
\end{CD} 
$.
Because $\{U_r,V_r\} \subset \upnh$, 
from $(e)$ with $s=r$ we get 
$U_r=\myId=V_r$. 
Using $(d)$ at $s=r$, we then get 
\begin{align*} 
\myId\ =\ U_r\ =&\ P_r \cdots P_1 U W_1^{-1} \cdots W_r^{-1} 
\ =\ PUW^{-1} 
\end{align*} 
and similarly 
$
\myId=V_r = X^{-1}VQ  
$.
Therefore $(PU,VQ)= (W,X)$ and 
\[
\begin{CD}
B @>(PU,VQ)>+> PUBVQ  = B'Q 
@>(P^{-1}U,VQ^{-1})>+> B' 
\end{CD} \ \ .
\]
This shows $(U,V): B\to B'$ is a positive equivalence. 

To finish the proof it remains to verify $(a)$-$(e)$ 
for $1\leq s\leq r$. 

{\it Proof of (a).} 
We have $U_0B_0V_0 =B'_0$. Suppose $0< s \leq r$ and 
$(a)$ holds at $s-1$. Then 
\begin{align*} 
U_sB_sV_s\ =& \ 
\Big( P_sU_{s-1}W_s^{-1} \Big)
\Big( W_sB_{s-1}X_s \Big)
\Big( X_s^{-1}V_{s-1}Q_s\Big) \\ 
=&\ P_s \Big( U_{s-1}B_{s-1}V_{s-1} \Big) Q_s\\
=&\ P_s B'_{s-1}Q_s \ 
= \ B'_s \ .
\end{align*} 

{\it Proof of (b).} We have $B'_0 = B' \in \mpppcnh$. 
Suppose  $1\leq s \leq r$ and  
$B'_{s-1}  \in \mpppcnh$.
An entry in $B'_{s-1}$ could decrease in
  $P_sB_{s-1}Q_s$ only as a result of addition of a term
  $(M_s\delta_i \delta_{ij} \delta_j)(g_j-1)$ or
  $(g_i-1)(M_s\delta_i \delta_{ij} \delta_j)$. But, these terms vanish,
  as $(g_i-1)\delta_i =0=  \delta_j(g_j-1)$. Therefore
  $P_sB_{s-1}Q_s \geq B_{s-1}$. Because $B_{s-1}\in \mpppcnh$,
  this implies $P_sB_{s-1}Q_s \in \mpppcnh$.

Now enumerate the coordinates of the nonzero off-diagonal entries of
$Q_s$ as 
$(a_1, b_1), \dots , (a_T, b_T)$. For $1\leq t \leq T$, let $E_t$ 
be the basic elementary matrix such that $E_t(a_t, b_t)
=Q_s(a_t,b_t)$. Because these entries lie in blocks $\{i_s,j\}$ 
with $i_s\prec j$, we have $Q_s= \prod_{t=1}^TE_t$. This 
$(\myId,Q): B'_{s-1} \to B'_{s-1}Q$ is a composition of equivalences 
\[
\begin{CD}
B'_{s-1}:= B'_{{s-1},0} @>(\myId,E_1)>> B'_{{s-1},1}  
@>(\myId,E_2)>> 
\cdots    
@>(\myId, E_T)>> B'_{{s-1},T} = B'_{s-1}Q_s  
\end{CD}
\]
By induction, for $1\leq t \leq T$,  $E_t$ is nonnegative
   and $B'_{s-1,t} \in \mpppcnh  $. 
It then  follows from 
  Lemma \ref{2by2} that each $(\myId,E_t)$ 
gives a positive equivalence. Thus 
$
\begin{CD}
B'_{s-1} @>(\myId,Q_s )>+> B'_{s-1}  Q_s
\end{CD}
$, and similarly 
$
\begin{CD}
B'_{s-1}Q_s @>(P_s,\myId)>+> PB'_{s-1}  Q_s
\end{CD}
$. By composition, 
$
\begin{CD}
B'_{s-1}Q @>(P_s,Q_s)>+> P_sB'_{s-1}  Q_s
\end{CD} 
$.

{\it Proof of (c).} 
We have $B_0 \in \mpppcnh$. 
Now suppose 
$1\leq s \leq r$ and 
$B_{s-1}\in \mpppcnh$. 
The matrices $W_s$ and $X_s$ are 
in $\upnh $, with all 
entries in $\Z_+G$, and 
$ B_{s-1} \leq W_s B_{s-1} X_s := B_s$. 
Therefore $B_s \in  \mpppcnh$. 
An argument very similar to the 
proof of claim (b) now shows that 
$
\begin{CD}
B'_{s-1} @>(W_s,X_s)>+> B_{s} 
\end{CD} 
$.

{\it Proof of (d).} The claim (d) follows by induction from the 
definitions $U_0=U$, $V_0=V$, 
$U_s = P_s U_{s-1} W_s^{-1}$ and 
$V_s = X_s^{-1} V_{s-1} Q_s$. 
 
{\it Proof of (e).} 
Suppose $1\leq s \leq r$ and $(e)$ holds at $s-1$. 
(At $s-1= 0$, $(e)$ is an empty statement.) We have 
$U_s = P_sU_{s-1}W_s^{-1}$, with $P_s$ and $W_s^{-1}$ equal to 
$\myId$ outside block $\{ i_s, j_s \}$. 
On account of the zero block structure of 
matrices in $\upnh$, we have $U_s = U_{s-1}$ except possibly 
in blocks $\{ i,j\}$ such that $i\preceq i_s$
 or  $j\preceq j_s$. 

At  $(i_s,j_s)$, we have 
\begin{align*} 
U_s\{ i_s,j_s \} = 
& \ \Big( (P_sU_{s-1}) W_s^{-1} \Big) \{ i_s,j_s\} \\ 
=& \ (P_sU_{s-1}) \{ i_s,j_s\}  + W_s^{-1}  \{ i_s,j_s\}\\ 
=& \ (P_sU_{s-1}) \{ i_s,j_s\}  - W_s  \{ i_s,j_s\} \ = \ 0\ . 
\end{align*} 

Now suppose $i\prec i_s$. Then $U_s\{i,j\} =U_{s-1}\{i,j\} $ 
except possibly in the case $j=j_s$, where 
\begin{equation} \label{possiblediff}
U_s\{i,j_s\} -U_{s-1}\{i,j_s\} \  
=\  U_{s-1}\{i,i_s\}  W_s^{-1}\{ i_s,j_s\} \ . 
\end{equation} 
The right side of \eqref{possiblediff} can be nonzero 
only if $i\prec i_s \prec j_s = j$. In this case, 
$\rho (i_s,j_s) < \rho (i,j)$, so $(i,j)$ cannot 
equal $(i_t, j_t)$ for any $t$ less than $s$. Thus if  
$1\leq t<s$, 
then $U_s\{i_t,j_t\} =U_{s-1}\{i_t,j_t\} $, 
which is zero by the induction hypothesis.  

The analogous argument for the case 
$j_s \prec j$ finishes the proof. 
\end{proof}

We make contact to the case with $U,V\in \uph$ from the general case using the following lemma in combination with a key result from \cite{BSullivan}. 

\begin{lemma} \label{generallemma}
Suppose $i\notin \mathcal C$, $E$ is a basic 
elementary matrix in $\elph$;
$E\{j,k\}=\myId\{j,k\}$ when 
$(j,k)\neq (i,i)$; 
$B,B'\in 
\mppph$;  
and 
\[
(E\{i,i\},\myId)\colon
B\{i,i\} 
 \xrightarrow[+]{}
B'\{i,i\} \ . 
\]
Then there exists $V$ in $\upnh$
such that 
\[
(E,V)\colon B \xrightarrow[+]{} EB'V \ .
\]
Similarly, if 
\[
(\myId,E\{i,i\})\colon
B\{i,i\} 
 \xrightarrow[+]{}
B'\{i,i\} \  
\]
then there exists $U$ in $\upnh$
such that 
\[
(U,E)\colon B \xrightarrow[+]{} UB'E \ .
\]
\end{lemma}
\begin{proof}
We will consider the equivalence $(E,\myId)$; the other 
case is similar. Let $E(s,t)=v$ be the nonzero 
off-diagonal entry of $E$.  
$E$ acts on $B$ from  the left to add $v$ times row $t$ of $B$
to row $s$ of $B$.  
If each block $\{i, \ell\}$ of $EB$ is $H_{i\ell }$-positive (e.g., 
if $v\geq 0$), then set 
$V=\myId$. 

Otherwise, pick $r$ an index for a column through 
the $\{i,i\}$ block.  For a  positive 
integer $L$, let $V$ be the matrix in 
$\upnh$ such that (i) if  $i\prec\ell$, then 
every entry of $V\{i,\ell\}$ equals $L\delta_{i\ell}$ and 
(ii) in other entries, $V$ agrees with $\myId$. 
Then for $(s,q)$ in block $\{i, \ell \}$, 
\[
(EBV)(s,q) \geq (EB)(s,q) + (B(s,r)-B(t,r))(L\delta_{i\ell}) \ .
\]
Because $(EB)\{ i,i \}$ is $H_i$-positive, for sufficiently 
large $L$ the displayed sum must for each 
such $\ell$ be $H_{i\ell}$-positive. Then 
$B\xrightarrow[+]{(\myId,V)} BV$ and  
$EBV\in \mpppcnh$, so 
\[
B\xrightarrow[+]{(\myId,V)} BV
\xrightarrow[+]{(E,\myId)} EBV 
\] 
as required. 
\end{proof}

\begin{proof}[Proof of Theorem \ref{theoremfactor}] 
It follows from Observation \ref{observation} and 
Proposition \ref{gettingpositive} that 
to prove Theorem \ref{theoremfactor} we may assume that $B,B'\in \mppph$.

Thus let $(U,V)\colon B\to B'$ be the 
given $\elpnh$
equivalence, with $B,B' \in \mppph$. 
Set $U' =\oplus_i U\{i,i\}$ and $V' =\oplus_i V\{i,i\}$. 
If $i\in \mathcal C$, then $n_i=1$ and  $U\{i,i\}=(1)=V\{i,i\} $. 
If $i\notin \mathcal C$, then
by \cite[Theorem 6.1]{BSullivan} we have that  
\[
(U\{i,i\},V\{i,i\})\colon \ 
B\{i,i\} \to 
B'\{i,i\} 
\] 
is a positive $\Z H_i$-equivalence through matrices 
which are $H_i$-positive. 
So, there is  a string 
$(E_1,F_1),\dots ,(E_T,F_T)$    
of elementary 
$\elpnh$-equivalences
which accomplishes 
the elementary positive equivalence decomposition inside 
the diagonal blocks, such that each 
$E_t$ and $F_t$ equals the identity outside 
diagonal blocks $\{ i , i\}$ with $i\notin \mathcal C$. 
By Lemma 
\ref{generallemma}, we may find $(U_1,V_1),\dots ,(U_t,V_t)$ 
with each $U_s$ and $V_s$ in $\upnh$
such that 
\[
B\xrightarrow[+]{(U_1,F_1)} \mbox{}
\xrightarrow[+]{(E_1,V_1)} \cdots  
\xrightarrow[+]{(U_t,F_t)} \mbox{}
\xrightarrow[+]{(E_t,V_t)} 
B^* \ . 
\]
Let $X= E_tU_t\cdots E_2U_2E_1U_1$. 
Let $Y= F_1V_1F_2V_2\cdots F_tV_t$. 
Then for all $i$ in $\mathcal P$, 
$X\{i,i\} = U\{i,i\}$ and 
$Y\{i,i\} = V\{i,i\}$, so 
$UX^{-1}\in \upnh$ and 
$Y^{-1}V\in \upnh$. 
It then follows from 
Lemma \ref{unipotentlemma} that 
\[
B^*
\xrightarrow[+]
{(UX^{-1},Y^{-1}V)} B' \ . 
\] 
Thus $(U,V)\colon B\to B'$ is the composition 
\[
B
\xrightarrow[+]
{(X,Y)} B^* 
\xrightarrow[+]
{(UX^{-1},Y^{-1}V)} B' 
\] 
and therefore $(U,V): B \xrightarrow[+]{ \ } B'$. 
\end{proof}


\section{Conclusion of proofs}\label{proof} 

We begin with the promised proof of Theorem \ref{classification}. 

$(1)\implies (2)$: This implication follows directly from Theorem \ref{necessary}.

$(2)\implies (1)$: Suppose (2) holds. Applying first Proposition \ref{gettingpositive} if needed, it follows from Theorem \ref{theoremfactor} that there is a positive $\elpmh$-equivalence from $I-A^{<0>}$ to $I-C^{<0>}$. There is therefore a positive $\mathbb{Z}G$-equivalence from $I-A^{<0>}$ to $I-C^{<0>}$. Since every positive $\mathbb{Z}G$-equivalence induces a $G$-flow equivalence (see Section \ref{backgroundsec}), it follows that $T_{A^{<0>}}$ and $T_{C^{<0>}}$ are $G$-flow equivalent. Since $T_{A^{<0>}}=T_A$ and $T_{C^{<0>}}=T_C$, we thus have that $T_A$ and $T_C$ are $G$-flow equivalent. It follows in a similar way from Proposition \ref{cohomologyaspositive} and Proposition \ref{permutationprop} that $T_B$ and $T_C$ are $G$-flow equivalent. Thus, we have that $T_A$ and $T_B$ are $G$-flow equivalent as wanted.
\qed

Next we describe which equivalence classes of 
matrices arise  in the equivalence  classes we 
use as $G$-flow equivalence invariants.   
(The invariance of these classes under 
stabilization was discussed 
in Section \ref{subsec:stabilization}.)  
\begin{theorem}\label{thm:range} Given $G$, $\mathcal P$, $\mathcal C$,
$\mathcal
  H$, and $\mathbf n$ with $n_i =1$ if and only if $i\in \mathcal C$, 
suppose $B$ is a matrix  in $\mpnh$. Then the following are 
equivalent. 
\begin{enumerate} 
\item 
There is a $\mathbf k \geq \mathbf n$, with $k_i=1$ if and only if 
$i \in\mathcal C $, and a matrix $A$ in $\mopckh$, such that 
$I-A$ is $\elpkh$ to $I-B^{<0>}$, where 
$B^{<0>}$ is the 0-stabilization of $B$ in 
$\mpkh$.
\item The following hold: 
\begin{enumerate} 
\item 
 If $i\in \mathcal C$, then the $1\times 1$ $i$th diagonal 
block of $B$ has the form $[g]$, with $H_i $ generated by $g$. 
\item 
If $i\prec j$, $D\in \mathcal R_{ij}$ and $\{i,j\} \subset \mathcal
C$ and the $1\times 1$ $ij$ block of $B$ is $\sum_{g\in G} n_g
g$, with each $n_g$ in $\Z$, then $\sum_{g\in D} n_g>0$. 
\end{enumerate}
\end{enumerate}
Moreover, given (2), the matrix $A$ can be chosen from $\mpppcmh$, 
where $m_i=1$ if $i\in \mathcal C$ and $m_i =n_i + 1 $ otherwise. 
\end{theorem}
\begin{proof} 
  (2)$\implies$(1):
      With $\mathbf m$ as defined in the \lq\lq
Moreover\rq\rq\ statement, 
let $B'$ be the stabilization of $B$ in $\mopcmh$.
  Let $M= B'-I$,   with diagonal 
blocks $M_i$. It suffices to apply $\elpmh$
equivalences 
to $M$ which produce a matrix  
in $\mpppcmh$ (recall Definition \ref{defnmpppcnh}).  
By  
\cite[Proposition 8.8]{BSullivan}, for $i$ not in $\mathcal C$ the 
matrix $M_i$ 
is $\text{El} (\Z H_i)$-equivalent to an $H_i$-positive
matrix, $M'_i$. After applying a block diagonal $\elpmh$ 
equivalence, we may assume for $i\notin \mathcal C$ that 
$M_i$ is $H_i$-positive. For these $i$, in increasing order: 
for $i\prec   j$, as needed multiply from the right by matrices 
in $\elpmh$ zero outside the $ij$ block to put all entries 
of the $ij$ block of $M$ into $\Z \mathcal H_{ij}$, with 
strictly positive coefficients. Then similarly for $j$ in 
decreasing order: for $i\prec j$, as needed 
 multiply from the left to achieve this positivity. 

At this point, all blocks of $M$ are in form for 
$\mopcmh$ except perhaps the $1\times 1$ $ij$ blocks 
with $\{i,j\}\subset \mathcal C$. 
First, for each $ D \in \mathcal R_{ij}$:  pick an element $x$ from $D$, 
and multiply from the 
left and right by basic elementary matrices, 
of the form $(g_i x g_j)$ in the $ij$ block, to effect the replacement 
of $\sum_{g\in D} n_g g$ with $(\sum_{g\in D} n_g)x$, which by 
(2)(b) is positive.  
For $D$ not in $\mathcal R_{ij}$, the coefficients of 
$\sum_{g\in D} n_g g$ are made positive by elementary
 multiplications  as in the $(i,j,D)\notin \mathcal
 R^{\mathcal C}$ step in the proof of Proposition
 \ref{gettingpositive}. We will refrain from reentering the 
details of this step.
 
(1)$\implies$(2):  
Suppose (1) holds. 
Condition (2) holds with $A$ in place of $B$, 
because $A \in \mopckh$ with $k_i=1$ for $i \in \mathcal C$.  
Let $I -B^{<0>} =U(I-A)V  $ be the assumed 
$\elpkh$-equivalence. For $i\in \mathcal C$, letting   
$A \{i,i\} = (g_i)$, we have 
\begin{align*}
(I -B) \{i,i\} &= 
(I -B^{<0>}) \{i,i\}=U \{i,i\} (I-A) \{i,i\} V \{i,i\} \\  
&= \big( (1)(1-g_i)(1) \big) \ . 
\end{align*} 
Therefore 
(2a) holds for $B$.  
Given $\{i,j\} \subset \mathcal C$, let $a,b,u,v$ denote 
the entries of the singleton $\{i,j\}$ subblocks of $A,B,U,V$. 
For $D\in \mathcal R_{ij}$, 
\begin{align*}
\pi_D \big( (1-b) \big) &=
\pi_D \big( (I -B^{<0>}) \{i,j\} \big) \\
&=\pi_D \big( U \{i,i\} (I-A) \{i,i\} V \{i,i\} \big) \\
&= 
\pi_D \big( (1-a) + u(1-g_j) + (1-g_i)v \big) \ . 
\end{align*} 
Clearly $\pi_D( u(1-g_j)) =0= \pi_D ((1-g_i)v) $, and therefore 
$\pi_D (1-b) = \pi_D(1-a)$, and therefore (2b) holds for $B$.
\end{proof} 

\begin{remark} 
  In Theorem \ref{thm:range}, 
  $I-B^{<0>}$ is a 1-stabilization 
of the matrix $L=I-B$. The realization can be stated in terms of 
1-stabilizations of a matrix $L$ by replacing 
\lq\lq $B$ has the form $g$\rq\rq\ in 2(a) with 
\lq\lq $L$ has the form $1-g$\rq\rq, and 
replacing $\sum_{g\in D} n_g>0$ with 
$\sum_{g\in D} n_g<0$ in 2(b). 
\end{remark}

Lastly, we prove a finiteness result. 
Given $G$ an abelian group, $\mathcal P = 
\{1, \dots , N\}$ a poset  and $A\in \mopG$, 
let $A_k$ be the $k$th diagonal block of $A$ and 
let $d(A)$ be the $N$-tuple $(\det (I-A_1), \dots , \det(I-A_N))$. 
Up to reordering, $d(A)$ is an invariant of $G$-flow equivalence 
of the $G$-SFT $T_A$ defined by $A$. 


\begin{theorem}\label{nonzerodet} 
Let $G$ be a finite abelian group , $\mathcal P=\{ 1, \dots , N\}
$  a poset, $\mathcal H$ a $(G,\mathcal P)$ coset structure,
 and $d=(d_1, \dots , d_N)$  an $N$-tuple  of elements of
$\ZG$ which are regular. 
Then there are only finitely many $\elph$-equivalence classes 
of matrices in the set $\mathfrak M(d):=\{ I-A: A \in \moph ,  d(A) = d \}$.
Consequently there are only finitely many flow equivalence classes 
of $G$-SFTs $T_A$ with $d(A)=d$. 
\end{theorem} 

\begin{proof} 
For $A$ in $\mathfrak M(d)$, the set $\mathcal C$ of cycle components 
must be empty, and  for 
 each $i$, the matrix $I-A_i$ is injective and 
the $\Z H_i$-module $\cok (I-A_i)$ has finite size, determined 
by $\det(I-A_i)$. 
We will use some facts from \cite[Section 9]{BSullivan}, which contains 
more detail.  A theorem of Fitting shows that if 
$I-A$ and $I-B$ are injective matrices over $\Z H_i$ with 
isomorphic cokernels, then there are $m,n$ such that 
$(I-A)\oplus I_m$ and $(I-B)\oplus I_n$  are $\GL( \Z H_i)$-equivalent
\cite[Lemma 9.1]{BSullivan}. Because 
$H_i$ is finite abelian, the group $\text{SK}_1(\Z H_i )$ 
is finite  \cite{Oliver}; then by 
\cite[Corollary 9.9]{BSullivan}, there are only finitely many 
$\El (\Z H_i)$-equivalence classes of matrices with 
determinant the regular element $d_i$.  
Given such choices for $1\leq i \leq N$, 
fix $A$ in $ \mopnh$ with diagonal blocks $I-A_i$ 
in the given $\El (\Z H_i )$ classes. 

 Suppose $B\in \moph$  with 
$I-A_i$ and $I-B_i$ are $\El (H_i)$-equivalent for each $i$.  
We first claim that 
$I-B$ is $\elph$-equivalent to a matrix in 
$ \mopnh$ with the same diagonal blocks as $I-A$. 
To show this, for each $i$ let 
$k(i), \ell(i) , m(i)$ be nonnegative integers 
such that   there are $U_i,V_i$ in $\El (m_i, H_i )$ 
such that 
\[
(I-A_i )\oplus I_{k(i)} \ =\ U_i \Big((I-B_i)\oplus I_{\ell (i) }\Big)
V_i \ . 
\]
 Let $\mathbf m = (m_1, \dots , m_N )$ 
and let $A',B'$ be the stabilizations of $A,B$ in $\mopmh$. 
Let $U=U_1 \oplus \cdots \oplus U_N$ and 
$V=V_1 \oplus \cdots \oplus V_N$. Set  
$I-C=U(I-B')V$. Then $I-C$ and $I-B'$ are $\elpmh$ 
equivalent and the $i$th diagonal block of 
$(I-C) $ equals $(I-A_i) \oplus I_{k(i)}$. 
After adding multiples of rows and columns from the 
$I_{k(i)}$, we may produce a matrix $I-D$, 
$\elpmh$-equivalent to $I-B$,  such that $D$ 
is zero outside its principal submatrix ($P$, say) 
on the indices 
used to define $A$. Now $I-P$ is $\elph$-equivalent 
to $I-B$ and its diagonal blocks equal those of $I-A$. 

To finish, it suffices to show $I-P$ is $\elpnh$-equivalent 
to a matrix with bounded entries. 
The $i$th diagonal block of $I-P$ is the 
 $n_i\times n_i$ matrix $I-A_i$.   
Let $\mathcal R_i$ be the image of the space of row vectors 
$(\Z H_i)^{n_i}$ under the map $v\mapsto v(I-A_i)$. 
Let $\kappa_i \in \N$  be the index of $\mathcal R_i$ 
in $(\Z H_i)^{n_i}$. Then 
$\mathcal R_i$  contains 
$\kappa_i(\Z H_i)^{n_i}$.  In the order $j=2,3, \dots , N$  do the 
following: for $i\prec j$, as needed, multiply
$I-P$  from the left by matrices of 
$\elpnh$ which are equal to
$I$ 
outside the $ij$th block  to reduce all $\Z$ coefficients in that 
block to lie in the interval $[0, \kappa_j)$. This shows $I-P$ is
$\elpnh$-equivalent to one of a bounded set of matrices, 
as required. 
 \end{proof} 

\appendix

\section{Cohomology as positive equivalence} \label{cohomologyappendix}

The next proposition was proved in \cite{BSullivan}, 
with (much) worse control over $\mathbf m$,  
using the positive K-theory polynomial strong shift 
equivalence equations from \cite{BoyleWagoner}.
 The elementary argument 
below gives a better bound on $\mathbf m$; and for
the proof of Theorem \ref{necessary}, 
we use the case where $m_i$ is controlled to be $n_i$. 
The identity element of $G$ is denoted $e$. 
 
\begin{proposition} \label{cohomologyaspositive} 
Suppose $D$ is an $n\times n$  diagonal matrix over $\ZZ_+G$ such that 
for each $s$, $D(s,s) = g_s\in G$. Suppose  
$A$ is an $n\times n$ matrix over $\Z_+$ and $B=D^{-1}AD$. 
Then there is an $m\leq n+1 $ 
and $m\times m$ stabilizations $A',B'$ of $A,B$
such that there is a positive $\Z G$ 
equivalence  $\rep{I-A'}\to \rep{I-B'}$. 
 
Now suppose in addition  that 
$A\in \mopcnh$ and for all $i$ in $\mathcal P$ 
that $g_s\in H_i$ whenever $s\in \mathcal I_i$. 
Then $B\in  \mopcnh$,  and there are $\mathbf m$ 
and stabilizations $A',B'$ of $A,B$ 
in $\mopcmh$ such that 
there is a positive $\elpmh$-equivalence
$\rep{I-A'}\to \rep{I-B'}$. The vector $\mathbf m$ can be 
chosen such that for all $i\in \mathcal P$, 
\begin{enumerate} 
\item 
$m_i\leq n_i +1$,  and 
\item if $g_s=e$ for all $s\in \mathcal I_i$ 
then $m_i=n_i$. 
\end{enumerate} 
\end{proposition} 

\begin{proof} 
In the second case, $B$ will be in $\mopcnh$ because 
$\mathcal H$ is a coset structure. 

We will describe given $s$ a positive $\Z G$-equivalence 
which has the effect of multiplying row $s$ from the 
left by $g_s^{-1}$ and multiplying column $s$ from the 
right  by $g_s$. The equivalence will satisfy the 
stabilization bounds and in the second case be 
a positive $\elph$-equivalence.  Applying such an 
equivalence for each $s$ proves the 
proposition. For concreteness, suppose $s=1$ and $g_1=g$.  

We will describe the equivalence as a finite sequence of 
 the row and column cuts 
from Section \ref{backgroundsec}. 
  We first consider the special case that
$A(1,1)=0$. The target matrix will be named $A'$. 
To lighten the notation (avoiding no technical difficulty), 
we  will suppose a nonzero entry of $A$ is a single element of $G$; 
e.g., $A(s,1)=a$ means an edge from vertex $s$ to vertex $1$ is 
labeled by $a$, as in the graph I below. 
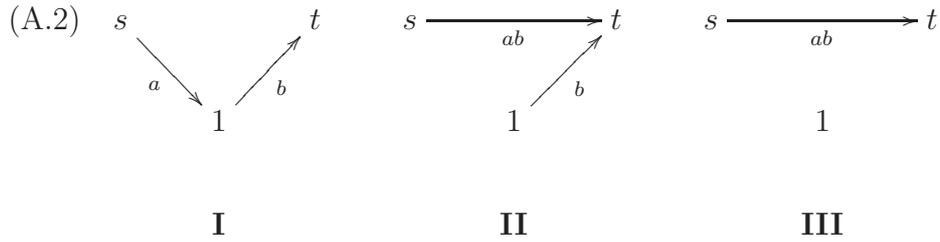
\begin{gather}\label{zerocase}
\xymatrix{
s \ar[dr]_a & & t  \\
&  1 \ar[ur]_b & \\ 
& \mathbf I & 
} 
\quad \quad 
\xymatrix{
s \ar[rr]_{ab} & & t  \\
&  1 \ar[ur]_b & \\ 
& \mathbf{II} & 
} 
\quad \quad 
\xymatrix{
s \ar[rr]_{ab} & & t  \\
&  1 &\\ 
& \mathbf{III} & 
} 
\end{gather} 

Multiplying $I-A$ from the left by $E_{s1}(a)$ effects the row cut of  the edge 
from $s$ to $1$ labeled $a$ in I and produces the change I$\to$II.
Do this for every edge into 1, producing a graph in which 
1 has no incoming edge. Then column cut every edge out of 1; the 
effect is to remove those edges, as in III, 
leaving 1 an isolated vertex.  
 Let $A''$ be the matrix produced from $A$ by these moves. 
Note that applying this procedure to 
the matrix $A'$
produces the same matrix $A''$: 
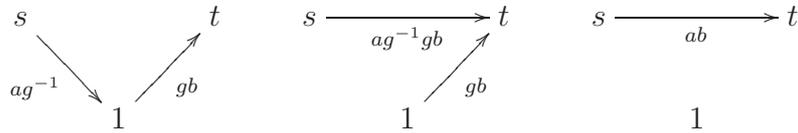
\begin{gather*}
\xymatrix{
s \ar[dr]_{ag^{-1}} & & t  \\
&  1 \ar[ur]_{gb} & 
} 
\quad \quad 
\xymatrix{
s \ar[rr]_{ag^{-1}gb} & & t  \\
&  1 \ar[ur]_{gb} &  
} 
\quad \quad 
\xymatrix{
s \ar[rr]_{ab} & & t  \\
&  1 &  
} 
\end{gather*} 

The positive equivalence for $A\to A''$ postcomposed with the inverse 
of the positive equivalence for $A'\to A''$ gives the required 
positive equivalence $A\to A'$ for this case. 

For the case that $A(1,1)=c\neq 0$, we introduce an additional 
isolated vertex, named $v$.  (If for some vertex $s$ there is 
no edge from $s$ to itself, then by applying the  
move corresponding to I$\to$III in \eqref{zerocase} we could 
isolate $s$ and avoid increasing the number of vertices.) 
 The argument again is described by
a finite sequence of evolving graphs. 

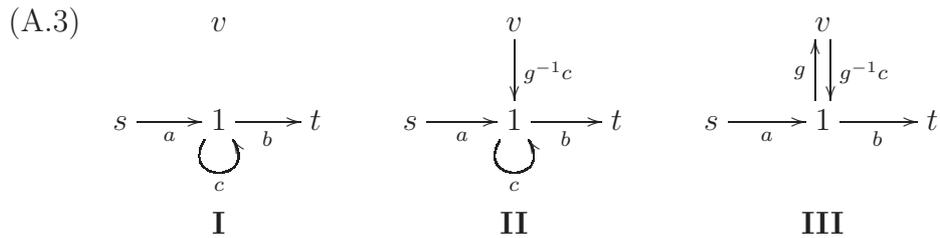
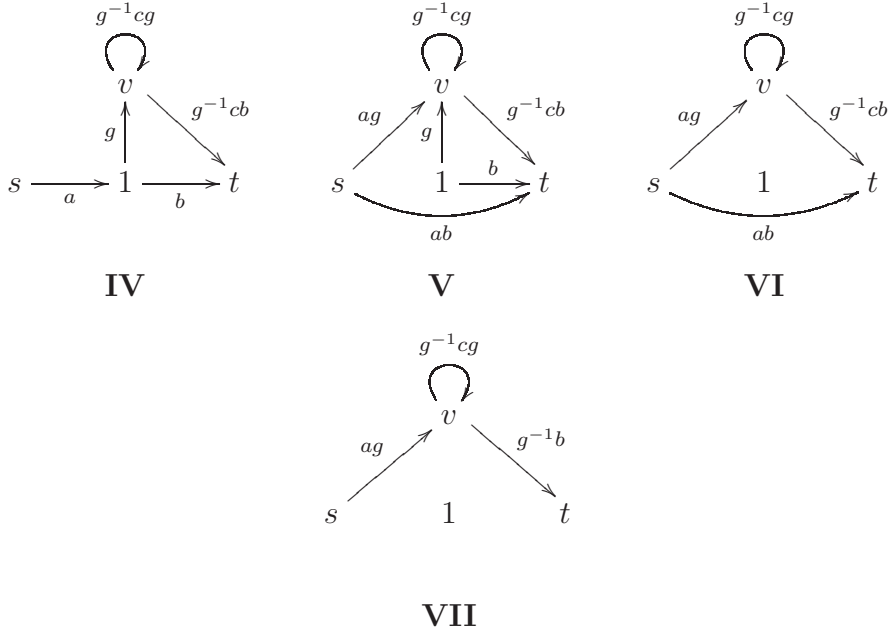
\begin{gather}\label{nonzerocase} 
\xymatrix{
& v & \\ 
s \ar[r]_a& 
  1 \ar@(ld,rd)[]_c
\ar[r]_b
& t \\ 
& \mathbf I & 
} 
\quad \quad 
\xymatrix{
& v\ar[d]^{g^{-1}c} & \\ 
s \ar[r]_a& 
  1 \ar@(ld,rd)[]_c
\ar[r]_b
& t \\ 
& \mathbf{II} & 
}
 \quad \quad 
\xymatrix{
& v
\ar@<+1mm>[d]^{g^{-1}c} 
& \\ 
s \ar[r]_a& 
  1 \ar@<+1mm>[u]^g 
\ar[r]_b
& t \\ 
& \mathbf{III} & 
} 
\end{gather} 

\begin{gather*}
\xymatrix{
& v
\ar@<+1mm>[rd]^{g^{-1}cb}
\ar@(lu,ru)[]^{g^{-1}cg}
& \\ 
s \ar[r]_a& 
  1 \ar[u]^g 
\ar[r]_b
& t \\ 
& \mathbf{IV} & 
} 
\quad \quad 
\xymatrix{
& v
\ar@<+1mm>[rd]^{g^{-1}cb}
\ar@(lu,ru)[]^{g^{-1}cg}
& \\ 
s  \ar@/_1pc/[rr]_{ab}
\ar[ur]^{ag} & 
  1 \ar[u]^g 
\ar[r]^b
& t \\ 
& \mathbf{V} & 
}
 \quad \quad 
\xymatrix{
& v
\ar@<+1mm>[rd]^{g^{-1}cb}
\ar@(lu,ru)[]^{g^{-1}cg}
& \\ 
s  \ar@/_1pc/[rr]_{ab}
\ar[ur]^{ag} & 
  1 
& t \\ 
& \mathbf{VI} & 
} 
\end{gather*} 
\[
\xymatrix{
& v
\ar@<+1mm>[rd]^{g^{-1}b}
\ar@(lu,ru)[]^{g^{-1}cg}
& \\ 
s 
\ar[ur]^{ag} & 
  1 
& t \\ 
& \mathbf{VII} & 
} 
\]

Here is a list of the corresponding positive equivalences. 
\begin{itemize} 
\item 
II$\to$I. Column cut the edge $v\to 1$. 
\item 
III$\to$II. Row cut the edge $1\to v$. 
\item 
III$\to$IV. Row cut the edge $v\to 1$. 
\item 
IV$\to$V. Row cut all incoming edges to $1$. 
\item 
V$\to$VI. Column cut all outgoing edges from $1$. 
\item 
VII$\to$VI. Row cut each outgoing edge from $v$ to 
a different vertex. 
\end{itemize} 
At this point, the move from I to VII in 
\eqref{nonzerocase} has replaced the given matrix $A$ with a 
matrix $A''$ which 
satisfies our conditions, except  
that the vertex $v$ is playing in $A''$ the role we require for vertex $1$. 
To remedy this, apply the procedure I$\to$VII above   
to $A''$, 
 but with 
$(s,t,v,1,e)$ in place of $(s,t,1,v,g)$. We end up with 
the required matrix $A'$ , with the additional isolated vertex $v$ 
(i.e., row $v$ and column $v$ of $A'$ are zero).  

\end{proof} 

\section{Permutation similarity as positive
  equivalence} \label{permutation appendix}

Suppose $A$ is an $n\times n$ matrix over $\Z_+G$ and $P$ 
is an $n\times n$ permutation matrix and $B=P^{-1}AP$. 
Then $A,B$ are elementary strong shift equivalent over 
$\Z_+G$, as $B= (P^{-1})(AP)$ and $ A = (AP)(P^{-1})$, 
and therefore $A$ and $B$ are $\Z G$ positive equivalent 
\cite{BoyleWagoner}. In the next proposition we show that we can obtain this positive equivalence through $(n+1)\times (n+1)$ matrices. We also show that if $A\in \mopnh$ and $P\in\mathcal M_{\mathcal P} (\mathbf n , \Z_+ G)$, then we get a positive $\elph$-equivalence
$\rep{I-A}\to \rep{I-B}$.

\begin{proposition} \label{permutationprop} 
Let $A,B,P,G$ be as above. Suppose there is
an index $s$ with $A(s,s)=0$. Then there is a positive 
$\Z G $-equivalence from $A$ to $B$ through $n\times n$ 
matrices. In any case there are  stabilizations $A',B'$  of 
$A,B$ which are positive $\Z G$-equivalent through 
$(n+1)\times (n+1)$ matrices. 

Now suppose in addition  that 
$A\in \mopcnh$ and $P\in\mathcal M_{\mathcal P} (\mathbf n , \Z_+ )$. 
Then $B\in  \mopcnh$,  and there are $\mathbf m$ 
and stabilizations $A',B'$ of $A,B$ 
in $\mopcmh$ such that 
there is a positive $\elpmh$-equivalence
$(U,V):\rep{I-A'}\to \rep{I-B'}$. The vector $\mathbf m$ can be 
chosen such that $m_i\le n_i+1$ for all $i\in \mathcal P$. If $P\{i,i\}=I$ where $i\in\mathcal P$, then $U$ and $V$ can be chosen such that $U\{i,i\}$ and $V\{i,i\}$ are the identity matrix.
\end{proposition} 

\begin{proof} 
Assume first that there is an index $s$ with $A(s,s)=0$. Let $t$ be an index different from $s$. We will describe a positive $\Z G$-equivalence which has the effect of permuting $s$ and $t$. If $t_1,t_2$ are arbitrary indexes, then we get a positive $\Z G$-equivalence which has the effect of permuting $t_1$ and $t_2$ by first permuting $s$ and $t_1$, then permuting $t_1$ and $t_2$, and then finally permuting $t_2$ and $s$. Since every permutation of $\{1,\dots,n\}$ is the product of transpositions, it will follow that there is a positive $\Z G$-equivalence $\rep{I-A}\to \rep{I-B}$. 

The procedure described in \eqref{zerocase} shows that there is a positive $\Z G$-equivalence $\rep{I-A}\to \rep{I-A_1}$ such that $s$ is an isolated index in $\mathcal G_{A_1}$. If also $A(t,t)=0$, then there is a positive $\Z G$-equivalence $\rep{I-A_1}\to \rep{I-A_2}$ such that $t$ is an isolated index in $\mathcal G_{A_2}$, and if we then postcompose the equivalence $\rep{I-A}\to \rep{I-A_2}$ with its inverse but with the role of $s$ and $t$ interchanged, then we get a positive $\Z G$-equivalence $\rep{I-A}\to \rep{I-A_{  3}}$ where $A_{  3}$ is the matrix obtained from $A$ by permuting $s$ and $t$. If $A(t,t)\ne 0$, then the procedure described in \eqref{nonzerocase} with $g=e$ shows that there is a positive $\Z G$-equivalence $\rep{I-A_1}\to \rep{I-A'_2}$ where $A'_2$ is obtained from $A_2$ by permuting $s$ and $t$. By postcomposing with the inverse of the equivalence $\rep{I-A}\to \rep{I-A_1}$ but with the role of $s$ and $t$ interchanged, we get a positive $\Z G$-equivalence $\rep{I-A}\to \rep{I-A'_3}$ where $A'_3$ is the matrix obtained from $A$ by permuting $s$ and $t$.  

If there is no index $s\in\mathcal G_A$ with $A(s,s)=0$, then we add a zero row and a zero column to $A$ and $B$ to obtain matrices $A'$ and $B'$, and then it follows from the argument above that there is a positive $\Z G$-equivalence $\rep{I-A'}\to \rep{I-B'}$.

Now suppose in addition  that 
$A\in \mopcnh$ and $P\in\mathcal M_{\mathcal P} (\mathbf n , \Z_+ G)$. 
Then $P\{i,j\}=0$ if $i\ne j$. It follows that $B\in  \mopcnh$. We let $P^*_i$ denote the matrix in $\mathcal M_{\mathcal P} (\mathbf n , \Z_+ G)$ such that $P^*_i\{i,i\}=P\{i,i\}$, $P^*_i\{j,j\}=I$ for $j\ne i$, and $P^*_i\{i',j'\}=0$ for $i'\ne j'$. Then $P=P^*_{i_1}P^*_{i_2}\cdots P^*_{i_N}$. Let $A_1=(P^*_{i_1})^{-1}AP^*_{i_1}$, $A_2=(P^*_{i_2})^{-1}A_1P^*_{i_2}$,\dots,$A_N=(P^*_{i_N})^{-1}A_{N-1}P^*_{i_N}=B$.
By
the first half of the proposition,  there are stabilizations $A',A'_1,A'_2,\dots,A'_N=B'$ of $A,A_1,A_2,\dots,A_N=B$ and
positive $\Z G$-equivalences
\[
\rep{I-A'}\to \rep{I-A'_1}\to \rep{I-A'_1}\to
\rep{I-A'_2}\to\dots\to \rep{I-A'_N}=\rep{I-B'}
\]
and that $B'$ can be chosen such that $B'\in\mopcmh$ with $m_i\le n_i+1$ for all $i\in \mathcal P$. It is not difficult to check that the described equivalence $(U,V):\rep{I-A'}\to \rep{I-B'}$ is a positive $\elpmh$-equivalence, and that if $P\{i,i\}=I$ where $i\in\mathcal P$, then $U$ and $V$ can be chosen such that $U\{i,i\}$ and $V\{i,i\}$ are the identity matrix.
\end{proof}

\section{Resolving extensions} \label{resolvingappendix}

\begin{proposition}  \label{labellift}
Suppose $A$ is a matrix in 
$\mopcnh$ and $A'$ is a matrix 
obtained from $A$ by  splitting  a row $s$ into 
two rows.
Let the rows of $A'$ be in 
the same order as corresponding rows of 
$A$, with the interpolation of a new row $s'$ 
directly following $s$. Let $A'$ have the 
natural $\mathcal P$ blocking: $s'$ 
is in the block of $s$, and every other
index is in the block of the row from 
which it was copied.  Let $\widetilde A$ be 
the matrix of size and blocking from $\mathbf{n'}$ 
 obtained by interpolating a zero $s'$ row and column 
into $A$.  

Then $\widetilde{A},A'\in\moprcnh$ and there is a positive $\elpnprimeh$-equivalence 
$(U,V): \rep{I-\widetilde A} \to \rep{I-A'}$.  If $A$ is 
upper triangular and $A(s,s)=A'(s,s)$, then $A'$ 
is upper triangular and 
the matrices $U,V$ can be chosen to be unipotent 
upper triangular. 

Moreover, the same conclusion holds 
if in the above statements
``row'' is replaced by  ``column'' and 
``following'' is replaced by 
``preceding''.
\end{proposition}

\begin{proof} 
  Let us first check that $A'\in\moprcnh$ (it is obvious that $\widetilde{A}\in\moprcnh$). It is easy to check that $A'\in\mopcnG\cap\mpnh$, so we just need to show that $\mathcal H$ is a $(G,\mathcal P)$ coset structure for $A'$. Since $\mathcal H$ is a $(G,\mathcal P)$ coset structure for $A$, there is a family of vertices $\{v(i)\}_{i\in\mathcal P}$ such that $v(i)$ belongs to the irreducible core of $A\{i,i\}$ for each $i\in\mathcal P$, and $H_{ij}$ is the set of weights of paths from $v(i)$ to $v(j)$ in $\mathcal G_A$. Let $i,j\in\mathcal P$. We aim to show that the set of weights of paths from $v(i)$ to $v(j)$ in $\mathcal G_{A'}$ is equal to $H_{ij}$.

  Notice that if $p$ is a path in $\mathcal G_A$ not starting at $s$, then there is a path in $\mathcal G_{A'}$ starting and ending at the same vertices as $p$ and with the same weight as $p$. Notice also that if $p$ is a path in $\mathcal G_A$ starting at $s$, then there is a path in $\mathcal G_{A'}$ starting at either $s$ or $s'$ and ending at the same vertex as $p$ and with the same weight as $p$. Similarly, if $p$ is a path in $\mathcal G_{A'}$, then there is a path in $\mathcal G_A$ which has the same weight as $p$ and which starts and ends at the same vertices as $p$ (except of course if $p$ starts/ends at $s'$ in which case the path in $\mathcal G_A$ starts/ends at $s$ instead). It follows that if $v(i)\ne s$, then the set of weights of paths from $v(i)$ to $v(j)$ in $\mathcal G_{A'}$ is equal to $H_{ij}$. 

Suppose that $v(i)= s$ and that $p$ is a path in $\mathcal G_A$ starting at $s$ and that there is a path in $\mathcal G_{A'}$ starting at $s'$ and ending at the same vertex as $p$ and with the same weight as $p$. Suppose that there is a path from $s$ to $s'$ in $\mathcal G_{A'}$ (if there is no path in $\mathcal G_{A'}$, then there must be a path from $s'$ to $s$ because $s=v(i)$ in the irreducible core of $A\{i,i\}$, and the we just interchange the role of $s$ and $s'$). Let $\gamma$ be the weight of this path. Since the set of weights of paths in $\mathcal G_{A'}$ from $s$ to $s$ is equal to the set of weights of paths in $\mathcal G_{A'}$ from $s$ to $s'$ and is a group (because it is a finite semigroup), it follows that there is a path in $\mathcal G_{A'}$ from $s$ to $s$ with weight $\gamma^{-1}$, and thus that there is there is a path in $\mathcal G_{A'}$ starting at $s$ and ending at the same vertex as $p$ and with the same weight as $p$. It follows that the set of weights of paths from $s=v(i)$ to $v(j)$ in $\mathcal G_{A'}$ is equal to $H_{ij}$. This shows that $\mathcal H$ is a $(G,\mathcal P)$ coset structure for $A'$ and thus that $A'\in\moprcnh$.
 
We then show that there is a positive $\elpnprimeh$-equivalence 
$(U,V): I-\widetilde A \to I-A'$.
We write $A$ in a $3\times 3$ block form, giving 
\[
A = 
\begin{pmatrix} 
q&r&t\\ 
u&v&w\\ 
x& y & z 
\end{pmatrix} 
\qquad 
\widetilde A = 
\begin{pmatrix} 
q&r&0&t\\ 
u&v&0&w\\ 
0&0&0&0\\
x& y &0& z 
\end{pmatrix} \ . 
\] 
(We omit the easier proof for the case that 
$s$ is a first or last index of $A$, and the block form is smaller.) 
The central index set of $A$ is  
$\{ s\} $, so $v$ is the $1\times 1$ matrix $A(s,s)$. 
The matrix $A'$ then has the 
block form 
\[
A' = 
\begin{pmatrix} 
q    &  r  &r     &t\\ 
u_1 &v_1 &v_1&w_1\\ 
u_2 &v_2 &v_2&w_2\\ 
x    &y   &  y   &   z 
\end{pmatrix} 
\]           
with $A'$ nonnegative and $(u_1,v_1,w_1)+(u_2,v_2,w_2) =(u,v,w)$. 
We then have a string of positive equivalences: 
\begin{align*}
I-\widetilde A \to E_1(I-\widetilde A) &= 
 \begin{pmatrix} 
1&0&0&0\\
0&1&-1&0\\
0&0&1&0\\
0&0&0&1 
\end{pmatrix} 
\begin{pmatrix} 
1-q&-r&0&-t\\ 
-u&1-v&0&-w\\ 
0 & 0& 1& 0 \\
-x& -y & 0&1- z 
\end{pmatrix} \\
&=
\begin{pmatrix} 
1-q&-r&0&-t\\ 
-u&1-v&-1&-w\\ 
0 & 0& 1& 0 \\
-x& -y & 0&1- z 
\end{pmatrix} 
:= I-A_1 \ . 
\end{align*}

\begin{align*} 
I-A_1 \to 
(I-A_1) E_2 &= 
\begin{pmatrix} 
1-q&-r&0&-t\\ 
-u&1-v&-1&-w\\ 
0 & 0& 1& 0 \\
-x& -y & 0&1- z 
\end{pmatrix} 
 \begin{pmatrix} 
1&0&0&0\\
0&1&0&0\\
-u_2&0&1&0\\
0&0&0&1 
\end{pmatrix} \\ 
&= 
\begin{pmatrix} 
1-q&-r&0&-t\\ 
-u_1&1-v&-1&-w\\ 
-u_2 & 0& 1& 0 \\
-x& -y & 0&1- z 
\end{pmatrix} 
:= I-A_2 
\end{align*}

\begin{align*} 
I-A_2 \to 
(I-A_2) E_3 &= 
\begin{pmatrix} 
1-q&-r&0&-t\\ 
-u_1&1-v&-1&-w\\ 
-u_2 & 0& 1& 0 \\
-x& -y & 0&1- z 
\end{pmatrix} 
 \begin{pmatrix} 
1&0&0&0\\
0&1&0&0\\
0&-v_2&1&0\\
0&0&0&1 
\end{pmatrix} \\ 
&= 
\begin{pmatrix} 
1-q&-r&0&-t\\ 
-u_1&1-v_1&-1&-w\\ 
-u_2 & -v_2& 1& 0 \\
-x& -y & 0&1- z 
\end{pmatrix} 
:= I-A_{  3} 
\end{align*}

\begin{align*} 
I-A_{  3} \to 
(I-A_{  3}) E_4 &= 
\begin{pmatrix} 
1-q&-r&0&-t\\ 
-u_1&1-v_1&-1&-w\\ 
-u_2 & -v_2& 1& 0 \\
-x& -y & 0&1- z 
\end{pmatrix} 
 \begin{pmatrix} 
1&0&0&0\\
0&1&0&0\\
0&0&1&-w_2\\
0&0&0&1 
\end{pmatrix} \\ 
&= 
\begin{pmatrix} 
1-q&-r&0&-t\\ 
-u_1&1-v_1&-1&-w_1\\ 
-u_2 & -v_2& 1& -w_2 \\
-x& -y & 0&1- z 
\end{pmatrix} 
:= I-A_4 
\end{align*}

\begin{align*} 
(I-A_4 ) \to 
(I-A_4 ) E_5 
&= 
\begin{pmatrix} 
1-q&-r&0&-t\\ 
-u_1&1-v_1&-1&-w_1\\ 
-u_2& -v_2& 1& -w_2 \\
-x& -y & 0&1- z 
\end{pmatrix} 
 \begin{pmatrix} 
1&0&0&0\\
0&1&1&0\\
0&0&1&0\\
0&0&0&1 
\end{pmatrix} 
\\
&= 
\begin{pmatrix} 
1-q&-r&-r&-t\\ 
-u_1&1-v_1&-v_1&-w_1\\ 
-u_2 & -v_2& 1-v_2 & -w_2 \\
-x& -y & -y&1- z 
\end{pmatrix} = I-A'\ . 
\end{align*} 
This exhibits the equivalence
$(U,V): I-\widetilde{A} \to I-A'$, 
with 
$U=E_1$ and $V=E_2E_3E_4E_5$. It is clear that $E_1,E_2,E_3,E_4,E_5\in \elpnprimeh$, and it is not difficult to check that $A_1,A_2,A_{  3},A_4\in\moprnh$. It follows that $(U,V): I-\widetilde{A} \to I-A'$ is a positive $\elpnprimeh$-equivalence.  

In general the matrices 
$U,V$ will not be upper triangular, because 
in general $E_2$ and $E_3$ are not upper triangular. However, 
if $A$ is upper triangular, then $u=0$, so $u_1=u_2=0$;  
and if $A'(s,s)=A(s,s)$, then $v_1=v$, so $v_2=v-v_1=0$. 
Thus under the additional assumptions, $A'$ is upper 
triangular and the matrices 
$U=E_1$ and $V=E_2E_3E_4E_5$ are unipotent upper triangular as required. 

The argument for the ``Moreover'' claim is essentially the same, 
and we omit it. 
\end{proof}

\section{$G$-SFTs following Adler-Kitchens-Marcus}
\label{akmappendix}

The purpose of this  appendix is to integrate our
matrix approach with  the classification of $G$-SFTs with 
the group actions framework
of Adler-Kitchens-Marcus \cite{akmgroup,akmfactor}.
 This appendix is not necessary for the statements 
or proofs of the flow equivalence results 
of earlier sections. Throughout, $G$ is a 
finite group. 

We make no conceptual advance on \cite{akmgroup}
(which in turn acknowledges a huge debt to the 
 ergodic-theoretic work \cite{rudolphcounts} of Rudolph).  
Still,  we give a detailed presentation 
of this framework (with some additional details),
as the ideas in \cite{akmgroup}
are intermingled with that paper's focus on almost topological
conjugacy, and the paper does not isolate all the explicit
statements we want. The paper \cite{parrynotes} of Parry also
covers much, but not all, of this framework. 

A matrix $A$ is  
\emph{$G$-primitive} if its entries 
lie in $\Z_+G$ and
in addition there is a positive integer $m$ such that every entry of 
$A^m$ is $ G$-positive.
An SFT is {\it nonwandering} if it has no wandering orbit;
equivalently, it 
is the disjoint union of finitely many irreducible SFTs 
(its {\it irreducible components}).  
A nonwandering/irreducible/mixing $G$-SFT is a 
$G$-SFT $(Y,T)$ which as an SFT is 
 nonwandering/irreducible/mixing. 
A nonwandering $G$-SFT was defined to be 
{\it $G$-transitive} \cite[Section 4]{akmfactor} if the $G$ action on 
irreducible components is transitive. A 
 nonwandering $G$-SFT is 
$G$-transitive if and only if the canonical factor map 
collapsing $G$-orbits maps each irreducible component 
onto the same irreducible SFT. Clearly the classification of nonwandering
$G$-SFTs reduces to the classification of $G$-transitive nonwandering
$G$-SFTs. 

 Let $G$ be a finite group and 
let $(Y,T)$ be a nonwandering $G$-transitive $G$-SFT. 
We take this left $G$-SFT $(Y,T) $ to be  
$Y= X\times G$ with $T: (x,g) \to (\sigma x, g \tau(x))$ 
with $\tau: X\to G$  continuous, 
and left $G$ action by $g: (x,h) \to (x,gh)$, 
as in Section \ref{backgroundsec} (recall our Standing Convention \ref{leftconvention}).   
Let  $C$ be an irreducible component of $Y$, with 
cyclically moving subsets 
$C^0, \dots , C^{p-1}$ . 
For $g\in G$, let 
$gC:=\{(x,gh): (x,h) \in C\}$. Then $gC$ is an  irreducible 
component of $Y$. 
The map $(x,h)\mapsto (x,gh)$ sending $C$ 
to $gC$ is a topological conjugacy of
SFTs  (but not of $G$-SFTs, when $G$ is not abelian).
The {\it stabilizer} of $C$ is 
the subgroup $H_C= H= \{ g\in G: gC=C\}$. For $g\in G$, we have 
$H_{gC} = gH_Cg^{-1}$.

The next result explains how to 
reduce the classification of nonwandering $G$-SFTs to the 
classification of irreducible $K$-SFTs, for
subgroups $K$ 
of $G$. 

\begin{theorem} \label{reductiontoirreducible}  
Suppose $G$ is a finite group and $(Y,T)$ and $(Y',T')$ are nonwandering
$G$-transitive $G$-SFTs, containing irreducible components $C,C'$ 
(resp.).
Let $H$ denote the stabilizer $H_C$  of $C$. 
Then the following are equivalent. 
\begin{enumerate} 
\item 
$(Y,T)$ and $(Y',T')$ are $G$-conjugate. 
\item 
There exists $g$ in $G$ such that 
the stabilizer of $gC'$ is $H$
(i.e., the stabilizers of $C$ and $C'$ are conjugate subgroups in $G$)  
and the irreducible $H$-SFTs $C$ and $gC'$ are 
$H$-conjugate.  
\end{enumerate} 
\end{theorem}

\begin{proof}[Proof of Theorem \ref{reductiontoirreducible}]
$(1)\implies (2)$: A $G$-conjugacy $(Y,T)\to (Y',T')$ 
will restrict to 
an $H$-conjugacy from  $C$ 
 to some irreducible  component $D$ of 
$Y'$ with stabilizer $H_D = H$. By the $G$-transitivity, 
there is some $g\in G$ such that $D=gC'$.  

$(2) \implies (1)$: 

Let $D=gC'$ and let $\phi_C: C\to D$ be a conjugacy of $H$-SFTs.
Pick elements $g_i$ of $G$, $0\leq i < |G/H|$, such that
$g_0=e$ and 
the $g_iH$ are the distinct left cosets of $H$ in $G$.
Define $\phi: G\to G$ by setting $\phi (g_i x) = g_i  \phi_C (x) $
for $g_i x $ in $g_iC$. Then for $x \in C$ and $g=g_ih\in g_iH$,
we have
\[
\phi (gx) = \phi (g_ihx) = g_i\phi_C (hx)
= g_ih\phi_C(x) = g\phi (x) \ .
\]
For $g_ix\in g_iC$ and $g\in G$,
it follows that $\phi (g(g_ix)) =
gg_i \phi (x) =g \phi (g_ix)$. Thus, $\phi g = g\phi$ for all $g$.
Then, $\phi T = T'\phi$, because for $g_ix\in g_i C$,
\begin{align*}
\phi (T(g_ix)) &= \phi(g_i(Tx)) = g_i\phi_C(Tx))
= g_iT'(\phi_C x) \\
&= T' (g_i (\phi_C x)) 
= T' (\phi (g_i x))\ .
\end{align*}
\end{proof} 


For a reduction of the classification of irreducible $G$-SFTs 
to a structure on mixing SFTs, we continue the
Adler-Kitchens-Marcus analysis.  Suppose
$H$ is a finite group and $T:Y\to Y$ is an irreducible
$H$-SFT (we use $H$ to match letters in \cite{akmgroup}) with period $p>1$.
Let $C^0, \dots , C^{p-1}$ 
denote the cyclically moving subsets of $Y$: the $C^i$ are disjoint; 
$T$ maps $C^i $ onto 
$C^{i+1}$
(superscripts interpreted mod $p$); and for each $i$,  the 
restriction of $T^p$ to $C^i$ is a mixing SFT.  
For $0\leq i < p$, set 

\begin{align*}
H^i_C= H^i &
=\{ g\in G: gC^0 = C^i \}
= \{  g\in H: gC^0 \cap C^i \neq \emptyset \} \\
&=\{ g\in G: gC^j = C^i \}\ , \quad
\text{for any }j\in \{ 0, 1, \dots , p-1\} 
\ .
\end{align*} 
$H^0$ is a normal subgroup of $H$. 
The sets $H^i$ are disjoint, with union $H$; each 
$H^i$, if nonempty,  is a left coset of $H^0$.  
Identifying $H^0$ and $H^p$, define $\kappa_T
=\min \{ i\in \mathbb N : 0< i \leq p, H^i \neq \emptyset \}$.
Then $\kappa_T$ divides $p$, and for $0<i<p$, $H^i$ is
nonempty if and only if $\kappa_T$ divides $i$.

It can happen that $H^1$ is
empty\footnote{The statements \cite[Part (iii) of the
    p.22 Lemma and p.23 Corollary]{akmgroup} neglect the
  possibility $H^1 = \emptyset$,
but the p.23 proof addresses it.}  (i.e., $ \kappa_T >1)$.  
For example, if $H=\{e\}$ and $T$ is a single  orbit of
length $p$, then $\kappa_T = p$. 
If $H= \Z_2 =\{e,g \}$, $T$ is a single orbit of length $6$,
and $g$ acts by $T^3$, then
$H^0=\{e\},H^3=\{g\} $ and 
$\kappa_T=3$.  However, there is a straightforward interpretation
of $\kappa_T$, as follows. (Recall, for a property $P$, $T$ is
totally $P$ if $T^k$ is $P$ for all $k>0$.) 

\begin{proposition} \label{meaning of kappa}
  Suppose $(Y,T)$ is an irreducible $H$-SFT. Let $\overline T$
  be the irreducible SFT which is the factor of $T$  obtained
  by collapsing $H$-orbits. Then $\kappa_T$ is the period of $\overline T$.
  The following are equivalent.
  \begin{enumerate}
  \item
    $\kappa_T = 1$.
  \item
    $T$ is totally $H$-transitive.
  \item
    $\overline T$ is mixing.
    \end{enumerate} 
\end{proposition}

\begin{proof}
  For $(1) \iff (2)$, note $H$-transitivity of all powers of $T$
  is equivalent to transitivity of the $H$-action on cyclically
  moving subsets, which is equivalent to $\kappa_T=1$. 
  For $(2) \iff (3)$, note
  every power of $T$ is $H$-transitive if and only if ever power
  of $\overline T$ is irreducible. This happens if and only if
  the irreducible SFT $\overline T$ is mixing.

  Now let $k$  be the period of $\overline T$.
  $\kappa_T$ is the smallest positive integer $j$ such that
  $T^j$ is the disjoint union of $j$ irreducible SFTs,
  each of which has $\kappa =1$. 
  For $j<k$, irreducible components of $\overline T^j$ are not mixing,
  so irreducible components of $T^j$ cannot have $\kappa =1$.
On the other hand, 
 $(\overline T)^k$
  is the disjoint union of $k$ mixing SFTs, so $T^k$ is the
  disjoint union of $k$ $H$-SFTs, each of which has $\kappa =1$.
Thus $\kappa_T=k$. 
  \end{proof}

Reduction to the case $\kappa_T=1$  clarifies issues and
simplifies notation. 
Note, for $k=\kappa_T$,
the restriction of $T^k$ to any of its irreducible components 
is $H$-invariant and therefore an $H$-SFT. 

\begin{proposition} \label{kappa=1reduction} Suppose
  $(Y,T)$ and $(Y',T')$ are irreducible $H$-SFTs
  with period $p>1$, with $\kappa_T = \kappa_{T'}:=k >1$.
  Then the restriction of $T^k$ or $(T')^k$ to any of its
  irreducible components has $\kappa =1$. 
If $C$ is an irreducible component for $T^k$, then
a conjugacy of $H$-SFTs $T,T'$ restricts to a conjugacy of
$H$-SFTs $T^k|C$, $(T')^k|C'$, for some irreducible component
$C'$ of $T^k$. Conversely, any conjugacy of $H$-SFTs
$T^k|C$, $(T')^k|C'$ extends uniquely to a conjugacy of
$H$-SFTS $T,T'$.
\end{proposition}

\begin{proof} 
  We will prove the extension claim. 
  Suppose $\phi_0:C\to C'$ is a conjugacy of $H$-SFTs
  $T^k|C$, $(T')^k|C'$. The
   unique
  extension of $\phi_0$ to a topological conjugacy of $T,T'$ as
  SFTs is given by 
  $\phi: T^ix\mapsto (T')^i(\phi_0x)$,
  for $x\in C$ and $0\leq i < k $.
  For $h\in H$, $0\leq i <k$, $x\in X$ and $y=T^ix$,
we then have
$  \phi (hy) = \phi (hT^ix)=\phi(T^ihx)=
(T')^i\phi_0 (hx) = (T')^ih\phi_0 (x)
= h(T')^i\phi_0 (x)=
h\phi(T^ix)=h\phi(y)$.
\end{proof}

Now suppose $T$ is an irreducible $H$-SFT of period $p>1$ with
$H^1=H^1(T)$ nonempty.
Pick $c$ in  $H^1$. We have $H^0$  a normal subgroup of $H$, 
$H^i = c^iH^0$ and the disjoint sets $H^i$ are  the cosets of $H^0$.
We call the  subgroup $H^0$ of $H$ the 
{\it primitive stabilizer} of $C$. We call 
$H^1$ the {\it stabilizer coset} (\lq\lq primitive 
stabilizer coset\rq\rq\ would be more accurate, but 
lengthier). Adler-Kitchens-Marcus are not responsible for the terms  
\lq\lq primitive stabilizer\rq\rq\ and 
\lq\lq stabilizer coset\rq\rq.

Given irreducible $H$-SFTs $T,T'$  of period $p$,
let $T_0$, $T'_0$ be  mixing SFTs given
by restriction of $T^p, (T')^p$ to some cyclically
moving subset of $T,T'$. If $p>1$ and $\kappa_T=1$, then
$T_0$ and $T'_0$ are not $H$-SFTs, but they are $H_0$-SFTs. 
However, a natural candidate reduction fails badly, as
 in the following example.

  \begin{example}\label{badhexample}
    Let $H= \Z_2 =\{e,g\}$.     We define two irreducible,
    period 2 $G$-SFTs, $T$ and $T'$, with $H^0=\{e\}$ and
    $H^1 =\{ g\}$, such that the restrictions $T_0,T'_0$ of
    $T^2$ and $(T')^2$ to irreducible components are conjugate
    $H_0$-SFTs, but $T$ and $T'$ are not conjugate as $H$-SFTs.
    Let $T=T_A$ for $A = g\left(\begin{smallmatrix} 1&2\\2&1
    \end{smallmatrix}\right)$, and similarly define 
    $T'$ from
    $A' = g\left(\begin{smallmatrix} 2&1\\1&2
\end{smallmatrix}\right)$.
    The SFTs $g^{-1}T$ and $g^{-1}T'$ are not conjugate, having
    different numbers of fixed points. But, $T_0$ and $T'_0$
    are conjugate, because $A^2=(A')^2$.
    \end{example} 

To find some reduction of the classification of
irreducible $H$-SFTs to a classification defined on
mixing SFTs, 
one is forced to consider the $\alpha$-skew $H$-SFTs.  
For this, we continue below to follow \cite{akmgroup}.



Suppose $H$ is a finite group and $\alpha:H\to H$ is a
group automorphism.
In \cite{akmgroup}, Adler, Kitchens and Marcus
defined a $\Z \otimes_{\alpha} H$ action on an 
SFT $T$ (we will call this pair an $\alpha$-skew $H$-SFT) to be 
 an embedding of $H$ as a group
of homeomorphisms such that $Th = \alpha (h) T$
(i.e. $Th (x) = \alpha (h) T(x)$ for
all $x$ and $h$). (The $H$-SFT case is the case that
$\alpha$ is the identity.) They showed
(see \cite[Observation 1, p.4]{akmgroup})   
that 
an $\alpha$-skew $H$-SFT can  be presented very concretely, as a
one-step SFT with an embedding $h\mapsto \pi_h$
of $H$ into the group of permutations of the alphabet, such
that for $x=(x_n)$ and $h\in H$, $hx$ is defined by 
$ (hx)_n = \pi_{\alpha^n(h)} (x_n)$.
They also showed
(see \cite[Theorem 1]{akmgroup}) 
these skew SFTs  are abundant:
for every $H$ and $\alpha$, every positive entropy irreducible SFT admits
a (not necessarily free)  $\alpha$-skew $H$-SFT,
which (by  \cite[Theorem 3]{akmgroup}) 
is then a 1-1 a.e. factor of an 
$\alpha$-skew $H$-SFT which is free (the $H$-orbit of $x$
has cardinality $|H|$, for every $x$).

Let $T$ be an irreducible $H$-SFT of period $p>1$,
with $H^1(T)$ nonempty. Fix $c$ in $H^1(T)$, and define $S=c^{-1}T$
(i.e., $S(x) = c^{-1}T(x)$). If
$c$ is in the center of $H$, then $S$ is an $H$-SFT;
otherwise, it is not. Two $\alpha$-skew $H$-SFTs are
by definition topologically
conjugate if they are conjugate as SFTs by a conjugacy intertwining the
$H$ actions. 

\begin{theorem}\label{skewsftreduction} 
  Suppose $T$ is
  an irreducible $H$-SFT of period $p>1$,
  with   nonempty coset $H^1(T)$. 
 Fix $c$ in $H^1(T)$, and define $S=c^{-1}T$.
  Let the cyclically moving subsets of $T$ be
  $C^i$, $0\leq i < p$. Let $S_i$ be the restriction of $S$ to $C^i$.
 Define $\alpha : h \to c^{-1}hc$ (with domain given by context). 
  Then the following hold. 
  \begin{enumerate}
  \item
    With the given $H$ action, 
    $S$ is a free $\alpha$-skew $H$-SFT
    (and by restriction of the action, $S$ is an
    $\alpha$-skew $H_0$-SFT).
  \item
    Each $S_i: C^i\to C^i$ is a free mixing ${\alpha}$-skew $H_0$-SFT.
\item Suppose $T'$ is another     
  irreducible $H$-SFT. 
  Assume the period of $T'$ is also
  $p$, and $H^1(T')=H^1(T)=H^1\neq \emptyset $.
  (These are necessary conditions for conjugacy of $T,T'$ as
  $H$-SFTs.) Let $C'^i$, $0\leq i < p$, be the cyclically moving
  subsets of $T'$. 
  Define $S'= c^{-1}T'$ and $S'_i=S'|{C'^i}$. 
  Then the following hold. 
  \begin{enumerate}
    \item  A conjugacy  $\phi$ of $H$-SFTs $T,T'$ restricts to a
      conjugacy of mixing $\alpha$-skew $H_0$ SFTs $S_0,S'_i$,
      for some $i$. Then $\phi \circ T^{-i}$ is a conjugacy of
the      mixing $\alpha$-skew $H_0$-SFTs $S_0,S'_0$. 
    \item
      Given a  conjugacy $\phi_0 : C^0 \to C'^0 $ 
      of the $\alpha$-skew $H_0$-SFTs $S_0,S'_0$,
      there is a unique conjugacy 
      $\phi$ of $H$-SFTs $T,T'$ such that $\phi = \phi_0$
      on $C^0$.
  \end{enumerate}
  \end{enumerate}       
\end{theorem}

\begin{proof}
  (1) Suppose $c^r=e$. Because $cS=Sc$, it follows that
  $S^r=(c^{-1}T)^r=T^r$, an SFT. As a root of an SFT, $S$
  must be an SFT. Next, given $h\in H$, we have 
    $S(hx) = c^{-1}T(hx)=c^{-1}hT(x)=c^{-1}hc(c^{-1}T)(x)=
  \alpha (h) (Sx)$. 
  $S$ is free because an $H$-SFT is by definition (in this paper) free.
  
  (2)    Clearly  $C^i$ is mapped to $C^i$ 
  by $S_i$ and $H_0$. Also, 
  $(S_i)^{pr} = (c^{-1}T)^{pr}|_{C^i}   =(T^p)^r|_{C^i}$.
  Roots and powers of  mixing SFTs are mixing SFTs,
  so $S_i$
    is a mixing $\alpha$-skew $H_0$-SFT.

    (3) The claim (a) follows from parts (1) and (2).
    Now suppose $\phi_0$ is given as
    in (b). 
    For $x\in C^0$ and     $0\leq i < p$, define
    $\phi (c^i x) = c^i \phi_0 (x)$. This $\phi$ is the 
 only possible extension of $\phi_0$ to
 a conjugacy of the $H$-SFTs $T,T'$.

 We claim  $S'\phi  = \phi S $.  
    For this, suppose $y=c^ix$, with $x\in C^0$ and     $0\leq i < p$. 
    Note $cS=Sc$  and $cS' = S'c$. So, 
    we have
    $\phi (Sy) = \phi (S(c^i x))
    =\phi (c^iS(x)) 
    =c^i\phi_0 (S x)
    =c^iS'_0\phi_0 ( x) = S'(c^i\phi_0 ( x) ) = S' \phi (y) $ . 

     Next, we claim $c\phi = \phi c$. For
    $x\in C^i $ with $0\leq i <p-1$, this is clear.
    For $y=c^{p-1} x \in C^{p-1}$, we have 
    $\phi (cy) = \phi_0 (c^px) = c^p \phi_0 (x)
     =c(c^{p-1}\phi_0(x)) = c\phi (c^{p-1} x)
     =c\phi(y) $.

We now have $\phi T =T' \phi$, because 
          $\phi T = \phi cS = c \phi S = c S' \phi  =T' \phi$. 
For $h$ in $H_0$ and $y\in Y$, it remains to  check
$\phi (hy)= h\phi y$.
Let $y=c^ix \in C^i$, $0\leq i <p$.
     Then
     $\phi (hy) = \phi (hc^ix) = \phi (c^i(c^{-i}hc^ix))
     =   c^i \phi_0 (c^{-i}hc^ix)
          =   c^i (c^{-i}hc^i) \phi_0 x 
    = hc^i \phi_0 x = h \phi (y) $ . 
    \end{proof}

  Theorem \ref{skewsftreduction}  reduces the problem of classifying irreducible
  $H$-SFTs, for all $H$, to the problem of classifying mixing 
   $\alpha$-skew $H$-SFTs (with $H$ acting freely),
  for all $\alpha $ and $H$.  For mixing $H$-SFTs, there is a
  reasonably satisfactory framework  of invariants arising from
  a theory of strong shift equivalence of matrices over $\Z H$
  (or  elementary equivalence of matrices over $\Z [t]$.
  One naturally has an analogous (somewhat ill defined) problem:

  \begin{problem} Find a satisfactory classification scheme 
    for mixing $\alpha$-skew $H$-SFTs. 
    \end{problem}

Finally, we  turn to relating matrix properties to the Adler-Kitchens-Marcus 
setting. 
Let $A$ be a square matrix over $\Z_+G$ with augmentation 
$\overline A$ over $\Z_+$ as in Section \ref{backgroundsec}.  
Let $a_{ij} = A(i,j)  $ and 
define nonnegative integers $a_{ijkg}$ by $A^k(i,j) = \sum_{g\in G}
a_{ijkg} g$; let  $a_{ijg}$ denote $a_{ij1g}$.

We recall some terminology.  $A$ is irreducible/primitive if $\overline A$ is
irreducible/primitive.  $A$ is nondegenerate if it has no zero 
row and no zero column. The nondegenerate core of $A$ is 
its maximum nondegenerate principal submatrix. For a property 
P, $A$ is essentially P if its nondegenerate core is P.  
$A$ is $G$-primitive if there exists $k>0$ such that 
$a_{ijkg}>0 $ for all $i,j,g$.

\begin{definition} \label{groupsdefinition}
Let $A$ be a square matrix over $\Z_+G$. 
For an index $i$ of $A$, the {\it weights group} (at $i$) is 
\[
W_i(A)= \{ g\in G: \exists k>0 \text{ such that }a_{iikg}>0\} \  
\] 
and the {\it ratio group} (at $i$) is 
\begin{align*}
\Delta_i (A)
&= \{ gh^{-1} : \exists k \text{ such that } a_{ijkg}>0 
\text{ and } a_{ijkh}>0 \} \\ 
&= \{ gh^{-1} : \exists k \text{ such that } a_{iikg}>0 
\text{ and } a_{iikh}>0 \}.
\end{align*}
\end{definition} 
$W_i(A)$ is clearly a finite semigroup, and hence a group. 
$\Delta_i(A)$ is a group, because given $(g_2)^j = e$,
we have
\[
g_1h_1^{-1}g_2h_2^{-1}
=
(g_1h_1^{j-1}g_2)(h_2h_1^j)^{-1} \ .
\]
$\Delta_i(A)$
 is named after a \lq\lq ratio group\rq\rq\ which 
plays an analogous role in the theory of Markov shifts 
\cite{MarcusTuncelwps,ParrySchmidt}.

Our next task is to compute the stabilizer data for $T_A$ 
from the matrix $A$. 
First  we recall a standard reduction; see the citation for a proof. 

\begin{proposition} \label{Hreduction}
\cite[Proposition 4.4]{BSullivan} 
Let $A$ be an irreducible matrix over $\Z_+G$. 
Let $i$ be an index of $A$ and let $H=W_i(A)$. Then there is a
diagonal 
matrix $D$ over $\Z_+G$ with each diagonal entry in $G$ 
(i.e., $\overline D = I$)  such that every entry of $DAD^{-1} $ lies 
in $\Z_+H$.  (It suffices for each $j$ to set $D(j,j)=g$ for
some $g$ such that for some $k$, $a_{ijkg}>0$.)  
\end{proposition}

An important technical point for proofs below is the following 
observation.  
\begin{remark} \label{vertexremark} 
Suppose  $A$ is an essentially irreducible square matrix over $\Z_+G$, 
$(w,g), (x,h)\in X_{\overline{A}}\times G$, $g=h$, and the 
initial vertices of $x_0$ and $w_0$ are equal. Then 
$(w,g)$ and $(x,h)$ are in the same cyclically moving subset 
of the same irreducible component of $X_{\overline{A}}\times G$. 
\end{remark} 

\begin{proposition} \label{weightsgroupprop}
Let $A$ be an irreducible matrix over $\Z_+G$. 
Let $x$ be a point in $X_{\overline A}$ 
with $x_0$ beginning at index $i$. 
Let $C$ be an irreducible component of $T_A$ 
containing  $(x,e)$ and choose $C^0, \dots , C^{p-1}$ 
such that $(x,e)\in C^0$. Then 
the following hold. 
\begin{enumerate} 
\item 
The stabilizer 
$H_C$ is the weights group $W_i(A)$. 
\item  
The matrix $B=DAD^{-1}$ over $\Z_+H$ from  
Theorem \ref{Hreduction}
defines an irreducible $H$-SFT $T_B$ which is $G$-cohomologous 
to the $H$-SFT $T|_C$. 
\item  The primitive 
stabilizer $ H_C^0 $ is the ratio group 
$\Delta_i (A)$. 
\end{enumerate} 
\end{proposition} 
\begin{proof}  (1): 
Suppose $g\in H_C$; then $(x,g)\in C$. By irreducibility of $C$, 
there must then be s ome path $z_0\dots z_{k-1}$ 
in 
$X_{\overline A}$
from 
$i$ to $i$ with weight $g$. Therefore $g\in W_i(A)$.

Conversely, suppose $g\in W_i(A)$. Then there is 
$k>0$ and a periodic point $w$  in $X_{\overline A}$ 
with $\ell (w_0) \ell (w_1)\cdots \ell (w_{k-1})=g$ (here $\ell (w_n)$ denotes the label of the edge $w_n$ in $\mathcal{G}_A$, see Section \ref{backgroundsec}) 
such that $i$ equals  the initial vertex of $w_0$ and the terminal vertex 
of $x_{k-1}$. The point $(w,e)$ must be in $C$. Therefore 
$(w,g)\in C$. Thus $gC \cap C \neq \emptyset$, so $gC=C$ 
and $g\in H_C$. 

(2): The diagonal matrix $D$ gives the $G$-cohomology. $T_B$ 
is irreducible because $C$ is irreducible.

(3): Given  $g,h,k$ as in the definition of $\Delta_i(A)$, 
using Remark \ref{vertexremark} one can see $g$ and $h$ are in 
$H^k_C$, and therefore $gh^{-1} \in H^0_C$. Conversely, 
suppose $g\in H^0_C$. Then $g$ there  is $k>0$ and a path 
$w_{0}\cdots w_{kp-1}$ from $i$ to $i$ with weight $g$. 
Then for arbitrarily large $j>0$ there is 
a path (by concatenations) of length $jp$ from $i$ to $i$ 
with weight $g$. For sufficiently large $j$, there is also 
a path of length $jp$ from $i$ to $i$ with weight $e$,
because the $H^0_C$-SFT $(T_A)^p|_{C^0}$ is mixing.
\end{proof}

\begin{proposition} \label{GreducibilitySummary}
  Let $G$ be a finite group and $A$ a square matrix
  over $\Z_+G$, defining a $G$-SFT $T_A$ as in Section 
\ref{backgroundsec}.   The following hold. 
\begin{enumerate} 
\item 
$A$ is essentially $G$-primitive $\iff$ $T_A$ is  mixing.  
\item
$A$ is essentially irreducible with a 
  weights group $W_i(A)=G$ $\iff$  the $G$-SFT $T_A$ is irreducible.  
\item
$A$ is essentially irreducible $\iff$ $T_A$ is  $G$-transitive 
and nonwandering. 

\item
$A$ is essentially irreducible, with a 
    weights group $W_i(A)=G$ and with $\overline A$
    essentially primitive $\iff$  $T_A$ is irreducible
    with $\kappa = 1$. 

\end{enumerate} 
\end{proposition} 
\begin{proof}
  (1)  is \cite[Cor. B.7]{BoSc2}.
  (2) is clear.
  (3) follows from (2) and part (1) of Proposition \ref{weightsgroupprop}.
  Then (4) follows from Proposition \ref{meaning of kappa}. 
  \end{proof}

Finally, we give algorithmically a presentation
(essentially following \cite{ParrySchmidt}) 
for the
classifying $\alpha$-skew $H_0$ SFT, from a given  presenting matrix
$A$ over $\Z_+ H$. 
 For an example, let $\Z_2 = \{ 1,c\}$ and
$A = \left( \begin{smallmatrix} c&1\\1&c \end{smallmatrix}\right)$.
Then $W_1(A)=\Z_2$, $\Delta_1(A)= \{1\}=H^0  $ and $H^1=\{c \}$;
with $D = \left( \begin{smallmatrix} c&0\\0&1 \end{smallmatrix}\right)$
and $\mathcal A =\left( \begin{smallmatrix} 1&1\\1&1 \end{smallmatrix}\right)$, 
we have 
and
$DAD^{-1} = c \mathcal A$.  

\begin{proposition} \label{ratiomatrixpresentation}
  Suppose $A$ is an  irreducible matrix over $\Z_+H$,
with  $\overline A$   primitive,     
and $H$ equal to a weights group $W_i(A)$ such that
$W_i(A) \neq \Delta (A)$ (i.e., $H\neq H^0$, so $T_A$ is
not mixing). 
Let $\beta$ be the weight of
a point of period $n+1$ from $i$ to $i$; let $\gamma$ be the weight
of a point of period $n$ from $i$ to $i$; let $c=\gamma^{-1}\beta$.
Pick $N$ such that
for each $j$,  there is a path of length $N$
from $i$ to $j$; choose $d_j$  the weight of some such path.
Let $D$ be the diagonal matrix with $D(j,j)=d_j$.

Then $DAD^{-1}= c \mathcal A$, where $\mathcal A$ is a
$H_0$-primitive matrix.  

$c^{-1}T$ acts on $C^0$ in $X\times G$ by the rule
$(x,g) \mapsto (\sigma x, (c^{-1}gc)\tau_{\mathcal A}(x))$,
where $\tau_{\mathcal A}(x)$ is the $\mathcal A$ label of the
edge $x_0$.

  For $c$  in the center of $H$,  the $H_0$-SFT
  $c^{-1}T_A$ is conjugate to $T_{\mathcal A}$. If $H^1$ contains
  an element of the center of $H$, then this element can be
  chosen as $c$.
  
\end{proposition} 
\begin{proof}
  If $z$ is a point of period $j$ from $i$ to $i$ in
  $X_{\overline A}$, with weight $h$, then $h$ must be in
  $H^j$ ($j$ interpreted mod $p$). This is because
  $T^j (z,e) = (z,h) $, and $h$ takes $(z,e)$ to $(z,h)$.
  It follows that the chosen $c$ lies in $H^1$. Also,
  let $g$ be the weight of a path from $k$ to $i$; then 
  \[
  DAD^{-1}(j,k) = (d_ja_{jk})(d_k)^{-1}
  = (d_ja_{jk}g)(d_kg)^{-1} \ . 
\] 
Interpreting  $(d_ja_{jk}g)$ and $(d_kg)$ as weights
of paths from $i$ to $i$, we see every entry of
$DAD^{-1}$ lies in $H^1$. Likewise, 
  every entry of
  $c^{-1}DAD^{-1}$ lies in $H^0.$ Finally,
  given another element $c'$ of $H^1$, let $c' = ch$ with $h\in H^0$;
  then we may pass from $c\mathcal A $ to $(ch)(h^{-1}\mathcal A)$.
  Finally, $\mathcal A$ being $H_0$-primitive follows from
  $c^{-1}T$ being mixing, as the latter means that for a large $n$,
  $T^n$ maps some point $(x,e)$ with initial vertex $i$ to a
  point with $(z,h)$ where $z$ has initial vertex $j$.
  Here, $h = c^{-n}eg$ such that $a_{ijng}>0$. This forces 
  $A$ to be $H$-primitive. 

The image of  $(x,g)\in X \times G$ under $c^{-1}T_A$ computes to be 
  \[
  (\sigma x, c^{-1}g\tau_A(x)) =
  (\sigma x, (c^{-1}gc)c^{-1}\tau_A(x))
  =(\sigma x, (c^{-1}gc)\tau_{\mathcal A}(x)) \ .
  \]
The final claim follows from 
  \[
  c^{-1}T_A : (x,g) \mapsto (\sigma x, c^{-1}g\tau_A(x))
  = (\sigma x, gc^{-1}\tau_A(x))  = (\sigma x, g\tau_{\mathcal A} (x)) \ .
  \]
  \end{proof}

\section{A special case}\label{specialcaseappendix}

Recall,  $G$ denotes a finite group.
The general theorems of this paper reduce
the classification of $G$-SFTs to an algebraic problem which is
far beyond the scope of this paper. Still,
in this section we study the algebraic invariants of
a natural initial class, including a complete solution in  
a meaningful (though very special)  case. The argument points to 
 algebraic challenges for  the general case.

Throughout, $\mathcal P$ is the poset $\mathcal P_2$, and
$\mathcal H$ is the coset structure for which $H_{11}=H_{12}=H_{22}=G$. 
To facilitate a quicker overview, proofs of some results do not
immediately follow statements.

Given $p,q,x$ in $\Z G$, let $M(p,q,x)=
\left(\begin{smallmatrix} p& x\\ 0&q \end{smallmatrix}\right)$. 
Let $\mathcal M^{++}(p,q,x)$ be the set of matrices $A$ 
in $\mopnG$ with coset structure $\mathcal H$  such that $I-A$ 
and $M(p,q,x)$ are $\text{El}_{\mathcal P}(\Z G)$-equivalent (i.e., they 
 have 1-stabilizations which for some $\mathbf n$ 
 are $\text{El}_{\mathcal P}(\mathbf n, \Z G)$-equivalent). 
 When $G$ is abelian, this forces
$\det(I-A\{1,1\})=p$ and 
$\det(I-A\{2,2\})=q$.
Let 
$\mathcal M^{++}(p,q) = \cup_x M^{++}(p,q,x)$.
In this appendix, we
study algebraic invariants of $G$-flow equivalence for the $G$-SFTs
defined by $\mathcal M^{++}(p,q) = \cup_x M^{++}(p,q,x)$: for general $G$,
then abelian $G$ satisfying a $K$-theory constraint, and finally for
$G=\Z_2$, where we give  a solution which is 
  complete and algorithmically
  practical. For example,  when $G=\Z_2$ 
and 
$\mathcal M^{++}(p,q)$ consists of finitely many $G$-FE classes,
 we can count them 
(Theorem \ref{classifforcounts}).

Before turning to the algebra, we
note the following consequence of
Proposition \ref{gettingpositive} (or a simple exercise), which shows
the algebraic study corresponds to actual $G$-SFTs. 

\begin{proposition}
Suppose $G$ is a finite group. 
For all $p,q,x$ in $\Z G$, 
$\mathcal M^{++}(p,q,x)$ is nonempty. 
\end{proposition}

Now we consider an arbitrary finite $G$. 
Given $p$ in $\Z G$ we define
$\ut(p)$ to be the set of $y$ in $\ZG$ such that
$\mu_y: v\mapsto yv$ induces an automorphism of 
$\Z G / (p\ZG)$. The induced map is
a right $G$-module homomorphism; it is a
 right $G$-module automorphism
 if and only if it an abelian group automorphism. 
The map   $\mu_y$ induces an automorphism
 if there exist $v,w,x $ in $\ZG$ such that
$vy = 1 +pw$ and $yv= 1+ px$. When $G$ is abelian,
this means simply that $[y]$ is a unit in
the quotient ring $\ZG / p(\ZG)$. 
Similarly, given $q\in \ZG$ 
we define
$\widetilde{\mathcal W}(q)$ to be the set of $z$ in $\ZG$ such that
$v\mapsto vz$ defines a group automorphism 
(equivalently, a left $G$-module
automorphism) of 
$\ZG / (\ZG)q$. This means there exist
$v,w,x $ in $\ZG$ such that
$yv = 1 + wq$ and $vy= 1+ xq$.

We need a little more. Let $M_k(p,q,x)$ be the 1-stabilization
of $M(p,q,x)$ in $\mpnzg$ with $\mathbf n = (k,k)$.
So, for $M=M_k(p,q,x)$, $M=I$ except that $M_{1,1}=p$,
$M_{k+1,k+1}=q$ and $M_{1,k+1}=x$. A matrix $A$ has a 1-stabilization
$\elzg$-equivalent to a 1-stabilization of $M(p,q,x)$
if and only if for all/any sufficiently large $k$, $A$ has a 1-stabilization
$\elzg$-equivalent to  $M_k(p,q,x)$. 
Now we define
$\ut^{eq}(p)$ to be the set of 
$ y$ in $ \ut(p)$ such that
for some $k$, there is an $\text{El}(k+1,\Z G)$ 
 equivalence
$U(p\oplus I_k)V = (p\oplus I_k)$ such that $U(1,1) = y$.
It will be convenient to write this equivalence in the form 
$U(p\oplus I_k) = (p\oplus I_k)W$ ($W$ =$V^{-1}$). 
Similarly,
we define
$\widetilde{\mathcal W}^{eq}(q)$ to be the set of 
$ z$ in $ \widetilde{\mathcal W}(q)$ such that
for some $k$, there is an
$\text{El}(k+1,\Z G)$ 
 equivalence
$U(q\oplus I_k) = (q\oplus I_k)W$ such that $W(1,1) = z$.

Finally, given $p,q$ in $\ZG$ we define the abelian group 
$L(p,q) = \{ pc + dq \in \ZG: c\in \ZG , d\in \ZG\}$.
When $G$ is abelian, $L(p,q)$ is also an ideal in $\ZG$.
For $G$ not abelian, $L(p,q)$  
 need not be even a onesided ideal. 
\begin{theorem}\label{theoremtwobytwo}
  Suppose $G$ is a finite group, $p,q,x\in\ZG$, and $A,A'$ are matrices in $\mathcal M^{++}(p,q,x)$,
  $\mathcal M^{++}(p,q,x')$ respectively. Then the following are equivalent. 
  \begin{enumerate}
  \item
    The $G$-SFTs $T_A$, $T_{A'}$ are $G$-flow equivalent.
  \item
    $M(p,q,x)$ and $M(p,q,x')$ have 1-stabilizations which are
    $\elpzg$-equivalent.
\item     
  There exist $y\in \ut^{eq}(p)$ and
  $z\in \widetilde{\mathcal W}^{eq}(p)$ such that $yx-x'z\in L(p,q)$.
  \end{enumerate}
Suppose $G$ is abelian, and consider quotient groups as rings. 
  Then  $\ut (p)$ is the set of $x$ in $\ZG$ such
  that $[x]$ is a unit in  $\ZG/\gen{p}$ (and similarly
  for $\ut (q)$). 
  Condition (3) holds 
 if and only if
 there exist elements $y\in \ut^{eq}(p)$,
  $z\in \widetilde{\mathcal U}^{eq}(q)$ such that
  $[y][x]=[z][x']$ in 
 $\ZG/L(p,q)$. 
  \end{theorem}

  Given $p,q$, with $L=L(p,q)$ 
  Theorem \ref{theoremtwobytwo} tells us\begin{footnote}{The relation (3)
    in the statement of Theorem \ref{theoremtwobytwo}
        is a special case version of an adaptation to $\ZG$
of  the invariant introduced by  Huang  in
\cite{dh:fersft} 
    for flow equivalence of reducible
    SFTs with two irreducible components.}\end{footnote}
    that the $G$-flow equivalence
  classes of $G$-SFTs defined from matrices in $\mathcal M^{++}(p,q)$
  are in bijective correspondence with
    the set
  of full orbits of
  points in the abelian group $\ZG /L$ under the action of
  the semigroup of homomorphisms $v+L\mapsto yv +L$
  and $v+L\mapsto vz +L$
  coming from $y\in \ut^{eq}(p), z\in \mathcal W^{eq}(q)$. 
  Often the group $\ZG / L(p,q)$ is finite, and in this case,
  with $G$ abelian, that orbit relation on $\ZG /L(p,q)$ can
  be computed mechanically.

  We now turn to some proofs. 

  \begin{proof}[Proof of Theorem \ref{theoremtwobytwo}]
    $(1)\iff (2)$: This follows from the classifying Theorem
    \ref{classification2}, by our choice of $\mathcal H$,
    because the only permutation of
    $\{1,2 \}$ respecting the $\mathcal P_2$ 
     relation $\preceq $ is the identity.

    $(2)\implies (3)$: For some $k$,
    we have an equivalence $UM_k(p,q,x) = M_k(p,q,x')W$
    of $2k\times 2k$ matrices, 
    which we can write in block form as 
    \begin{align} \label{cokernelequivalence} 
&\begin{pmatrix} 
U\{1,1\} &       U\{1,2\} \\ 
 0 &       U\{2,2\}
\end{pmatrix} 
\begin{pmatrix} 
    \left(\begin{smallmatrix}
      p&0\\0&I_{k-1}
    \end{smallmatrix} \right) &
    \left(\begin{smallmatrix}
      x&0\\0&0
    \end{smallmatrix} \right) \\ 
        \left(\begin{smallmatrix}
      0&0\\0&0
        \end{smallmatrix} \right) &
    \left(\begin{smallmatrix}
      q&0\\0&I_{k-1}
    \end{smallmatrix} \right)
\end{pmatrix} \\ \notag 
=
&\begin{pmatrix} 
    \left(\begin{smallmatrix}
      p&0\\0&I_{k-1}
    \end{smallmatrix} \right) &
    \left(\begin{smallmatrix}
      x'&0\\0&0
    \end{smallmatrix} \right) \\ 
        \left(\begin{smallmatrix}
      0&0\\0&0
        \end{smallmatrix} \right) &
            \left(\begin{smallmatrix}
      q&0\\0&I_{k-1}
    \end{smallmatrix} \right)
\end{pmatrix}
\begin{pmatrix} 
W\{1,1\} &    W\{1,2\} \\
 0 &       W\{2,2\}
\end{pmatrix}
\ .
\end{align}
    Computation of the $1,k+1$ entry of \eqref{cokernelequivalence}
    produces
    \[ U_{1,1}x + U_{1,k+1}q = pW_{1,k+1} + x' W_{k+1,k+1} \ .
    \]
    Set     $y=U_{1,1}$ and $ z= W_{k+1,k+1}$. If follows that
    $yx - x'z \in L(p,q)$. 
    
    Next, view    $(\ZG)^{k}$ as a set of column vectors, and
    let $v = (v_1, \dots , v_{k})^T$ denote an
    element of $(\ZG)^{k}$. 
    Because 
$U\{1,1\}    \left(\begin{smallmatrix}
      p&0\\0&I_k
    \end{smallmatrix} \right) =
        \left(\begin{smallmatrix}
      p&0\\0&I_k
        \end{smallmatrix} \right)W\{1,1\}$,
        the map $v\mapsto U\{1,1\}v$ induces a
        group automorphism (also a right $\ZG$-module
        automorphism)
        of 
        $\cok  \left(\begin{smallmatrix}
      p&0\\0&I_k
        \end{smallmatrix} \right)
= (\ZG)^{k+1}/ \left(\begin{smallmatrix}
      p&0\\0&I_k
    \end{smallmatrix} \right)(\ZG)^{k+1}$. 
The map $\pi : v \mapsto v_1$ induces an isomorphism         
        $\cok  \left(\begin{smallmatrix}
      p&0\\0&I_k
        \end{smallmatrix} \right)
\to \cok (p) =\ZG /p (ZG) $. Here $\pi$ pushes
the automorphism of
       $\cok  \left(\begin{smallmatrix}
      p&0\\0&I_k
        \end{smallmatrix} \right)$ 
down to the
automorphism of $\ZG /p (ZG) $ induced by $v_1 \mapsto yv_1$.
A similar argument, using the equivalence
$U\{2,2\}    \left(\begin{smallmatrix}
      q&0\\0&I_k
    \end{smallmatrix} \right) =
        \left(\begin{smallmatrix}
      q&0\\0&I_k
        \end{smallmatrix} \right)W\{2,2\}$
        and considering 
        the cokernel of
        $\left(\begin{smallmatrix}
      q&0\\0&I_k
    \end{smallmatrix} \right) =
        \left(\begin{smallmatrix}
      q&0\\0&I_k
        \end{smallmatrix} \right)$ with respect to
        the action on row vectors, shows that
        the rule $v_{k+1 }\mapsto v_{k+1} z$ induces an
        automorphism of 
        $\ZG / (\ZG)q$. This finishes the proof that
        $(2)\implies (3)$.

        $(3)\implies (2)$:
        By assumption, for large enough $k$ we have an elementary equivalence
        $G
        \left(\begin{smallmatrix}
      p&0\\0&I_{k-1}
        \end{smallmatrix} \right)
        H =
                \left(\begin{smallmatrix}
      p&0\\0&I_{k-1}
                \end{smallmatrix} \right)$
                such that $G_{1,1}= y$.
                Then
                \[
                \begin{pmatrix} G&0\\0&I_{k} \end{pmatrix}
\begin{pmatrix} 
  \left(\begin{smallmatrix}
      p&0\\0&I_{k-1}
    \end{smallmatrix} \right) &
    \left(\begin{smallmatrix}
      x&0\\0&0
    \end{smallmatrix} \right) \\ 
        \left(\begin{smallmatrix}
      0&0\\0&0
        \end{smallmatrix} \right) &
    \left(\begin{smallmatrix}
      q&0\\0&I_{k-1}
    \end{smallmatrix} \right)
\end{pmatrix} 
\begin{pmatrix} H&0\\0&I_k \end{pmatrix}
=
\begin{pmatrix} 
  \left(\begin{smallmatrix}
      p&0\\0&I_{k-1}
    \end{smallmatrix} \right) &
    \left(\begin{smallmatrix}
     y x&0\\w&0
    \end{smallmatrix} \right) \\ 
        \left(\begin{smallmatrix}
      0&0\\0&0
        \end{smallmatrix} \right) &
    \left(\begin{smallmatrix}
      q&0\\0&I_{k-1}
    \end{smallmatrix} \right)
\end{pmatrix} 
  \]
  where $w$ is a size $k-1$ column vector. Adding
  suitable multiples of columns $2, \dots , k$ to
  column $k+1$ to zero out $w$ implements a
  unipotent (hence elementary) equivalence to 
  $M(p,q,yx)$, which therefore is elementary
  equivalent to   $M(p,q,x)$. Similarly
  $M(p,q,x')$ is elementary
  equivalent to   $M(p,q,x'z)$. 
  Because $yx-x'z \in L(p,q)$, say
  $yx + pr = x'z + sq$, another application of unipotent equivalence
shows 
$M(p,q,yx)$ and  $M(p,q,x'z)$ are $\elpzg$-equivalent.

The final claims, for abelian $G$, are easily checked. 
\end{proof}

  We recall some facts from the book \cite[pages 1-4]{Oliver} of Oliver. 
  Suppose the finite group $G$ is  abelian. 
  Then  $SK_1( \ZG )$ is  the subgroup of $K_1 (\ZG )$ 
  represented by matrices with determinant 1; it is    
  the torsion subgroup of $K_1 (\ZG )$, and it is finite.
  The meaning of $SK_1( \ZG )$ being
  trivial is precisely that  every matrix over $\ZG$
  with determinant 1 has a 1-stabilization which is an elementary
  matrix over  $\ZG$. 
Being abelian, $G$  
 is the direct sum of its Sylow $p$-subgroups,
  and $SK_1( \ZG )$ is trivial if and only if
  one of the following hold\begin{footnote}{See 
    \cite[pp. 1-19]{Oliver} for an overview of
    $K_1(\ZG)$ and its history.}\end{footnote}
    (in which $C_k$ denotes the cyclic
  group of order $k$).
  \begin{enumerate}
  \item $G$ is a direct sum of copies of $C_2$.
  \item
Each Sylow $p$-subgroup of $G$ has the form $C_{p^n}$
    or  $C_p\oplus C_{p^n}$. 
  \end{enumerate}

  \begin{proposition} \label{u=ueq}
    Suppose $G$ is  abelian and
    $SK_1(\ZG)$ is trivial. Then 
    $  \ut^{eq}(p)=\ut(p)$
    and     $  \widetilde{\mathcal W}^{eq}(q)=\widetilde{\mathcal W}(q)$.  
  \end{proposition}

  \begin{proof}
    Because $G$ is abelian, obviously it is enough to
    prove   $  \ut^{eq}(p)=\ut(p)$.  
So, suppose
 $y\in \ut(p)$; 
we will show $y\in \ut^{eq}(p)$. 
Because $x\mapsto yx$ induces an automorphism of $\ZG /p(\ZG)$ 
there exist $a$ in $\ZG$ (implementing the inverse automorphism) 
and $b$ in $\ZG$ such that $ay = 1+bp$. Thus there is an
$\text{SL}(2,\ZG)$-equivalence 
\begin{align*}
    \left(\begin{smallmatrix}  y&p \\ b & a
  \end{smallmatrix}\right)   
  \left(\begin{smallmatrix}  p&0 \\ 0 & 1
  \end{smallmatrix}\right)   
  \left(\begin{smallmatrix}  a& -1 \\ -bp & y 
  \end{smallmatrix}\right)   
     = 
  \left(\begin{smallmatrix}  yp&p \\ bp & a
  \end{smallmatrix}\right)   
  \left(\begin{smallmatrix}  a& -1 \\ -bp & y 
  \end{smallmatrix}\right)   
  = 
  \left(\begin{smallmatrix}  ypa -bp^2&0 \\ 0 & -bp+ ay 
  \end{smallmatrix}\right)   
= \left(\begin{smallmatrix}  p&0 \\ 0 & 1
  \end{smallmatrix}\right)   
\end{align*}
and so, for arbitrary $k>0$ there is also an
$\text{SL}(k+2, \ZG)$
 equivalence
\[
    \left(\begin{smallmatrix}  y&p&0 \\ b & a &0 \\ 0 &0& I_k 
  \end{smallmatrix}\right)   
  \left(\begin{smallmatrix}  p&0&0 \\ 0 & 1 & 0 \\ 0&0&I_k 
  \end{smallmatrix}\right)   
  \left(\begin{smallmatrix}  a& -1 &0\\ -bp & y &0 \\ 0&0&I_k
  \end{smallmatrix}\right)   
  =
    \left(\begin{smallmatrix}  p&0&0 \\ 0 & 1 & 0 \\ 0&0&I_k 
    \end{smallmatrix}\right)   \ .
\]    
Because $SK_1(\ZG)$ is trivial, for some $k$ this $\text{SL}(k+2, \ZG)$-equivalence is
an $\text{El}(k+2, \ZG)$-equivalence. This shows $y\in \ut^{eq}$. 
  \end{proof}

  To finish, we  work out 
  the classification in complete detail for the case $G=\Z_2$. From
  here, $G$ denotes $\Z_2$. We use  $\Z_n$ to denote
  $\Z/n\Z$ and $G =\{e,g\}$.  
  
  {\bf The ring $R \cong \ZG $ and its ideals.}
  Let $R$ denote the subring of $\Z^2$ consisting of all
  $(\alpha , \beta)$ with $\alpha \equiv \beta \mod 2$.
  The map $\ZG \to R$ given by $ae+bg\mapsto (a+b,a-b)$
  is a well known ring isomorphism (with inverse
  $(\alpha, \beta)\mapsto \frac{\alpha+\beta}2 e + \frac{\alpha-\beta}2g$).
  We will  work with $R$ rather than $\ZG$.
  We define $E$ to be the even elements of $R$, i.e. the $(\alpha, \beta)$
  with $\alpha $ and $\beta$ even integers. The set $R\setminus E$ is 
  the set of odd elements. 
  Let  $E_+=\{ (\alpha , 0) \in R\} = \gen{(2,0)}$.
    Let  $E_-=\{ ( 0, \beta) \in R\} = \gen{(0,2)}$.
    $J$ will always refer to an ideal in $R$. Given $J$,
    let $J_+,J_-$ be the ideals and $j_+,j_-$ the nonnegative integers
    such that $J_+=E_+ \cap J = j_+E_+ = \gen{(2j_+,0)}$ and
    $J_-=E_- \cap J = j_-E_- = \gen{(0,2j_-)}$.
        If $(\alpha,\beta) \in J$, then $(2\alpha,0)\in J_+$ and
    $(0,2j_-)\in J_-$ From this we deduce that
        either $J=J_+\oplus J_-$ or $|J/(J_+\oplus J_-)|=2$.
        $J$ has rank 2 (as an abelian group) if
        and only if $j_+ \neq 0 \neq j_-$.
        If $|J/(J_+\oplus J_-)|=2$, then  $J$ is a rank two principal ideal,
    $J=\gen{j}$ with $j=(j_+,j_-)$.   For a nonnegative vector $v=(a,b)$,
    $\gamma_v:R \to \Z_a \oplus \Z_b$ denotes 
     the obvious map  $(\alpha, \beta) \mapsto ([\alpha],[\beta])$.

    \begin{proposition}\label{quotienttypes}
      For $J$ an ideal of $R$,  by cases 
      the following maps $\rho$ give a presentation of the abelian group
      epimorphism $R\to R/J$ (i.e., they define group epimorphisms
      with kernel $J$). 
      \begin{enumerate}
      \item $J=\gen{j}$, $j$ odd, $j=(j_+,j_-)$. Then
                $\rho : R\to \Z_{j_+} \oplus \Z_{j_-}$. 
        \begin{enumerate}
          \item 
        $\rho : (\alpha ,\beta) \mapsto \gamma_j(\alpha /2, \beta/2)$ 
        if $(\alpha ,\beta)\in E$, 
        \item
          $\rho : x \mapsto
          \gamma_j(x-j)$ 
          if $(\alpha ,\beta)\in R\setminus E$.
        \end{enumerate}
      \item $J= j_+E_+ \oplus j_-E_-$. Then
        $\rho =\gamma_{2j}:  R\to \gamma_{2j}R$, where
        $\gamma_{2j}R=\{ (a,b) \in  \Z_{2j_+} \oplus \Z_{2j_-}:
        a\equiv b \text{ mod } 2 \}$.        
      \item
        $J = \gen{j}$ rank 2 (i.e. $j_+\neq 0 \neq j_-$), $j=(j_+,j_-)$ even.
        Let $\overline j = \gamma_{2j}(j)$. Then 
        $\rho :  R\to \gamma_{2j}R/\{ 0, \overline j\}$.
          Here $\rho$ is $\gamma_{2j}$ followed by the quotient
          map from $\gamma_{2j}R$ with kernel the
          two-point subgroup $\{ 0, \overline j\}$.
          Moreover, $\rho (x) = \rho (x')$ if and only if
            $\gamma_{2j} (x)=\gamma_{2j}(x')$ or
            $\gamma_{2j} (x-j)=\gamma_{2j}(x')$.
      \end{enumerate}
    \end{proposition}
    \begin{proof} Let $J' =\gen{(2j_+,0),(0,2j_-)}$. 

      {\bf Case 1.}
      If $x$ is even, then $x\in J$ iff $x\in
J'$. If $x$ is odd, then
      $x\in J$ iff $ x-r\in 
      J'$.
      Let $\rho' (x) = \gamma_{2j}(x)$ if $x\in E$, and
      $\rho' (x) = \gamma_{2j}(x-j)$ if $x\in R\setminus E$.
      Then $\rho': R \to
      \Z_{2j_+} \oplus \Z_{2j_-}$ is a homomorphism and $\ker (\rho')=J$.
      The image of $\rho'$ is
      the subgroup $H$ of $([\alpha],[\beta])$ with $\alpha$ and $\beta $
      even. The map 
      $([\alpha] ,[\beta] )\mapsto ([\alpha /2], [\beta/2])$
      is a group isomorphism  $H\to \Z_{j_+} \oplus \Z_{j_-}$.

      {\bf Case 2.} This is clear, because
      $J = J'$.

      {\bf Case 3.} 
       $\{ 0, \overline j\}$ is a group, because $2\overline j =0$ in $\Z_{2j_+} \oplus \Z_{2j_-}$. Then  
\[
\ker (\rho ) = 
\gamma_{2j}^{-1} \{0, \overline j\} = 
J'  \cup  ( j+J' )
 = J \ .
 \]
 The final ``Moreover'' statement follows. 
    \end{proof}

    In Case 1 above,
    on account of the final map
    $H\to \Z_{j_+} \oplus \Z_{j_-}$, 
    the group epimorphism  $\rho$  is  not
    a ring homomorphism
    (unless $J=R$)
    .

    Given  $v\in R$, we let $\mu_v$ denote the map $R\to R$ given
by $w\mapsto vw$, and also (for a lighter notation) the
 homomorphism 
induced on a quotient group  of $R$. 
For an ideal $J$, let $\ut(J)$ be the set of $v$
in $R$ such that $\mu_v$ is  an
automorphism of $R/J$.
For elements $v_1, \dots , v_k$
of $R$, $\ut ( v_1, \dots , v_k)$ denotes $\ut (J)$ for $J
=\gen{v_1, \dots , v_k}$.

\begin{theorem} \label{Junitstheorem}
   $\ut(J)$ is characterized 
  by cases.
  \begin{enumerate}
  \item $J=\gen{j}$, $j=(j_+,j_-)$ odd.
Then    $(\alpha,\beta)\in\ut(J)$ if and only if
    $    \gcd(\alpha,j_+)=1=\gcd(\beta,j_-)$ .
\item $J \subset E$. 
  Then 
$(\alpha,\beta)\in\ut(J)$ if and only if
$(\alpha,\beta)$ is odd with  $    \gcd(\alpha,j_+)=1=\gcd(\beta,j_-)$ .
  \end{enumerate}
\end{theorem}
In Case 1, $(2^k, 2^{\ell})\in \ut(J)$ if $k$ and $\ell$ are
positive. In Case 2, if $J=\gen{j}$ with $j$  even, then 
the requirement that $(\alpha , \beta)$ must be odd is redundant;
if $J=  j_+E_+ \oplus j_-E_-$, then
$J\subset E$, and  when $(j_+,j_-)$ is odd the
requirement is not redundant. For $k$ nonzero in $\Z$,
$\gcd (k,0) = |k|$. So, if  $j_+ \neq 0 =j_-$, Theorem \ref{Junitstheorem}
gives 
$\ut(J) = \{ (\alpha , \pm 1)\in R\colon \gcd(\alpha ,j_+)=1\}$.
Likewise, if  $j_+ = 0 \neq j_-$, then 
$\ut(J) = \{ (\pm 1, \beta)\in R\colon \gcd(\alpha ,j_+)=1\}$. 
If $J=\gen{0}$, then $\ut(J) =\{\pm 1, \pm 1\}$, the units of $R$.
If $J=R$, then $\ut(J)=R$. 
\begin{proof}[Proof of Theorem \ref{Junitstheorem}]
  Let $v=(a,b)\in R$. 
We use the maps $\rho $ from 
Proposition \ref{quotienttypes} which present $R\to R/J$.

  {\bf Case 1.} Here $\rho: R\to \Z_{j_+} \oplus \Z_{j_-}$. 
 For $x=(\alpha , \beta)\in E$ 
  the induced  map $\rho (x)\mapsto \rho (vx)$ is 
  \[
    \big([\alpha/2] , [\beta/2]\big) \mapsto  
    \big([a\alpha/2] , [b\beta/2]\big) =
    \big(a[\alpha/2] , b[\beta/2]\big) \ ,
\]
    and for $x\in R\setminus E$ the map is 
    \begin{align*}
      \big([(\alpha-j_+)/2] , [(\beta-j_-)/2]\big) \mapsto &\, 
\big([(a\alpha-j_+)/2] , [(b\beta-j_-)/2]\big) \\
=&\, \big(a[(\alpha-j_+)/2] , b[(\beta-j_-)/2]\big)
\end{align*}
where the last equality holds because $(j_+,j_-)$ is odd. 
It follows that 
the induced map $\mu_v$ is an automorphism  
if and only the gcd conditions hold.

{\bf Case 2.} 
First suppose $J= j_+E_+ \oplus j_-E_-$.
Here $\rho= \gamma_{2j}: R \to \Z_{2j_+}\oplus \Z_{2j_-}$.
 The map $\rho(x)\mapsto \rho (vx)$ is
$([\alpha], [\beta]) \mapsto ([a\alpha ],[b\beta])$,
so it is an automorphism iff $\gcd (a,2j_+)=1=\gcd (b,2j_-)$.

Lastly, suppose      $J=\gen{j}$, with $j=(j_+,j_-)$ even and 
$j_+ \neq 0 \neq j_-$.
Here $\rho :  R\to (\Z_{2j_+} \oplus \Z_{2j_-})/\{ 0, \overline j\}$
presents $R\to R/J$.
If $v$ is even or the gcd condition fails,
then $\mu_v  $ as an endomorphism of
$(\Z_{2j_+} \oplus \Z_{2j_-})/\{0,\overline j\}$ has a nontrivial kernel.
If  $v$ is odd, then
$\mu_v  $ as an endomorphism of
$(\Z_{2j_+} \oplus \Z_{2j_-})$ fixes $\overline j$,
so it  defines an automorphism of
$(\Z_{2j_+} \oplus \Z_{2j_-})/\{0,\overline j\}$ if and only if
defines an automorphism of
$(\Z_{2j_+} \oplus \Z_{2j_-})$, which is equivalent to the gcd conditions. 
\end{proof}

\begin{corollary} \label{Junitscorollary} For elements $p,q$ of $R$, let
  $J= \gen{p,q}$. Then
  \[
\ut (p) \cup \ut (q)  \ \subset \ 
  \ut (p,q) = \{ vw :
v\in \ut (p), w \in \ut (q) \}\ .
  \]

  \end{corollary}

\begin{proof}
  For an ideal $J$ and element $x$ of $R$,
  $x\in \ut (J)$ iff its image in $R/J$, considered as a ring,
  is a unit. As  a unit in $R/\gen{p}$ pushes down to a unit in
  $R/\gen{p,q}$, we have $\ut (p)\subset \ut (p,q)$, and  likewise
  $\ut (q)\subset \ut (p,q)$.  
  
It remains to show  $\ut (p,q) \subset \{ vw :
v\in \ut (p), w \in \ut (q) \}$.
 We appeal to Theorem \ref{Junitstheorem}, by cases. 
  Define $d_+ = \gcd (\alpha_p, \alpha_q)$ if at least one of
  $\alpha_p, \alpha_q$ is nonzero, and $d_- = \gcd(\beta_p, \beta_q) $
  if at least one of $\beta_p, \beta_q$ is nonzero.

  {\bf Case I.} $J=\gen{r}$ with $r=(\alpha_r,\beta_r)$ odd.

  By Lemma \ref{doubled}, $r=(d_+, d_-)$. 
  If e.g. 
  $\gcd (\alpha,\alpha_r)=1$ and $\alpha $ is an odd prime,
  then $(\alpha,1)\in \ut (p)\cup \ut(q)$. We see that 
   an odd element $u$ of $\ut(p,q)$
 is the product 
  of elements of the form
  $(\alpha , 1)$,  $(1, \beta)$ from $\ut (p)\cup \ut(q)$.
  If $k$ and $\ell$ are positive, then
  $(2^k,2^{\ell}) \in \ut (p)\cup \ut(q)$,
  because at least one of $p$ and $q$ must be odd
  (because $r$ is odd). Thus $\ut (p)\cup \ut(q)$ generates
  $\ut (p,q)$. A product of several elements from $\ut (p)\cup \ut(q)$
  is a product of a single element of $\ut (p)$ and
  a single element of $ \ut(q)$.

{\bf Case II.} 
$J\subset E$.
Here $p$ and $q$ must be even, and 
there can be no odd element in any of $\ut(p),\ut(q),\ut(J)$.
The rest of the argument proceeds as in Case I by considering
gcd and generating units $(\alpha,1)$, $(1,\beta)$.
\end{proof}

\begin{definition} \label{lastequivdefn}
  Given elements $x,x'$ of a ring $S$, we say
  $x\sim x'$ in $S$ if there is a multiplicative
  unit $u$ in $S$ such that $ux=x'$.
\end{definition}

In the product ring $\Z_m \oplus \Z_n$, the units
are the pairs $([a],[b])$ such that
$\gcd (a,m)=1=\gcd (b,n)$, and these are the units
defining $\sim$ in  part (3) of the statement of
Theorem \ref{classifeasydecide} below.
(The map $\rho$ in part (3a) is not a ring homomorphism,
but the logic of the proof does not need it to be.) 
Also, 
$\mathcal M^{++}(p,q,x)$ was defined with $\{ p,q,x\} \subset \ZG$. 
In Theorem \ref{classifeasydecide}, 
we use the corresponding elements in $R$,
without introducing more notation. E.g., $x=ae+bg$ becomes
$x=(a+b, a-b)$.


\begin{theorem}\label{classifeasydecide} 
  Let $J$ be the ideal $ \gen{p,q}$ in $R$, with $J\cap E = j_+E_+ \oplus j_-E_-$.
Suppose $x,x'$ are in $R$. Then the 
following are equivalent.
\begin{enumerate}
    \item 
  Matrices in $\mathcal M^{++}(p,q,x)$ and
  $\mathcal M^{++}(p,q,x')$ define $G$-flow equivalent $G$-SFTs.
      \item 
        $[x]\sim [x']$ in $R/J$. 
      \item The  conditions below hold, according to the type of  $J$. 
    \begin{enumerate}
    \item  $J=\gen{j}$ with $j=(j_+,j_-)$ odd.  
      Then    $\rho (x) \sim \rho (x')$ in $\Z_{j_+}\oplus \Z_{j_-}$, 
                where  
          \begin{enumerate}
          \item $\rho (x) = \gamma_j(x/2)$ if $x\in E$, and
            \item 
              $\rho (x) = \gamma_j((x-j)/2)$ if $x\notin E$.
                      \end{enumerate} 
      \item
        $J= j_+ E_+ \oplus j_-E_- $ .  
        Then $\gamma_{2j} (x) \sim \gamma_{2j} (x')$ in  
        $\Z_{2j_+}\oplus \Z_{2j_-}$.  
      \item $J=\gen{j}$, $j=(j_+, j_-)\in E$, $j_+\neq 0 \neq j_-$.
        \\
        Then  $\gamma_{2j} (x) \sim \gamma_{2j}(x')$ or 
        $\gamma_{2j} (x-j ) \sim \gamma_{2j} x'$ in
        $\Z_{2j_+}\oplus \Z_{2j_-}$. 
                \end{enumerate}
  \end{enumerate}
  \end{theorem}

\begin{proof}
  Because $SK_1(\Z_2)$ is trivial, it follows from
  Proposition \ref{u=ueq} and Theorem \ref{theoremtwobytwo} that
  (1) holds if and only there are elements $y\in\ut (p)$ and
  $z\in\ut(q)$ such that $[y][x]=[z][x']$ in $R/J$. 
    By Corollary \ref{Junitscorollary}, these elements $y,z$
  exist if and only if
  $[x]\sim [x']$ in $R/J$. We have shown $(1)\iff (2)$.

  We next show $(1)\iff (3)$. 
  Let $x=(\alpha , \beta)$ and $x' = (\alpha', \beta')$ be elements of $R$. 
  We will use the map $\rho$ of Proposition \ref{quotienttypes} which  
  presents $R\to R/J$. Then
  \begin{align}
\notag    &[x]\sim [x'] \text{ in } R/J
    \iff \\ 
\notag     &\exists (a,b) \in \ut (J)\ ,  
         [(a\alpha,b\beta)] = [(\alpha',b\beta)] \text{ in } R/J
                  \iff \\ 
                  \label{quotidian}              &
                  \exists (a,b) \in \ut (J)\ ,\rho ((a\alpha, b\beta)) =
                  \rho ((\alpha', \beta')) \ . 
  \end{align}
  
  {\bf Case (a)}: $J=\gen{j}$ with $j=(j_+,j_-)$ odd.
  Here $\rho: R \to \Z_{j_+} \oplus \Z_{j_-}$. 
By Theorem \ref{Junitstheorem} and Corollary \ref{Junitscorollary},
$(a,b)\in \ut (J)$ if and
only if $\gcd (a, j_+) = 1 = \gcd (b,j_-)$.

Suppose $(a,b)\in \ut (J)$ and 
$(\alpha, \beta) \in R$. Then
\begin{align*} 
  \rho (\alpha, \beta) & = 
  \begin{cases}
  \gamma_j(\alpha /2, \beta/2)
  &\text{if } (\alpha,\beta) \in E \ , \\ 
  \gamma_j((\alpha -j_+)/2, (\beta -j_-)/2))
  & \text{if } (\alpha,\beta) \in R\setminus E \ ;
  \end{cases}\\ \\
  \rho (a\alpha, b\beta) & = 
  \begin{cases}
  \gamma_j(a \alpha /2, b \beta/2)
  & \text{if } (\alpha,\beta) \in E  \ , \\ \\
  \gamma_j(a \alpha /2, b \beta/2)
  & \text{if } (\alpha,\beta) \notin E \text{ and }\\
  &\qquad \qquad\quad (a,b) \in E \ , \\
  \gamma_j((a\alpha -j_+)/2, (b\beta -j_-)/2))
  &  \text{if } (\alpha,\beta) \notin E \text{ and }\\
  &\qquad \qquad\quad (a,b)\notin E .
  \end{cases} 
\end{align*}
In each case, we can check that $\rho (a\alpha, b\beta)$ equals the
image of $\rho(\alpha ,\beta)$ under the
automorphism of $\Z_{j_+} \oplus \Z_{j_-}$ defined by
$([m],[n])\mapsto ([am],[bn])$. Therefore \eqref{quotidian} is equivalent
to $\rho (x) \sim \rho (x')$ in $\Z_{j_+} \oplus \Z_{j_-}$.

{\bf Case (b)}: 
$J= j_+ E_+ \oplus j_-E_- = \gen{(2j_+,0), (0,2j_-)}$.
Here $\rho = \gamma_{2j}$.
By Theorem \ref{Junitstheorem} and Corollary \ref{Junitscorollary},
$(a,b)\in \ut (J)$ if and
only if $\gcd (a, 2j_+) = 1 = \gcd (b,2j_-)$.
Thus \eqref{quotidian} is equivalent to
$(\alpha,\beta) \sim (\alpha',\beta')$ in
$\Z_{2j_+} \oplus \Z_{2j_-}$. 

{\bf Case (c)}: $J=\gen{j}$, $j=(j_+,j_-)\in E$, $j_- \neq 0 \neq j_+$ . 
By Proposition \ref{quotienttypes}, here $\rho (x) = \rho (x')$
if and only if $\gamma_{2j}(x)= \gamma_{2j}(x')$ or
$ \gamma_{2j}(x+j) = \gamma_{2j}(x')$. So,
\eqref{quotidian} holds iff
there exists $(a,b)\in \ut (J)$ such that
\begin{equation} \label{quotidian2}
  \gamma_{2j}((a\alpha,b\beta)) = \gamma_{2j}(\alpha', \beta') \quad
  \text{or} \quad 
  \gamma_{2j}((a\alpha,b\beta)+j) = \gamma_{2j}(\alpha', \beta') \ .
  \end{equation} 
As in Case (b),
$(a,b)\in \ut (J)$ if and
only if $\gcd (a, 2j_+) = 1 = \gcd (b,2j_-)$.
For $(a,b)\in\ut (J)$, 
$   \gamma_{2j}((a\alpha,b\beta)+j)
= \gamma_{2j}((a(\alpha +j_+),b(\beta+j_-)) $
because $(a,b)$ is odd. 
Thus \eqref{quotidian2} is equivalent to
$\gamma_{2j}(x)\sim \gamma_{2j}(x')$ or
$\gamma_{2j}(x+j)\sim \gamma_{2j}(x')$ in
$\Z_{2j_+}\oplus \Z_{2j_-}$. 

\end{proof}

The next two propositions explain how to compute the ideal $J$
in the form of Theorem \ref{classifeasydecide} from its generating
polynomials $p,q$. The (trivial) Proposition \ref{soeasy} is
stated to give a complete list of possibilities. 

\begin{proposition}\label{soeasy}
Let $p=(\alpha_p, \beta_p)$ and
$q=(\alpha_q, \beta_q)$.
Suppose $J=\gen{p,q}$ with $\{p,q\} \subset E_+ \cup E_-\}$
(i.e., $\alpha_p\beta_p=0=\alpha_q\beta_q$). 
  \begin{enumerate}
    \item 
      $J=0$ if and only if $p=q=0$.
      \item 
        Suppose $\beta_p = 0=\beta_q$ and at least one of $\alpha_p,\alpha_q$
        is nonzero. Let $d_+= \gcd (\alpha_p,\alpha_q)$. 
        Then $J$ has rank 1, and
        $J= \gen{ (d_+,0)} =j_+E_+$ with  $j_+= d_+/2$.
              \item 
        Suppose $\alpha_p = 0=\alpha_q$ and at least one of $\beta_p,\beta_q$
        is nonzero. Let $d_-= \gcd (\beta_p,\beta_q)$.
        Then $J$ has rank 1, and
               $J= \gen{ (0,d_-)}= j_-E_-$ 
        with  $j_-= d_-/2$.

        \item Suppose  one of $p,q$ is
          $(2m,0)\neq 0$ and the other is $(0,2n)\neq 0$ with $m\neq 0\neq n$.
          Then $J$ has rank 2, and 
          $J=j_+E_+\oplus j_-E_-$ with $(j_+, j_-) = (m,n)$.
  \end{enumerate}
\end{proposition}

\begin{proposition}\label{doubled}
  Suppose  
  elements
  $p=(\alpha_p, \beta_p) $ and $q=(\alpha_q, \beta_q)  $
  generate a rank 2 ideal $J= \gen{p,q}$,
  and at least one of $p,q$ lies outside $E_+ \cup E_-$. 
    Set $d_+= \gcd (\alpha_p, \alpha_q)$ and $d_-=
    \gcd (\beta_p, \beta_q)$.
Then     one of the following holds.
  \begin{enumerate}
  \item
    $J=  \frac{d_+}{2}E_+ \oplus  \frac{d_-}{2} E_-
= \gen{(d_+,0), (0,d_-)}$ . 
  \item $J$ is the principal ideal 
    $\gen{(d_+ , d_-)} $.
  \end{enumerate}
Write $\alpha_p= r_p d_+$,
$\beta_p=  s_p d_-$, 
$\alpha_q= r_q  d_+$  and $\beta_q=  s_q  d_-$. 
Then  $J$ is principal if and only if
  \begin{equation} \label{congruence} 
    r_p\equiv s_p \mod 2 \quad \textnormal{and}
    \quad r_q\equiv s_q \mod 2\ .
  \end{equation}
    \end{proposition} 
\begin{proof}
The claim when  $\{p,q\} \subset E_+ \cup E_-$
is obvious. Now 
  suppose $p \notin (E_+ \cup E_-)$ (the argument when
  $q \notin (E_+ \cup E_-)$ is essentially the same). 
Then   $\alpha_p \neq 0 \neq \beta_p$.  
  Let $H$ be the subgroup of $R$  generated by the groups
$H_p = \alpha_p E_+ \oplus \beta_p E_- $ and 
$
H_q = \alpha_q E_+ \oplus \beta_q E_- $. 
Then 
  \begin{align*} H &=
    d_+E_+ \oplus d_- E_- \\
    &= \gen{(2d_+,0) , (0, 2d_-)} \ \subset J
    \subset \ \gen{(d_+,0) , (0, d_-)} \ . 
\end{align*}
    There are integers $\ell_1, \ell_2$ such that
  $d_+  = \ell_1 \alpha_p  +
    \ell_2 \alpha_q
    $. Let 
  $v=\ell_1 p + \ell_2 q $.
    Then
    \begin{align*}
v &= \ell_1(\alpha_p, \beta_p)
+ \ell_2 (\alpha_q, \beta_q) \\ 
&= \ell_1(r_pd_+, s_pd_-)
+ \ell_2 (r_qd_+,s_qd_-) \\ 
& = \big((\ell_1r_p+\ell_2r_q)d_+, (\ell_1s_p + \ell_2 s_q)d_-\big)
 = \big(d_+, c_1d_-\big) 
  \end{align*} 
 for some integer $c_1$. 
 Likewise for some $w$ in $J$ we have
 $w= (c_2 d_+ , d_-) \in J$.
 So, $(d_+ , d_-) \in J$, and $(d_+,0)\in J$
 iff $(0,d_-) \in J$. 

  We can now see that $J$ is principal if and only the following
  holds for all integers $\ell_1, \ell_2$:
  \begin{equation} \label{evenodds} 
    (\ell_1 r_p + \ell_2 r_q) \textnormal{ is odd }
    \iff 
    (\ell_1 s_p + \ell_2 s_q ) \textnormal{ is odd .}
    \end{equation} 
Clearly \eqref{congruence} implies 
\eqref{evenodds}. For the converse,
  assume \eqref{evenodds}.
If $r_p,r_q$ are both odd, then \eqref{congruence} holds,
because $s_p,s_q$ are not both even, as $\gcd (s_p, s_q)=1$. 
Suppose $r_p$ is even; then  $r_q$ is odd. 
Let  integers $\ell_1, \ell_2$ give  $\ell_1 r_p + \ell_2 r_q =1$.
Here, 
 $\ell_2$ must be odd, but we may choose $\ell_1$ even or odd, because 
$(\ell_1 + r_q)r_p + (\ell_2 - r_p)r_q = 1$.
Because 
$\ell_1 s_p + \ell_2 s_q $ must be odd in either case,  
it follows that 
 $s_p$ must be even, which forces $s_q$ to be odd,
and therefore \eqref{congruence} holds.
The argument for the case that $r_q$ is even is essentially the
same. 
\end{proof}

To use Theorem \ref{classifeasydecide},
we recall  some elementary number theory.
Suppose $n$ is a nonnegative integer.
For integers $\alpha$ and $ \alpha'$, 
$[\alpha]\sim [\alpha']$ in $\Z_n$ means there exists
an integer $a$ such that $\gcd (a,n)=1$
and $[a\alpha]= [\alpha']$. 
If $n=0$, then 
$[a]\sim [a']$ means 
$a= \pm a'$.
If  $n=1$ then  
$[a]\sim [a']$  holds for all $a,a'$. 
Suppose $n=p^m$ with $p$ prime and $m>0$.
Write $y$ in $\Z_{p^m}$  as  $y= \sum_{j=0}^{m-1}y_jp^j$
with each $y_j$ in $\{ 0, 1, \dots , p-1 \}$;
 $y$ is a unit if and only if $y_0 \neq 0$.  
Now,  $[y] \sim [y']$ in $\Z_{p^m}$ if and only if
$y=0=y'$ or 
$\min \{ j: y_j \neq  0\} =
\min \{ j: y'_j \neq 0\}$.
Equivalently,
$\max\{k: 1\leq k\leq m, p^k|y\} =
\max\{k: 1\leq k\leq m, p^k|y'\}$. 
For the equivalence relation $[a]\sim [a']$
 in
$\Z_{p^m}$, there are exactly $m+1$ equivalence classes. 

Finally, suppose $n>1$ 
with prime power factorization 
$n =\prod_{i=1}^k p_i^{m_i}$.
As a ring, $\Z_n$ is isomorphic to 
$\prod_{i=1}^k \Z_{p^{m_i}}$. 
In this presentation, let $x=(x_1, \dots , x_k)$
and $x'=(x'_1, \dots , x'_k)$. Then $x\sim x'$
in $\prod_{i=1}^k p_i^{m_i}$ if and only if 
$x_i\sim x'_i$
in $\Z_{ p_i^{m_i}}$, for $1\leq i\leq k$.
To express this another way,
for $a$ in $\Z$ and $p=p_i$, let $\delta_{p,n}(a)
= \max\{k: 1\leq k \leq m_i: p^k | a\}$.
Then  $[a]\sim [a']$ in $\Z_n$ if and only if
$\delta_{p,n}(a)=\delta_{p,n}(a')$, for each prime
$p$ dividing $n$. Given $a,a',n$ this is straightforward
to compute. 

For $n =\prod_{i=1}^k p_i^{m_i}$ as above, 
define $\kappa (n) = \prod_{i=1}^k (m_i + 1)$,
and also define  $\kappa (1)=1$.
Then for positive integers $m,n$
 it follows that
for the equivalence relation
$\sim $ in the product ring $\Z_m\oplus \Z_n$,
the number of equivalence
classes is $\kappa (m)\kappa (n)$.

We will need to count $\sim$ classes
in  $\Z_m\oplus \Z_n$ represented by elements
of  $\gamma_{2j}R=\{([a],[b]): a+b \equiv 0 \text{ mod } 2\}$.
Let $j_+ = 2^{k_+}g_+$, with $g_+$ an odd integer.  
The $\sim$  classes
in $\Z_{2j_+}$ arising from odd and even integers are disjoint.
The number of classes arising from odd integers is $\kappa (g)$; 
the number arising from even integers is $(k_++1)\kappa (g)$.
The analogous statements hold for $j_-=2^{k_-}g_-$, with $g_-$
 an odd integer. 
So, in $\gamma_{2j}R$, the number of $\sim$ classes arising from
odd elements of $R$ is $\kappa_{g_+}\kappa_{g_-}$ and the
number arising from even elements of $R$ is
$(k_++1)(k_-+1)\kappa_{g_+}\kappa_{g_-}$.

\begin{theorem} \label{classifforcounts}
  Let $J= \gen{p,q}$ be rank 2 in $R$, i.e. $J\cap E = j_+E_+ \oplus j_-E_-$
  with $j_+\neq 0 \neq j_-$.
  Write $j_+ = 2^{k_+}g_+$, $j_- = 2^{k_-}g_+$ with $g_+, g_-$ odd integers. 
  Let $f_{p,q} $ be the number of distinct $G$-flow equivalence
  classes defined by matrices in $\mathcal M(p,q)$. There are three
  cases: 
    \begin{enumerate}
    \item  $J=\gen{j}$, with $j$ odd. 
    \item $J = j_+E_+ \oplus j_-E_-= \gen{(2j_+,0),(0,2j_-)}$ .
    \item $J=\gen{j}$, with $j$ even.
    \end{enumerate}
    The computation of $f_{p,q}$ in these cases is 
    \begin{enumerate}
    \item       $f_{p,q} = \kappa(g_+)\kappa(g_-)
     \  =\ 
      \kappa (j_+) \kappa (j_-)$.
    \item    $f_{p,q} = \kappa(g_+)\kappa(g_-)\, \big((k_++1)(k_-+1)+1\big)$.
    \item  
$  f_{p,q} = \kappa (g_+)\kappa (g_-)\, (k_+ k_- +2)  $ . 
  \end{enumerate}
  \end{theorem}

\begin{proof}
 The cases (1) and (2) follow immediately from
 Theorem \ref{classifeasydecide} and 
 the discussion preceding the theorem.  Now suppose we are in case (3).
 For $x\in R$, let $C(x)$ be the $\sim$ class of
$\gamma_{2j}(x) $  in $\gamma_{2j}R$. 
 By Theorem \ref{classifeasydecide}, we must compute the number of
 distinct sets of the form $C(x) \cup C(x+j)$. 
For any integer $\alpha$ and odd prime $p$,
 $\delta_{p,2j_+}(\alpha +j_+) = \delta_{p,2j_+}(\alpha )$.
If $\delta_{2,2j_+}(\alpha )= i$, then 
\begin{alignat*}{2}
  \delta_{2,2j_+}(\alpha +j_+ ) &= i \quad
& &  \text{if } i< k_+ \\
  &= k_++1 \quad
&&  \text{if } i= k_+\\
  &= k_+ \ \ && 
  \text{if } i=k_+ +1 
   \ . 
\end{alignat*}
The analogous statements hold for  $j_-, k_-$.
Thus $C(x)\cup C(x+j)=C(x)$ 
 if  $x=(\alpha , \beta )$ with 
$  \delta_{2,2j_+}(\alpha ) < k_+ $ and
 $  \delta_{2,2j_-}(\beta ) < k_- $.
 The number of  sets $C(x)$ of this form 
 is 
 $\kappa(g_+)\kappa(g_-)\big((k_+-1)(k_--1)+1)$.
 (Note $k_+\geq 1 $ and $k_- \geq 1$ because $j$ is even.)

 If $x=(\alpha , \beta)$ with 
 $  \delta_{2,2j_+}(\alpha ) \in \{ k_+, k_++1\}$,
 then $C(x)$ and $C(x+j)$ are disjoint;  
 also, 
 $\alpha$ is an even  number, so 
$(\alpha, \beta)\in R$ iff $\beta $ is even. 
 Let $C^+_i$ be the collection of sets $C(x)$ for which
 $x=(\alpha , \beta )\in R$ with 
 $\delta_{2,2j_+}(\alpha )=i$. Then
   $|C^+_{k_+}| =| C^+_{k_++1}|   
   = \kappa (g_+) \kappa (g_-) (k_-+1) $ .
   The collections    $C^+_{k_+}, C^+_{k_++1}$ are disjoint,
   and the rule $C(x)\mapsto C(x+j)$ induces a
   bijection $C^+_{k_+}\leftrightarrow  C^+_{k_++1}$.
   So,  $C^+_{k_+}\cup  C^+_{k_++1}$ gives rise to
   $  \kappa (g_+) \kappa (g_-) (k_-+1) $ distinct sets
   of the form $C(x)\cup C(x+j)$.
   We can reverse the roles of $+$ and $-$ and say likewise
that 
$C^-_{k_-}\cup  C^-_{k_-+1}$ gives rise to
   $  \kappa (g_-) \kappa (g_+) (k_++1) $ distinct sets
   of the form $C(x)\cup C(x+j)$.

   Finally, to compute $f_{p,q}$ we add the three counts above,
   and correct for the doublecount:
\begin{align*}
  f_{p,q} &=
  \kappa (g_+)\kappa (g_-)
\Big( \big(  (k_+-1)(k_--1)+1 \big) 
+  \big(
k_++1\big) +   \big(k_-+1\big)  -2 \Big) \\
&=  \kappa (g_+)\kappa (g_-)
  (k_+k_-+2) \ . 
\end{align*} 
\end{proof}

\begin{example} \label{z2example}
  We will work out the classification of $G$-FE classes
  in $\mathcal M^{++}(p,q)$ for the example
  $p=(12,96)$, $q=(8,24)$. (These elements of $R$ correspond to the
  elements $54e-42g$, $16e-8g$ in $\ZG$.) Here
  $d_+=\gcd(12,8)=4$ and $d_-=\gcd(96,24)=24$.
We compute  
  $(12,96)= (r_pd_+, s_pd_-)=(3(4), 4(24))$, and
see $r_p\not\equiv s_p \mod 2$.
By  Proposition \ref{doubled},  it follows that
$\gen{p,q}= \gen{(2j_+,0), (0,2j_-)} $
with $(2j_+,2j_-) =(d_+,d_-)=(4,24)$.
By Theorem \ref{classifeasydecide},
matrices in $\mathcal M^{++}(p,q,x)$, $\mathcal M^{++}(p,q,x')$
define $G$-FE $G$-SFTs iff $\gamma_{(4,24)}(x)\sim \gamma_{(4,24)}(x)$
in $\Z_4\oplus \Z_{24}$. We will write this as $x\sim x'$.
Write $x=(\alpha, \beta), x'=(\alpha',\beta')$.
As $4=2^2$ and $24=2^33^1$, we have $x\sim x'$ iff the following hold:
$\delta_{2,4}(\alpha)=\delta_{2,4}(\alpha')$,
$\delta_{2,24}(\beta)=\delta_{2,24}(\beta')$, and
$\delta_{3,24}(\beta)=\delta_{3,24}(\beta')$. The allowed
combinations of possible values of these invariants is displayed
in Table 1, 
with a few examples. 

\begin{table} [ht]
  \caption{} 
  \centering
  \begin{tabular}{|c| c c c|}
    \hline 
    $(\alpha, \beta )$\ \  & $\delta_{2,4}(\alpha)$ \ \ &
    $\delta_{2,24}(\beta)$ \ \ & $\delta_{3,24}(\beta)$ \\
    \hline
    odd  & 0    & 0   & 0,1 \\
    even & 1,2  & 1,2,3  & 0,1  \\
    \hline
    (0,18)& 2 & 1 &  1 \\
    (2,6) & 1 & 1 & 1 \\
    (3,11)& 0 & 0 & 0 \\
    \hline
  \end{tabular}
  \label{tablesim}
\end{table}
We can check that e.g. $(0,6)\sim (0,18) \not\sim (0,12)$ and
$(1,3) \sim (3,15)\not\sim (3, 11)$. 
From Table 1 we see that 14 $\sim$ classes arise here for
elements of $R$, consistent with Theorem
\ref{classifforcounts}(2), which in our case gives
\begin{align*}
(j_+,j_-) & =(2,12) = (2^{k_+}g_+,2^{k_-} g_-) = (2^1(1), 2^3(3)) \\ 
f_{p,q} & = \kappa(g_+)\kappa(g_-)\, ( \, (k_++1)(k_-+1)+1\, ) \\
&=(1)(2)\, (\, (1+1)(2+1)+1)\, )= 14 \ . \ \ \qed 
\end{align*}
\end{example}

Next, we 
work to Proposition \ref{decidempqx}, which shows 
how to determine
        when a matrix $A$  lies in some $\mathcal M^{++}(p,q,x)$, and then
         compute $p,q,x$. 
         Two matrices are $\text{El}(R)$-equivalent when there exists $k$
         such that they have 1-stabilizations of size $k\times k$ which are
         $\text{El}(k,R)$-equivalent. If $B$ is $\text{El}(R)$-equivalent
         to a $1\times 1$ matrix $(p)$, then $p$ must be $\det(B)$. 

         \begin{lemma}\label{smithlemma}
Given an $n\times n$  matrix $B$  over $R$, 
  there is an algorithm which determines whether
  $B$ is  $\textnormal{El}(n,R)$-equivalent to
  some $1\times 1$ matrix $(p)$, and
  computes an explicit
  $\El(n,R)$-equivalence
from $B$ to  $(p\oplus I_{n-1})$,  
when this is the case. 
    \end{lemma} 
         \begin{proof}
A  matrix over $ R$ (even a $2 \times 2$ matrix)
need not be $\text{GL}(R)$
equivalent to a diagonal  matrix, or even a triangular matrix
(see e.g.\cite[Example 8.7]{BSullivan}). 
But, for the  special case we consider here,  a
modification of the Smith form argument applies to
produce the  
      explicit $\El(n,R)$-equivalence required for (1), or
a failure which shows there is no such equivalence. 
      For detail, see \cite[Lemma 8.2, Remark 8.3]{BSullivan}, and the
      algorithmic       proof of \cite[Lemma 8.2]{BSullivan}.
    \end{proof}

    \begin{proposition} \label{decidempqx}
      Suppose $B\in \mathcal M_{\mathcal P}(\mathbf n, R)$,
      with $\mathbf n = (n_1, n_2)$ and $\mathcal P = \mathcal P_2$.
      There is an algorithm to determine whether there exist
      $p,q,x$ such that $B$ is $\El_{\mathcal P}(\mathbf n, R)$-equivalent
      to the matrix $M(p,q,x)=
      \left(\begin{smallmatrix} p& x\\ 0&q \end{smallmatrix}\right)$, and
      in this case to compute $p,q,x$. 
    \end{proposition}

    \begin{proof}
             The required elementary equivalence will exist only if
             there are $\El(R)$-equivalences of the diagonal blocks of $B$
             to $(p)$ and $(q)$.
             By Lemma \ref{smithlemma}, we can decide
             whether this occurs, and in the case it does we can
             construct elementary equivalences, 
      $UB\{1,1\}V = (p)\oplus I_{n_1-1}$ and
      $U'B\{2,2\}V' = (q)\oplus I_{n_2-1}$. 
Given this data, we
       define an elementary equivalence of $B$,
       \begin{align*}
         \begin{pmatrix} U& 0\\ 0 & U'\end{pmatrix}
           \begin{pmatrix} B\{1,1\}& B\{1,2\}\\ 0 & B\{2,2\}\end{pmatrix}
         \begin{pmatrix} V& 0\\ 0 & V'\end{pmatrix}             
           = \begin{pmatrix}
             \left(\begin{smallmatrix} p&0\\0&I_{n_1-1}
             \end{smallmatrix}\right) & Y \\ 0&
             \left(\begin{smallmatrix} q&0\\0&I_{n_2-1}
             \end{smallmatrix}\right) \end{pmatrix} \  .
           \end{align*}
           Then for suitable blocks $Z,Z'$,
           \begin{align*}
             \begin{pmatrix} I& Z\\ 0 & I \end{pmatrix}
\begin{pmatrix}
             \left(\begin{smallmatrix} p&0\\0&I_{n_1-1}
             \end{smallmatrix}\right) & Y \\ 0&
             \left(\begin{smallmatrix} q&0\\0&I_{n_2-1}
             \end{smallmatrix}\right) \end{pmatrix}
\begin{pmatrix} I& Z'\\ 0 & I \end{pmatrix}
=
\begin{pmatrix}
             \left(\begin{smallmatrix} p&0\\0&I_{n_1-1}
             \end{smallmatrix}\right) &
\left(\begin{smallmatrix} x&0\\0&0
             \end{smallmatrix}\right) 
 \\ 0&
             \left(\begin{smallmatrix} q&0\\0&I_{n_2-1}
             \end{smallmatrix}\right) \end{pmatrix}
           \end{align*}
           where $x$ is the upper left entry of $Y$.
\end{proof}  

Note, for $B$ in Proposition \ref{decidempqx}: 
if  $B$ is not $\El(n,R)$-equivalent to
a 1-stabilization of the matrix $(p)\oplus I_{n-1}$,
then no  1-stabilization of $B$ is 
$\El(n,R)$-equivalent to
a 1-stabilization of the matrix $(p)\oplus I_{n-1}$.



\mbox{}\\ \small
\noindent
\textsc{Department of Mathematics, 
University of Maryland,
College Park, MD 20742-4015, USA}

\emph{E-mail address: }\texttt{mmb@math.umd.edu}\\[0.5cm]

\noindent
\textsc{Department of Science and Technology, University of the Faroe
  Islands, Vestara Bryggja 15, FO-100 T\'orshavn, The Faroe Islands}

\emph{E-mail address: }\texttt{toke.carlsen@gmail.com}\\[0.5cm]

\noindent
\textsc{Department of Mathematical Sciences,  University of Copenhagen, DK-2100 Copenhagen \O, Denmark}

\emph{E-mail address: }\texttt{eilers@math.ku.dk}

\end{document}